\documentclass{elsarticle}

\usepackage[linguistics]{forest}
 \usepackage{tikz-cd}
\usepackage{amssymb}
\usepackage{amsmath}
\usepackage{amsthm}
\usepackage{amsfonts}

 \newtheorem{theorem}{Theorem} 
  \newtheorem{lemma} [theorem] {Lemma}
\newtheorem{corollary} [theorem] {Corollary}
\newtheorem{definition} [theorem] {Definition} 
\newtheorem{proposition} [theorem] {Proposition} 
\newtheorem{open problem} [theorem] {Open Problem} 
 
\newtheorem{example} [theorem] {Example}

\newtheorem{claim} [theorem] {Claim}

 \DeclareMathOperator{\Left}{left}
\DeclareMathOperator{\Right}{right}

\begin{document}

\begin{frontmatter}

 \title{Undecidability in First-Order  Theories of Term Algebras Extended with a  Substitution Operator    }
 \author{Juvenal Murwanashyaka}
 \address{Department of Mathematics, University of Oslo,  Norway}
 \journal{a journal}

\begin{abstract}

 We introduce a first-order theory of finite full binary trees and then identify decidable and undecidable fragments of this theory.
We  show that  the analogue of Hilbert`s 10th Problem is undecidable  by constructing  a many-to-one reduction of 
Post`s Correspondence Problem. 
By a different method, 
we show that deciding truth of sentences  with  one existential quantifier and one bounded universal quantifier is undecidable. 

\end{abstract}

\end{frontmatter}

\section{Introduction}

Consider a finite  one-sorted   first-order language $ \mathcal{L} $ with equality, at least one constant symbol and no relation symbols. 
The canonical $ \mathcal{L} $-structure $ \mathcal{T}  ( \mathcal{L} ) $,  called  the term algebra of $\mathcal{L}$,   
has as a  universe  the set of all variable-free $ \mathcal{L} $-terms  and  is such that  each variable-free  term  is interpreted as itself. 
As a consequence  of decidability of  the monadic second-order theory of two successor functions, 
 $ \mathsf{S2S} $    (see Rabin  \cite{rabin1969}), 
if $ \mathcal{L} $ has only unary function symbols, then the monadic second-order theory of $ \mathcal{T} ( \mathcal{L} ) $ is decidable.
Indeed, if the constant symbols of $ \mathcal{L} $ are $ c_1, \ldots , c_k $ and the (unary) function symbols of  
$ \mathcal{L} $ are $f_1, \ldots , f_m $, 
then we can interpret  $ \mathcal{T} ( \mathcal{L} ) $   in  $ \mathsf{S2S} $ by letting 
 $ c_i = 01^{i } 0 $ and  $ f_j x = x 0 1^{ k+j } 0 $. 
 It is not difficult to see that the  closure of $ \lbrace 010, \ldots , 0 1^{ k } 0 \rbrace $ under the $ f_j$`s is  definable 
in   $ \mathsf{S2S} $ extend with  the prefix relation. 
But,  the prefix relation is known to be definable in $ \mathsf{S2S} $
(see for example  B\"orger  et al.  \cite{borger2001}  p.~317).
For general term algebras, quantifier elimination can be used to show  that the first-order theory of   
$  \mathcal{T}  ( \mathcal{L} ) $ is decidable
(see  for example  Hodges  \cite{Hodges1993} Section 2.7).

Since  first-order theories  of term algebras are decidable, they are restricted in  expressibility. 
In Kristiansen \& Murwanashyaka \cite{cie20}, we consider the first-order language  $\mathcal{L}_{ \mathsf{T} } =   \lbrace \perp , \langle \cdot , \cdot \rangle, \sqsubseteq \rbrace  $  where $ \perp$ is a constant symbol,  $ \langle \cdot , \cdot \rangle $ is a binary function symbol and $ \sqsubseteq $ is a binary relation symbol. 
The intended  $\mathcal{L}_{ \mathsf{T} }  $-structure $ \mathcal{T}  ( \mathcal{L}_{ \mathsf{T} } ) $  extends the term algebra given by
$  \lbrace \perp , \langle \cdot , \cdot \rangle \rbrace  $ by interpreting $ \sqsubseteq  $ as the subterm relation. 
We introduce two  theories    $ \mathsf{WT} $, $ \mathsf{T} $ with  simple purely universal axiomatizations
(see Figure \ref{AxiomsOfT} for the axioms of  $ \mathsf{WT} $ and  $ \mathsf{T} $). 
We show that  $ \mathsf{WT} $ is mutually interpretable with the  Tarski-Robinson-Mostowski theory of arithmetic $ \mathsf{R} $
and that $ \mathsf{T} $ interprets   Robinson Arithmetic  $ \mathsf{Q} $.
It is not difficult to see that  $ \mathsf{T} $ is an extension of  $ \mathsf{WT} $. 
Since $ \mathsf{WT} $  interprets $ \mathsf{R} $,   G\"odel`s first incompleteness theorem holds for  $ \mathsf{WT} $. 
That is,  any consistent recursively axiomatizable extension of  $ \mathsf{WT} $  is incomplete. 
In particular,  the first-order theory of  $ \mathcal{T}   ( \mathcal{L}_{ \mathsf{T} } ) $ is undecidable. 
There cannot exist an algorithm that takes as input a first-order  $\mathcal{L}_{ \mathsf{T} }  $-sentence $ \phi $ 
and decides whether $ \phi $ is true in  $ \mathcal{T}  ( \mathcal{L}_{ \mathsf{T} } ) $.

\begin{figure}
\[
\begin{array}{r l  c  c r l  }
&{\large \textsf{The Axioms of } \mathsf{WT} }
&
\\
\\
\mathsf{WT_{1} } 
&  s \neq t      \     \      \mbox{ if } s,t \mbox{ are distinct variable-free  }   \mathcal{L}_{ \mathsf{T} }  \mbox{-terms} 
\\
\mathsf{WT_{2} } 
& \forall x  \;    [   \   x \sqsubseteq  t     \leftrightarrow    \bigvee_{   s \in \mathsf{Sub} (t)  }   x = s         \     ] 
\end{array}
\] 
\[
\begin{array}{r l  c  c r l  }
&{\large \textsf{The Axioms of } \mathsf{T} }
&
\\
\\
\mathsf{ T_{1} } 
& \forall x y \;  [     \  \langle x, y \rangle \neq \perp    \         ]   
\\
\mathsf{ T_{2} } 
& \forall x y z w \;    [  \  \langle x, y \rangle  =   \langle z, w \rangle \rightarrow 
(  \  x=z \wedge y=w  \   )    \     ] 
\\
\mathsf{ T_{3} } 
&  \forall x  \; [  \ x \sqsubseteq \perp \leftrightarrow x= \perp  \  ]   
\\
\mathsf{ T_{4} } 
&   \forall x y z  \; [      \       x \sqsubseteq   \langle y , z  \rangle \leftrightarrow 
(  \  x =  \langle y , z  \rangle  \vee x \sqsubseteq y \vee x \sqsubseteq z  \  )       \       ] 
\end{array}
\] 

\caption{
Non-logical axioms of the first-order theories  $ \mathsf{WT} $,  $ \mathsf{T} $.
The axioms of $ \mathsf{WT} $ are given by axiom schemes where $s, t $ are variable-free   $\mathcal{L}_{ \mathsf{T} }  $-terms. 
Given a variable-free term $t$,   $ \mathsf{Sub} (t)  $ denotes the set of all subterms of $t$.
}
\label{AxiomsOfT}
\end{figure}

Since the first-order theory of  $ \mathcal{T}   ( \mathcal{L}_{ \mathsf{T} } ) $    is undecidable, 
a natural question is whether  it is possible to give a good characterization of the boundary between what we can and cannot effectively decide over  $ \mathcal{T}     ( \mathcal{L}_{ \mathsf{T} } ) $. 
In \cite{VenkataramanK1987}, 
Venkataraman   shows  that deciding truth of existential sentences is NP-complete
 while deciding truth of $\Sigma$-sentences    is undecidable.  
 Sentences are formulas without free variables and 
 $\Sigma$-sentences are sentences on negation normal form  where  universal quantifiers occur  bounded, i.e., they  are of the form 
$ \forall x \sqsubseteq t $. 
Existential sentences  are $ \Sigma$-sentences with no occurrence of universal quantifiers. 
 By inspecting the proof of Theorem 4.1 of \cite{VenkataramanK1987}, 
 we find that  Venkataraman actually proves  that deciding truth of  $\Sigma$-sentences with 
 3 existential quantifiers and 7 bounded universal quantifiers is undecidable. 
In Section \ref{UndecidableFragmentsPartISubtreeRelation}, 
we show  that deciding truth of   $ \mathcal{L}_{ \mathsf{T} }  $-sentences  of the form 
$ \exists x   \;  \forall y  \sqsubseteq x   \;   \forall z  \sqsubseteq x   \;  \phi $, where $ \phi $ is quantifier-free,  is undecidable. 
In  Section \ref{UndecidableFragmentsPartII}, 
we show that deciding truth of $ \Sigma$-sentences  in the meager language 
$   \mathcal{L}_{ \mathsf{T}^- }   = \lbrace \sqsubseteq \rbrace $
is undecidable.

We obtain a more expressive structure by replacing the subterm relation with  a  substitution operator  
$ \cdot  [  \cdot \mapsto \cdot ]  $ on 
variable-free  $ \mathcal{L}_{ \mathsf{T} }  $-terms: 
$ t [ r \mapsto s ] $ is the term we obtain by replacing each occurrence of $r $ in $t$ with $s $. 
Let  $    \mathcal{L}_{ \mathsf{BT} }  $ and   $ \mathcal{T} ( \mathcal{L}_{ \mathsf{BT} } ) $ denote the corresponding language and structure.
It is not difficult to see that the subterm relation is definable in  $ \mathcal{T} ( \mathcal{L}_{ \mathsf{BT} } ) $ 
by  a quantifier-free $   \mathcal{L}_{ \mathsf{BT} } $-formula. 
Indeed,  $ \mathcal{T} ( \mathcal{L}_{ \mathsf{BT} } ) $ is more expressive than  $ \mathcal{T} ( \mathcal{L}_{ \mathsf{T} } ) $
if we take quantifier complexity into account. 
In  Section \ref{UndecidableFragmentsPartISubstitutionOperator}, 
we  show  that    it  is undecidable whether a   sentence of the form 
$ \exists x   \;  \forall y  \sqsubseteq x   \;  \phi $, where $ \phi $ is quantifier-free,  is true in  
$ \mathcal{T} ( \mathcal{L}_{ \mathsf{BT} } ) $. 
Given this result, 
to  characterize the boundary between what we can and cannot effectively decide over  $ \mathcal{T} ( \mathcal{L}_{ \mathsf{BT} } ) $, 
we need to investigate the expressive power of  the existential fragment of   $ \mathcal{T} ( \mathcal{L}_{ \mathsf{BT} } ) $. 
In Section \ref{UndecidableFragmentsPartIII}  and Section \ref{UndecidableFragmentsPartIV}, 
we show that  the existential theory of  $ \mathcal{T} ( \mathcal{L}_{ \mathsf{BT} } ) $, 
denoted $  \mathsf{Th}^{ \exists }  (  \mathcal{T}  ( \mathcal{L}_{ \mathsf{BT}   }  )  )       $, 
  is undecidable. 
In Section \ref{AnalogueOfHilbers10Problem}, 
we show that this implies that the analogue of Hilbert`s 10th Problem is unsolvable.   
That is, we show that there cannot exist an algorithm that takes as input an existential 
$  \mathcal{L}_{ \mathsf{BT} }  $-sentence $ \psi $ of the form $ \exists \vec{x} \;  [   \  s =  t   \   ]  $
  and decides whether $ \psi $ is true  in   $ \mathcal{T} ( \mathcal{L}_{ \mathsf{BT} } ) $.

We give two  proofs of undecidability  of   $  \mathsf{Th}^{ \exists }  (  \mathcal{T}  ( \mathcal{L}_{ \mathsf{BT}   }  )  )       $. 
In Section   \ref{UndecidableFragmentsPartIII}, 
we prove undecidability of  $  \mathsf{Th}^{ \exists }  (  \mathcal{T}  ( \mathcal{L}_{ \mathsf{BT}   }  )  )       $ 
by constructing an existential interpretation of    $ (  \mathbb{N} , 0, 1, + , \times ) $ 
in  $ \mathcal{T} ( \mathcal{L}_{ \mathsf{BT} } ) $. 
An existential interpretation is a relative interpretation that maps existential formulas to existential formulas. 
In Section   \ref{UndecidableFragmentsPartIV}, 
we give a direct proof of  undecidability of   $  \mathsf{Th}^{ \exists }  (  \mathcal{T}  ( \mathcal{L}_{ \mathsf{BT}   }  )  )       $ 
by constructing a many-to-one reduction of Post`s Correspondence Problem.

\section{Preliminaries}
\label{Preliminaries}

We consider the first-order  languages  
\[
 \mathcal{L}_{ \mathsf{T}^- } = \lbrace  \sqsubseteq  \rbrace 
,   \    \   
 \mathcal{L}_{ \mathsf{T} } = \lbrace \perp ,  \langle \cdot , \cdot \rangle , \sqsubseteq  \rbrace 
,   \    \   
   \mathcal{L}_{ \mathsf{BT} } = \lbrace \perp ,  \langle \cdot , \cdot \rangle ,  \cdot [  \cdot \mapsto \cdot ]    \rbrace   
\]
where $\perp$ is a constant symbol, $\langle \cdot , \cdot \rangle $ is a binary function symbol,  $\sqsubseteq  $ is a binary relation symbol and $ \cdot [  \cdot \mapsto \cdot ]   $ is a ternary function symbol. 
The standard structures  for these languages  are term models: 
The universe  $ \mathbf{H} $   is the set of all variable-free $  \mathcal{L}_{ \mathsf{T} } $-terms. 
The function symbol  $ \langle \cdot , \cdot \rangle $ is interpreted as the function that maps the pair $(s, t ) $ to the term  
$ \langle s, t \rangle $. 
The relation symbol  $\sqsubseteq$ is interpreted as the subterm relation: $s $ is a subterm of $t$ iff $ s = t $ or 
 $ t = \langle t_1 , t_2 \rangle$ and $s$ is a subterm of $t_1 $ or $t_2$. 
The function symbol $ \cdot [  \cdot \mapsto \cdot ]   $ is interpreted as a  term substitution operator: 
$ t [ r \mapsto s ] $ is the term we obtain by replacing each occurrence of $r $ in $t$ with $s $. 
We define  $ t [ r \mapsto s ] $ by recursion as follows 
\[
 t [ r \mapsto s ]   
 =  \begin{cases}
 s                                                  &  \mbox{ if }  t = r 
 \\
  \perp                                                  &  \mbox{ if }  t \neq  r     \mbox{ and  }    t = \perp 
  \\
\big\langle    t_1   [ r \mapsto s ]      \,  ,  \,    t_2  [ r \mapsto s ]       \big\rangle          
                                     &  \mbox{ if }  t \neq  r    \mbox{ and  }       t = \langle t_1,  t_2  \rangle    \             .
 \end{cases}
\]
We will occasionally refer to   variable-free  $  \mathcal{L}_{ \mathsf{T} } $-terms as finite (full) binary trees  
and the relation $ \sqsubseteq $ as the subtree relation.
We let 
\[
\mathcal{T} (  \mathcal{L}_{ \mathsf{T}^- }   )  =    \big(  \mathbf{H} ,   \sqsubseteq    \big) 
,   \       
\mathcal{T} (  \mathcal{L}_{ \mathsf{T} }   )   =  \big(  \mathbf{H} , \perp , \langle \cdot , \cdot \rangle  , \sqsubseteq    \big)  
,   \     
 \mathcal{T} (  \mathcal{L}_{ \mathsf{BT} }   )  =   
 \big(  \mathbf{H} , \perp , \langle \cdot , \cdot \rangle  ,  \cdot [  \cdot \mapsto \cdot ]      \big) 
 \            .
\]

We introduce the bounded quantifiers  $ \exists x \sqsubseteq  t \; \phi$,  $\forall x \sqsubseteq  t \; \phi$ as shorthand notation for 
$ \exists x \;   [   \     x \sqsubseteq t    \, \wedge \,    \phi \    ] $ and  
$ \forall x \;   [   \     x \sqsubseteq t    \, \rightarrow \,    \phi \    ] $, respectively. 
In the case of  $ \mathcal{T} (  \mathcal{L}_{ \mathsf{BT} }   ) $, 
we let  $  x \sqsubseteq t  $ and  $  x  \not\sqsubseteq t  $  be shorthand for 
$ t [   \,   x     \,   \mapsto       \,    \langle  x , x   \rangle   \,   ] \neq t $ 
and $ t [  \,   x     \,   \mapsto       \,    \langle  x , x   \rangle   \,  ] =  t $, respectively.
We define  $\Sigma$-formulas inductively: 
 $\phi$ and $\neg \phi$ are $\Sigma$-formulas if $\phi$ is an atomic formula, 
$(\phi \vee \psi)$,  $(\phi \wedge \psi)$,  
$\exists x  \sqsubseteq t \;   \phi$,   $\forall x  \sqsubseteq t \;  \phi$,  $ \exists x \;   \phi$ 
 are  $\Sigma$-formulas if $\phi$ and $\psi$ are $\Sigma$-formulas 
and  $x$ is a variable that does not occur in  the term $t$.
An  existential formula is  a $\Sigma$-formula that  does not contain  bounded  quantifiers.
A sentence is a formula without free variables.

The main   focus  of this paper is to determine decidable and undecidable fragments of 
$ \mathcal{T} (  \mathcal{L}_{ \mathsf{T}^- }   ) $, 
$ \mathcal{T} (  \mathcal{L}_{ \mathsf{T} }   ) $    and 
$ \mathcal{T} (  \mathcal{L}_{ \mathsf{BT} }   ) $.
In Venkataraman   \cite{VenkataramanK1987}, 
it is shown  that the set of  existential $  \mathcal{L}_{ \mathsf{T} } $-sentences true in  
$ \mathcal{T} (  \mathcal{L}_{ \mathsf{T} }   ) $ 
is computable  and that the set of true $ \Sigma$-sentences is not computable. 
As a step towards  determining  the boundary between what we can and cannot effectively decide, 
we classify $ \Sigma$-formulas  according to the number and the type of quantifiers they contain: 
 A $ \Sigma_{n, m, k } $-formula  is a  $\Sigma$-formula that contains $n$ unbounded existential quantifiers, $m$ bounded existential quantifiers and $k$ bounded universal quantifiers. 
The fragment $\Sigma_{n, m, k }^{ \mathfrak{A} }  $ is the set of  all $\Sigma_{n, m, k } $-sentences that are  true in   $ \mathfrak{A}$.
The existential theory (existential fragment) of $ \mathfrak{A} $ is the set of all existential sentences  that are true in $ \mathfrak{A} $. 
We let $ \mathsf{Th}^{ \exists } ( \mathfrak{A}  ) $ denote the  existential theory  of $ \mathfrak{A} $.
The $ \Sigma$-theory of $ \mathfrak{A} $, denoted  $  \mathsf{Th}^{ \Sigma  } ( \mathfrak{A}  )  $, 
is the set of all $ \Sigma$-sentences in the language of  $ \mathfrak{A}  $ that are   true in  $ \mathfrak{A}  $.
That is,   
$  \mathsf{Th}^{ \Sigma  } ( \mathfrak{A}  )    =   \bigcup_{ n , m, k \geq 0 }   \Sigma_{n, m, k }^{  \mathfrak{A}   }    $. 
We let   $ \mathsf{Th}^{ \mathsf{H10} } ( \mathfrak{A}  ) $ denote the set of all  sentences of the form 
$ \exists \vec{ x } \;  [   \  s =  t  \   ]  $ that are  true in   $ \mathfrak{A}$.

We will be interested in comparing   first-order  structures   using a notion that is finer than many-to-one reducibility. 
Ever since  Yuri Matiyasevich  proved undecidability of $ \mathsf{Th}^{ \exists }  ( \mathbb{N}  , 0, 1, + , \times ) $ 
(see for example Davis \cite{Davis1973}),    
a standard technique for showing that a structure has undecidable existential theory  has been to show that   addition and multiplication are existentially definable on an  existentially definable domain. 
This is the notion we intend to use and  choose to refer to it as
\emph{$\exists $-interpretability}
 (\emph{existential interpretability}). 
Indeed, it is a special case of the more general notion of relative interpretability introduced by Alfred Tarski  \cite{tarski1953}. 
The structures we consider are mutually  interpretable with respect to this more general notion.
We  restrict ourselves to one-dimensional parameter-free relative interpretations and  treat equality as a logical operator.
For a more general notion of existential interpretability, see sections 5.3 and 5.4a of  Hodges \cite{Hodges1993}.

Let $ \mathcal{L}_0 $ and $ \mathcal{L}_1 $ be   finite (one-sorted)  first-order languages.
A $ \mathcal{L}_0 $-structure $ \mathfrak{A}  $  is \emph{$\exists $-interpretable} in  a $ \mathcal{L}_1 $-structure $ \mathfrak{B}  $  if 
\begin{itemize}

\item[(1)] we can find  an  existential  $ \mathcal{L}_1 $-formula $ \delta (x) $ that defines a non-empty subset 
$ A^{ \prime } $ of  the universe $B$  of $ \mathfrak{B}  $

\item[(2)] for each constant symbol $c$ of  $  \mathcal{L}_0 $,
 we can find an existential   $ \mathcal{L}_1 $-formula $ \phi_{c} (x)  $ that 
defines a unique element $ c^{ \prime }  \in A^{ \prime } $

\item[(3)] for each $n$-ary function  symbol $f$ of  $  \mathcal{L}_0 $,
 we can find an existential    $ \mathcal{L}_1 $-formula 
$ \phi_{f} (x_1, \ldots , x_n , y )  $ that 
defines a function from $(A^{ \prime })^n  $ to   $  A^{ \prime } $

\item[(4)] for each $n$-ary relation   symbol $R$ of  $  \mathcal{L}_0 $,
 we can find  existential     $ \mathcal{L}_1 $-formulas 
$ \phi_{R} (x_1, \ldots , x_n  )  $,  $ \phi_{R^c} (x_1, \ldots , x_n  )  $ that 
define  disjoint  sets  $ R^{ \prime }  \subseteq ( A^{ \prime }  )^n $,  $ \;    (R^c )^{ \prime }  \subseteq ( A^{ \prime }  )^n $ 
such that $ R^{ \prime }   \cup   (R^c )^{ \prime }  =    ( A^{ \prime }  )^n   $

\item[(5)]  (1)-(4) define a $ \mathcal{L}_0 $-structure $ \mathfrak{A}^{ \prime }  $ that is isomorphic to   $ \mathfrak{A}  $. 

\end{itemize}

The    formulas that occur in  (1)-(4) are parameter-free. That is, we display all free variables. 
If  $ \mathfrak{A}  $  is  $ \exists $-interpretable in $ \mathfrak{B}  $  and 
 $ \mathfrak{B}  $  is $ \exists $-interpretable in $ \mathfrak{A}  $, 
 we say that  $ \mathfrak{A}  $  and $ \mathfrak{B}  $  are  \emph{mutually $\exists $-interpretable}.

The following proposition summarizes important properties of this notion. 
They are straightforward and the proof is therefore omitted.

\begin{proposition}
Let  $ \mathfrak{A}   $, $ \mathfrak{B}  $ and $ \mathfrak{C}  $  be   first-order structures in finite languages. 
\begin{itemize}

\item If  $ \mathfrak{A}  $  is  $ \exists$-interpretable in $ \mathfrak{B}  $  and 
 $ \mathfrak{B}  $  is $ \exists$-interpretable in $ \mathfrak{C}  $, 
 then  $ \mathfrak{A}  $  is $ \exists$-interpretable in $ \mathfrak{C}  $.

 \item If  $ \mathfrak{A}  $  is $ \exists$-interpretable in $ \mathfrak{B}  $  and 
 $ \mathsf{Th}^{ \exists } ( \mathfrak{B}  ) $ is decidable,
  then    $ \mathsf{Th}^{ \exists } ( \mathfrak{A}  ) $  is decidable.

\end{itemize}

\end{proposition}

We have the following three natural problems that we have not been able to settle.

\begin{open problem}    \label{MainOpenProblemExistentialInterpretabilityDegree}

\begin{enumerate}

\item[\textup{(1) } ] Let $\mathfrak{A} $ be a  computable  first-order structure. 
Assume   $ \mathsf{Th}^{ \exists } ( \mathfrak{A}  ) $  is undecidable. 
Is $ (  \mathbb{N} , 0, 1, + , \times  ) $   $ \exists $-interpretable  in  $\mathfrak{A} $?

\item[\textup{(2) }  ] 
 Let $ \mathcal{U} $ denote the class of all  computable  first-order  structures   with  undecidable existential theory. 
Does $ \mathcal{U} $ have a minimal element with respect to  $ \exists $-interpretability?

\item[\textup{(3) } ] 
Let $ \mathcal{D} $ denote the  class of computable  first-order structures  with decidable existential theory. 
Does $ \mathcal{D} $ have a maximal element  with respect to   $ \exists $-interpretability?

\end{enumerate}

\end{open problem}

\section{Undecidable Fragments I}   \label{UndecidableFragmentsPartI}

In this section, we show that the fragment $ \Sigma_{1, 0, 2 }^{  \mathcal{T} (  \mathcal{L}_{ \mathsf{T} }   ) }   $ is undecidable. 
 That is, we show that  there cannot exist an algorithm that takes as input a  $ \mathcal{L}_{ \mathsf{T} } $-sentence $ \phi $ of the form 
 $ \exists x \;   \forall y \sqsubseteq x \,   \forall z \sqsubseteq x \;  \phi_0  $,  where $ \phi_0 $ is  quantifier-free, 
  and decides whether    $ \phi $ is true in  $ \mathcal{T} (  \mathcal{L}_{ \mathsf{T} }   )  $. 
 We give two different proofs. 
 The first proof is very short but the method cannot be used to analyze  fragments of lower complexity since  it necessitates the use of  two bounded universal quantifiers. 
 The second proof is longer but the method allows us  to also show that the fragment  
 $ \Sigma_{1, 0, 1 }^{  \mathcal{T} (  \mathcal{L}_{ \mathsf{BT} }   ) }   $    is  undecidable.
 There is also the possibility that the method can be improved to show that the  fragment 
 $ \Sigma_{1, 0, 1 }^{  \mathcal{T} (  \mathcal{L}_{ \mathsf{T} }   ) }   $ is undecidable.

\begin{open problem}

Is the  fragment $ \Sigma_{1, 0, 1 }^{  \mathcal{T} (  \mathcal{L}_{ \mathsf{T} }   ) }   $      undecidable?

\end{open problem}

\subsection{Finite Sets of  Finite  Binary Trees}

The proofs we give depend  on our ability to code finite sequences of finite  binary trees. 
For this purpose, it will be convenient for us   to think of finite binary trees as finite sets.
We restrict ourselves to looking at finite sets of certain finite binary trees in order to have  a quantifier-free definition of the membership relation. 
It will however be the case that each finite binary tree determines a finite set.

\begin{definition} \label{SetMembershipDefinition}

Let  $ \alpha  \in \mathbf{H} $. 
Then, $ x \in_{ \alpha }  y $ is shorthand for 
$ \langle x, \alpha \rangle \sqsubseteq y \;  \wedge  \;   \alpha \not\sqsubseteq   x $.

\end{definition}

It is important to notice that our definition of the membership relation is quantifier-free.
In particular,  subformulas of the form $ x \in_{ \alpha }  y $ will not hide any  quantifier complexity.  
This is also true of  the structure 
 $\mathcal{T} (  \mathcal{L}_{ \mathsf{BT} }   ) $  
since we chose to define the subtree relation as follows: 
$  x \sqsubseteq y  \equiv  \;   y [ x \mapsto   \langle x ,  x  \rangle  \,   ]  \neq y   \ $.

We give a few examples.

\begin{example}

Let $ \alpha  $ be $ \perp $. 
Then, every finite  binary tree  encodes the empty set since all  finite binary trees have  $ \perp $ as a subtree.

\end{example}

\begin{example}

Let $ \alpha  $ be $  \langle   \perp ,  \perp  \rangle  $. 
The only finite binary tree that does not have $ \langle   \perp ,  \perp  \rangle   $ as a subtree is $ \perp $. 
Hence, every finite binary tree encodes the empty set or the singleton set $ \lbrace \perp \rbrace $. 

\end{example}

\begin{example}

Let $ \alpha   \equiv  \;    \langle  \perp   \,  ,   \,     \langle  \perp ,  \perp  \rangle  \,   \rangle $. 
We now have  more sets since there are infinitely many finite binary trees that do not have $  \alpha$ as a subtree. 
For example, 
if  $ w_0, w_1, w_2 $ are finite binary trees that do not have $  \alpha $ as a subtree, 
then each one of  the   trees    in Figure \ref{ExamplesSetMembership}   encodes  the set $ \lbrace w_0, w_1, w_2   \rbrace  $.

\begin{figure}

\begin{center}

\begin{forest}
[  
[  
 [    [$w_0$] [$ \alpha $]    ]  
 [    [$w_1$] [$ \alpha $]    ]
]
 [    [$w_2$] [$ \alpha $]    ]
]
\end{forest}      
\qquad 
\begin{forest}
[ 
 [    [$w_1$] [$ \alpha $]    ]
[
 [    [$w_0$] [$ \alpha $]    ]
 [    [$w_2$] [$ \alpha $]    ]
 ]
]
\end{forest}

\end{center}

\caption{Representations of the finite set $ \lbrace w_0, w_1, w_2   \rbrace  $. }
\label{ExamplesSetMembership}
\end{figure}
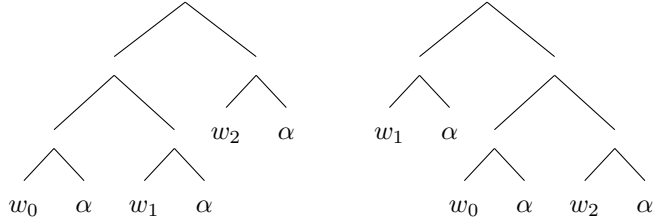

\end{example}

 \subsection{Notation}  \label{SecondNotationSystem}

We will encounter many cases where we need to associate a finite sequence of binary trees on the meta level with a term in the formal language. 
 To improve readability, we introduce the following notation
 \begin{itemize}

\item[-]    $  \langle  x  \rangle   \equiv \   x $ and 
  $ \langle   x_1, \ldots ,  x_n ,  x_{n+1 }   \rangle  \equiv  \   
   \langle   \,   \langle  x_1, \ldots , x_n   \rangle  \,  ,  \,  x_{n+1} \rangle $  for $n >0  $

\item[-]  $  \langle   x_1, \ldots , x_{n }   \rangle   ^{ \frown }   \langle   y_1, \ldots , y_{m }   \rangle     \equiv  \  
 \langle     x_1, \ldots , x_{n } , y_1, \ldots , y_m  \rangle   $

  \item[-] $ t^1 \equiv   \   \langle  t  \rangle  $ and $  t^{ n+1 }  \equiv  \      t^n  \, ^{ \frown } \,  \langle  t  \rangle     $ for $ n> 0 $

\item[-]  $ \langle   x_1, \ldots , x_{n }  \rangle^1 \equiv  \  \langle   x_1, \ldots , x_{n }  \rangle $

\item[-]  $  \langle   x_1, \ldots , x_{n }  \rangle^{k+1}  \equiv  \
      \langle   x_1, \ldots , x_{n }  \rangle^{ k }   \,  ^{ \frown}    \,      \langle   x_1, \ldots , x_{n }  \rangle    \;       $.

\end{itemize}

For example, 
occurrences of   $ x ^{ \frown } y $,  $ x^2 \,  ^{ \frown }  \,  y $  and $ x  \,  ^{ \frown }  \,  y^2 $
in formal formulas should be interpreted as  shorthand notation for 
$ \langle  x, y  \rangle  $,  $ \langle  x, x, y \rangle $ and $ \langle  x, y, y  \rangle  $, respectively.

 \subsection{Post`s  Correspondence Problem}     \label{UndecidableFragmentsPartISubtreeRelation}

 We show that  $ \Sigma_{1, 0, 2 }^{  \mathcal{T} (  \mathcal{L}_{ \mathsf{T} }   ) }   $  is undecidable by giving 
 a many-to-one reduction of  Post`s correspondence problem  (see Post  \cite{Post}).

 Let $ \lbrace  0, 1   \rbrace^{+} $ denote the set of all nonempty binary strings.

\begin{definition}

{\em The  Post  Correspondence Problem  (PCP)}  is given by
\begin{itemize}

\item Instance: a list of pairs $\langle a_{1}, b_{1} \rangle,\ldots , \langle a_{n}, b_{n}  \rangle$ where 
$a_i ,b_i  \in   \lbrace  0, 1   \rbrace^{+}$

\item Solution:  a finite nonempty sequence $i_{1},...,i_{m}$ of indexes such that 
\[
a_{i_{1}} a_{i_{2}}\ldots  a_{i_{m}} = b_{i_{1}}  b_{i_{2}}   \ldots  b_{i_{m}} 
\  .
\]

\end{itemize}

\end{definition}

Venkataraman  \cite{VenkataramanK1987}   proved that the fragment 
$ \bigcup_{ n, m , k  \in \mathbb{N} }  \Sigma_{n, m , k }^{  \mathcal{T} (  \mathcal{L}_{ \mathsf{T} }   ) }   $
is undecidable by giving a many-to-one  reduction of PCP. 
An inspection of the proof shows that what is proved  is actually that  the fragment 
$ \Sigma_{1, 2, 7 }^{  \mathcal{T} (  \mathcal{L}_{ \mathsf{T} }   )    } $   is undecidable. 
Venkataraman introduces a constant symbol $c$, 
a  unary operator $ d_i ( \cdot ) $ for each letter $d_i $ of the alphabet and a ternary function $f$. 
 The  string $ d_{i_1 }   \ldots d_{i_m } $ is represented as the term $ d_{i_1 } (  \ldots (d_{i_m } (c ) )  \ldots ) $. 
Venkataraman  then tries to capture that an instance of PCP has a solution if and only if there exists a witnessing term $t$  of the form 
 \[
 t = f( r_1, s_1, f (r_2, s_2, f ( \ldots  f( r_{m+1}, s_{m+1} , c )  \ldots )  )   )   
 \]
where $ r_{j } =  a_{i_{j}}\ldots  a_{i_{m}}  $, $ s_{j } =  b_{i_{j}}\ldots  b_{i_{m}}  $ and $r_{m+1} = s_{m+1 } $  is  the empty string.
It becomes  immediately clear  that at least three bounded universal quantifiers are needed since we need to talk about arbitrary subterms  of $t$ of the from $ f( x, y , z ) $. 
Two bounded existential quantifier are necessary to say that $t$ has the form $ t = f( u, u, v ) $.
We get a stronger result by encoding PCP differently.

\begin{theorem}   \label{FirstUndecidabilityProof}

The fragment   $ \Sigma_{1, 0, 2 }^{  \mathcal{T} (  \mathcal{L}_{ \mathsf{T} }   ) }   $    is  undecidable.

\end{theorem}

\begin{proof}

We start by translating concatenation of finite strings. 
Let   $  0 \equiv   \;  \perp^3 $ and $  1  \equiv   \;  \perp^4 $. 
Consider a nonempty  binary string  $ w = w_1  \ldots w_k $ where  $ w_i \in   \lbrace  0, 1   \rbrace   $ for each $ 1 \leq i \leq k $.
We represent $w $ in the formal language as $  \langle   w_1, \ldots , w_k   \rangle   $. 
We represent $ xw $ as  $  x^{ \frown}  w  $.
Recall that $  x^{ \frown}  \langle   w_1, \ldots , w_k  \rangle    \equiv   \;   \langle   x, w_1, \ldots , w_k  \rangle  $. 
For example,  if  $ w = w_1 w_2 w_3 $, 
then $0$, $1$,  $w$ and  $  xw  $ are drawn, respectively,   in Figure \ref{TranslationOfStringConcatenation}.

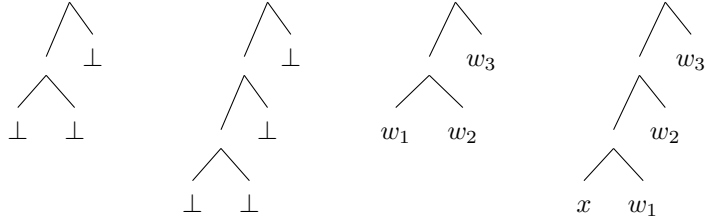
\begin{figure}

\begin{center}

\begin{forest}
[
[
[$\perp$ ]
[ $\perp$ ]
]
[$\perp$]
]
\end{forest}   
\qquad  
\begin{forest}
[
[
[
[$\perp$ ]
[ $\perp$ ]
]
[ $\perp$ ]
]
[$\perp$]
]
\end{forest}    
\qquad  
\begin{forest}
[
[
[$w_1$ ]
[ $w_2$ ]
]
[$w_3$]
]
\end{forest}  
\qquad  
\begin{forest}
[ 
[
[
[$x$]
[$w_1$ ]
]
[ $w_2$ ]
]
[$w_3$]
]
\end{forest}

\end{center}

\caption{Translation of the binary strings  $0$, $1$,  $w = w_1 w_2 w_3$ and  $  x w  $
 in the proof of Theorem   \protect\ref{FirstUndecidabilityProof}.}
\label{TranslationOfStringConcatenation}
\end{figure}

Before proceeding, we need to choose  a suitable parameter $ \alpha $  for our  definition of set membership  
(see Definition \ref{SetMembershipDefinition}). 
Since we want to talk about finite sets of binary strings, 
 a representation of a  binary string cannot have $ \alpha $ as a subtree.
 We let 
\[
x \in y \equiv  \    x \in_{ \alpha }  y 
\   \   
\mbox{ where } 
 \alpha  \equiv   \;       \langle   \perp   \,  ,  \,     \perp^2      \rangle
\      .
\]
We remind the reader that if $ x \in y $, then $ x \sqsubseteq y $. 
This is important to see that the  bounded universal quantifiers in the formula $ \phi $ below cover the search space we are interested in.

Given an instance  
$\langle  a_{1}, b_{1}  \rangle ,    \ldots ,  \langle  a_{n}, b_{n}   \rangle  $ 
of PCP, we need to compute a $ \Sigma_{1, 0, 2 }$-sentence $ \phi $ that is true in 
$  \mathcal{T} (  \mathcal{L}_{ \mathsf{T} }   ) $ 
if and only if there exists a finite nonempty sequence $i_{1},  \ldots  ,i_{m}$ of indeces such that 
$
a_{i_{1}} a_{i_{2}}\ldots  a_{i_{m}} = b_{i_{1}}  b_{i_{2}}   \ldots  b_{i_{m}} 
$.
The existence of   $i_{1},  \ldots  ,i_{m}$  is equivalent to the existence of a finite set   $T$ that satisfies the following 
\begin{enumerate}

\item[-] there exists $ i \in \lbrace  1 , \ldots ,  n  \rbrace  $ such that     $ \langle a_i , b_i \rangle    \in  T $

\item[-] if $ \langle L, R   \rangle   \in T $ and $ L \neq R $, then 
$ \langle  L ^{ \frown } a_i    ,  \,   R ^{ \frown } b_i        \rangle  \in  T  $
for some  $ i \in \lbrace  1 , \ldots ,  n  \rbrace  $.

\end{enumerate}

Given a solution $i_{1},   \ldots  ,i_{m}$,   the witnessing set  $T$ can be any finite  binary tree that encodes the set 
\[
 \lbrace   \langle  a_{i_1} ^{ \frown } \ldots ^{ \frown }  a_{ i_j }   , \;     b_{i_1} ^{ \frown } \ldots ^{ \frown }  b_{ i_j }      \rangle 
:    \   1 \leq j \leq m \rbrace  
\     .
\]

We let $ \phi $ be the following sentence  
\begin{multline*}
\phi \equiv   \  
\exists T  \;   \forall L, R  \sqsubseteq T   \;    
 \Big[         \    
\bigvee_{ i = 1 }^{ n } 
 \langle   a_i , \,  b_i  \rangle   \in  T     
\     \wedge   \   
\\
\Big(   \       
\big(        \               \langle  L , \,  R   \rangle  \in  T    \;  \wedge   \;   L \neq R       \       \big)  
\rightarrow
\\
\bigvee_{ i = 1 }^{ n }     
\langle  L ^{ \frown } a_i    ,  \,     R ^{ \frown } b_i   \rangle   \in   T   
 \       \Big) 
\        \Big] 
\          . 
\qedhere
\end{multline*}

\end{proof}

 \subsection{The Modulo Problem}    \label{UndecidableFragmentsPartISubstitutionOperator}

In this section, we show that the fragment  $ \Sigma_{ 1 , 0, 1}^{ \mathcal{T} (  \mathcal{L}_{ \mathsf{BT} }   ) } $ is undecidable. 
We cannot prove this result by giving a many-to-one reduction of PCP since expressing that an instance of PCP has a solution necessitates the use of two bounded universal quantifiers. 
Instead, we give  a many-to-one reduction of the   Modulo Problem,  which is an arithmetical problem. 
The Modulo Problem is introduced in  Kristiansen \& Murwanashyaka \cite{KristiansenMurwanashyakaAML} where it is used to characterize undecidable fragments of finitely generated free semigrooups extended with natural binary relations on strings such as the prefix relation and the substring relation.  
Undecidability of the Modulo Problem follows from  undecidability of a generalized version of the Collatz conjecture
studied first  by    Conway  \cite{Conway}   and  then by   Kurtz \& Simon \cite{kurtz}. 
The main difference between  the   Modulo Problem  and PCP  is that  the  Modulo problem is about sequences of 1-tuples  while PCP is about sequences of 2-tuples.

\begin{definition} \label{ModuloProblemDefinition}

Let  $f^N$ denote  the $N^{\mbox{{\scriptsize th}}}$ iteration of  $f$: 
  \[
  f^0 (x)=x  \     \       \mbox{  and   }    \       \       f^{N+1}(x)= f(f^N(x))
  \           .
  \] 
{\em The Modulo Problem} is given by
\begin{itemize}
\item Instance: a list of pairs $\langle A_{0}, B_{0} \rangle,\ldots , \langle A_{M-1}, B_{M-1} \rangle$ where 
$M>1$ and $A_i,B_i \in \mathbb{N} $  for $i=0,\ldots, M-1$.
\item Solution:  a natural number $N$ such that $f^N(3)=2$ where
$$
f(x) \; = \; A_jz+ B_j
 $$
if there exists $j \in \lbrace 0, 1, \ldots ,  M-1 \rbrace $ such that $x=Mz+j$.
\end{itemize}

\end{definition}

We need to show that given an instance of the Modulo Problem, 
 we can compute a   $ \Sigma_{1,0, 1 } $-sentence that is true in   
 $  \mathcal{T} (  \mathcal{L}_{ \mathsf{BT} }   )   $  
  if and only if  the instance has a solution. 
As we can see from the definition, to encode   the Modulo Problem, we need to encode  linear polynomials in one variable.

\begin{theorem} \label{UndecidabilitySigma101BT}
The fragment   $ \Sigma_{ 1 , 0, 1}^{ \mathcal{T} (  \mathcal{L}_{ \mathsf{BT} }   ) } $    is  undecidable.
\end{theorem}

\begin{proof}

We continue to work with the following definition of set membership 
\[
x \in y \equiv  \    x \in_{ \alpha }  y 
\   \   
\mbox{ where } 
 \alpha  \equiv   \;       \langle   \perp   \,  ,  \,     \perp^2      \rangle
\      .
\]

Recall that   $ t[ r \mapsto s ] $  denotes  the term we obtain by replacing each occurrence of $r$ in $t$ with $ s $. 
We encode  natural numbers as follows:  $  n \equiv \;   \perp^{ n+2 }   $. 
For example, the natural numbers  $0$, $1 $, $2$, $3$  are drawn  in Figure \ref{TranslationOfNaturalNumbers}.

\begin{figure}

\begin{center}

\begin{forest}
[
[ $ \perp $  ]    [ $ \perp $  ]
]
\end{forest}
\qquad
\begin{forest}
[
[
[ $ \perp $  ]    [ $ \perp $  ]
]
[ $ \perp $  ]
]
\end{forest}
\qquad
\begin{forest}
[
[
[ 
[ $ \perp $   ]    [ $ \perp $   ]
]
[ $ \perp $    ]
]
[ $ \perp $    ]
]
\end{forest}
\qquad
\begin{forest}
[
[
[ 
[
[ $ \perp $   ]    
[ $ \perp $   ]    
]
[ $ \perp $   ]
]
[ $ \perp $   ]
]
[ $ \perp $    ]
]
\end{forest}

\end{center}

\caption{
Translation of the natural numbers $0, 1, 2, 3 $ in the proof of Theorem  \protect\ref{UndecidabilitySigma101BT}.
}
\label{TranslationOfNaturalNumbers}
\end{figure}
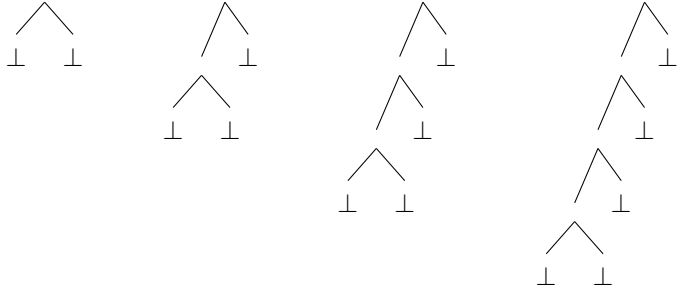

The next step is to associate    linear polynomials in one variable with  $ \mathcal{L}_{ \mathsf{BT}  }  $-terms.
Let $ L(z)  \equiv   \    z [   0   \mapsto    z  ]     $. 
If $z $ represents the natural number $q$, then $ L(z) $ represents the natural number $ 2q$
since $0 $ has exactly one occurrence in $ z $
 (see Figure  \ref{TranslationOfNaturalNumbers}).
 Recall that $ L^{ 0 }  (z) = z $ and $ L^{ k+1 } (z) = L ( L^{ k } (z) ) $.
Hence, if $ n>0 $, then $ L^{ n-1} (z) $ represents the natural number  $ nq $.
If $ n >0 $,  then the term $ m [ 0 \mapsto  L^{ n-1} (z)  ] $ represents the natural number $ nq + m $.
We complete our translation of linear polynomials  in one variable as follows: 
For any  formula $ \phi (x) $ where $x$ is a free variable in $ \phi $
\[
\phi (  n z +m   )     \equiv  \   \begin{cases}
 \phi ( m  )                                                                          &    \mbox{ if } n = 0 
 \vspace*{0.05cm}
\\
\phi (  m [ 0 \mapsto  L^{ n-1} (z)  ]     )                              &     \mbox{ if  }  n>0            \           .
\end{cases}
\]

Given an instance  $\langle A_{0}, B_{0} \rangle,\ldots , \langle A_{M-1}, B_{M-1} \rangle$ , 
we need to compute  a  $ \Sigma_{ 1 , 0, 1} $-sentence  $ \psi $ that is true in
 $\mathcal{T} ( \mathcal{L}_{ \mathsf{BT} }  )  $ 
 if and only if the instance has a solution. 
The sentence  $ \psi $ needs to say that there exists a  finite set  $T$ such that 
\begin{itemize}

\item[-]  $3\in T  $ and $ 2 \in T $

\item[-]  if  $  2 \neq  Mz +j \in T $ and $ 0 \leq j < M $, then   $ A_j   z +  B_j   \in T $.

\end{itemize}

With this in mind,  we   let $ \psi $ be the following sentence 

\begin{multline*}
\exists T   \;    \forall z  \sqsubseteq T     \Big[     \    
3 \in  T  
 \    \wedge      \  
\\
  \bigwedge_{ j= 0}^{ M-1}   
\Big(    \    
(   \      Mz + j  \in  T     \;   \wedge   \;      Mz + j \neq 2     \         )   
  \rightarrow  
 A_j z + B_j   \in  T  
\     \Big) 
\    \Big] 
\         .
\qedhere 
\end{multline*}

\end{proof}

The preceding result says that we have undecidability with only one existential quantifier and one bounded universal quantifier. 
Clearly, it is decidable whether a $ \Sigma$-sentence  with no occurrence of unbounded existential quantifiers is true. 
Hence, the next question is whether we can  have undecidability without bounded universal quantifiers. 
In Section \ref{UndecidabilityExistentialTheorySubstitutionOperator} and 
Section \ref{ManyToOneReductionOfPCPExistentialFragment}, 
we show that  existential theory of $   \mathcal{T} (  \mathcal{L}_{ \mathsf{BT} }   )  $ is undecidable.

 \subsection{The Modulo Problem II}

In this section, 
we give another proof of undecidability of   $ \Sigma_{1, 0, 2 }^{  \mathcal{T} (  \mathcal{L}_{ \mathsf{T} }   ) }   $  
by encoding the Modulo Problem. 
We let  $  x  \sqsubset  y  \equiv   \   x \sqsubseteq y   \,  \wedge   \,  x  \neq y  $.

\begin{proof} [Second proof of Theorem \ref{FirstUndecidabilityProof} ]

We continue to work with the following definition of set membership 
\[
x \in y \equiv  \    x \in_{ \alpha }  y 
\   \   
\mbox{ where } 
 \alpha  \equiv   \;    \langle   \perp   \,  ,  \,     \perp^2      \rangle
\      .
\]

We need to modify our translation of  addition and scalar multiplication since they use  the substitution operator. 
We encode natural numbers as follows:  $ n \equiv  \    \perp^{ n+1 }  $. 
On the meta level, we  associate a number $n $ with almost all terms  with $n+1$ occurrences of $\perp$.
We exclude  terms  that have $ \alpha $ as a subterm due to our definition of the membership relation.  
With this in mind,  we  encode  linear polynomial in one variable  as follows:
For  any natural numbers $n, m $ and any formula $ \psi (y) $ where $y$ is a free variable in $ \psi $  
\[
\psi  (  n  x  \oplus  m   )   \equiv    \   
 \begin{cases}
 \psi (nk+m  ) 
 &   \mbox{ if } x \mbox{ is the variable-free term  }    \perp^{ k+1 }
 \\  
  \psi ( m  )
&     \mbox{ if  }   n =   0       \vspace*{0.1cm}
\\
(       \   x = 0  \;  \wedge  \;   \psi (m)     \       )  
\   \vee   \   
&     \mbox{ if  }   n   \neq   0     \mbox{  and  }  m= 0 
\\
\quad
(      \   x \neq  0  \;  \wedge  \;   \psi (x^n )     \        )           \vspace*{0.1cm}
\\
(       \   x = 0  \;  \wedge  \;   \psi (m)     \       )  
\   \vee   \   
&       \mbox{ if }   n,  m    \neq   0 
\\
\quad
(      \   x \neq  0  \;  \wedge  \;   \psi (  \langle x^n , m   \rangle   )     \        )  
     \              .
     \end{cases}
\]

Recall that given an instance  $\langle A_{0}, B_{0} \rangle,\ldots , \langle A_{M-1}, B_{M-1} \rangle$ of the Modulo Problem, 
we need to compute  a $ \Sigma_{1, 0, 2 } $-sentence  $ \phi $ that is true in 
$   \mathcal{T} (  \mathcal{L}_{ \mathsf{T} }   )     $
  if and only if the instance has a solution.
 As before, we want $ \phi $ to express that  there exists a finite set  $T$ such that 
\begin{itemize}

\item[(1)]  $3\in T  $ and $ 2 \in T $

\item[(2)]  if  $  2 \neq  Mz  \oplus  j \in T $ and $ 0 \leq j < M $, then   $ A_j z \oplus B_j   \in T $.

\end{itemize}

We need to work harder  to find a formula that captures correctly (1)-(2). 
The problem is our definition of  $ \psi  (  n  x  \oplus  m   )  $. 
Assume $z $ represents the natural number $ q > 0 $ and   the antecedent in (2) holds. 
Assume $ A_j , B_j \neq 0 $ and $ A_j \neq M $. 
Then,   the binary tree $ \langle z^{ A_j }  \,    ,   \,    B_j \rangle   $  is an element of $T$ 
and represents the natural number $ A_j q + B_j $.
Now, assume $   2 \neq A_j q + B_j $ and    $ A_j q + B_j  = M d + i  $ where  $0 \leq  i < M $.  
For simplicity, assume  $ i \neq 0 $.
We need to ensure that $T$ contains a binary tree that represents the natural number $ A_i d + B_i $. 
This follows from Clause  (2)  if   $ \langle d^{ M} , i  \rangle $ is an element of $ T $.
Hence, the sentence $ \phi $  needs to    contain a subformula that ensures that  $ \langle d^{ M} , i  \rangle  \in T $.

Let $u$ be a finite binary tree with $ t+1 $ occurrences of $ \perp $. 
Let  $ \widetilde{u} $ denote the binary tree with the same number of $ \perp $ as $ u $ 
and that is of the form 
\[
\widetilde{u}    =  \begin{cases}
r              &  \mbox{ if  }   t = r 
\\
w^{ M }      &   \mbox{ if  } t = Ms 
\\
\langle w^{ M }  , r   \rangle     &    \mbox{ if  } t = Ms  + r
\end{cases}
\]
where $ 0 \leq r < M $,  $ s \neq 0 $ and $ w = \perp^{ s+1 }   $.
We can think of   $ \widetilde{u}  $ as the canonical representation of the natural number $ t $  with respect to   $M$.
The reasoning in the preceding paragraph shows that 
we  need  to ensure that    $ \widetilde{   A_j z \oplus B_j  }    \in T $ whenever $  A_j z \oplus B_j   \in T $.

 We want $ \phi $ to be  a sentence  of the form 
 
\begin{multline*}
\exists T \;    \forall R, S   \sqsubseteq  T   \;  
\Big[     \   
  \bigwedge_{ j = 1 }^{ M-1 }    \alpha_{j }  (R, S, T )    
  \   \wedge   \
  \\
3 \in  T  
  \   \wedge      \
  \\
   \bigwedge_{ j= 0}^{ M-1}   
   \Big(    \     
(        \      MR \oplus j   \in  T       \  \wedge   \       MR \oplus j \neq 2     \          )   
  \rightarrow  
  A_j R  \oplus B_j   \in  T  
  \     \Big) 
\    \Big]  
\end{multline*}
where  the quantifier-free  formula  $  \alpha_{j }  (R, S, T )      $ ensures that 
 $   \widetilde{   A_j z \oplus B_j  }    \in T   $ 
 when $  A_j z \oplus B_j  \in T  $.

We proceed to determine  $  \alpha_{j }  (R, S, T )      $. 
If  $A_j = 0 $ or $ A_j = M $, 
then we may  assume that  $A_j R \oplus B_j  $   is defined to be of the correct form.
 Hence, it suffices to  consider the following cases: 
(1)  $ 1 \leq A_j < M $, 
(2)  $ M< A_j $. 
 The idea is to express that there exists a sequence $ T_0,  \ldots , T_{K } $  of elements of $T$ such that 
 $ T_0 =    A_j z \oplus B_j   $  and 
 $ T_{K} =    \widetilde{   A_j z \oplus B_j  }   $. 
The formula  $  \alpha_{j }  (R, S, T )      $  describes how we obtain $T_{i+1} $ from $T_{ i } $. 
We introduce the following notation 
\[
x  +  k   \equiv   \  
\begin{cases}
x                                                         &  \mbox{ if }  k = 0 
\\
 x  ^{ \frown }   \perp^{ k^{ \prime }  }                        &  \mbox{ if }  k = k^{ \prime }  + 1  
 \end{cases}
\]
and 
\[
A x \oplus  B y  \oplus   m \equiv      \   
 \begin{cases} 
 x^{ A }  \,  ^{ \frown }  \,  y^{ B } 
 &    \mbox{ if  } m = 0 
\\
\\
\big\langle    x^{ A }   \,    ^{ \frown }  \,    y^{ B }     \, ,  \,  m   \big\rangle 
 &    \mbox{ if }  m \neq 0 
 \            .
\end{cases}
\]

We consider case (1). 
So, $ 1 \leq A_j < M $. 
Let $ M = k_j  A_j +  r_j $ where $ 0 \leq r_j < A_j $. 
Before  we define $  \alpha_{j }  (R, S, T )      $,  we explain how the formula  works. 
Assume for example $B_j \neq 0 $,  $ A_j = 2 $ and $ M = 5 $. 
At the start, we know that  $T$ contains  an element $ A_j R_0 \oplus B_j   $
(see the leftmost  tree in Figure \ref{ModuloProblemCaseOne}). 
The first step is to transform $ A_j R_0 \oplus B_j   $ into   an element of $T$ of the form 
$ A_j R_1  \oplus   (M-A_j )  1  \oplus   m_1 $ 
(see the second tree  from the left in Figure \ref{ModuloProblemCaseOne}) 
by decreasing $ R_0 $. 
Then, we want to transform  elements  of $T$ of the form 
$ A_j R_2  \oplus (M-A_j )   S_2  \oplus   m_2 $
 (see the third  tree from the left in Figure \ref{ModuloProblemCaseOne}) 
into  elements of $T$ of the form  $ A_j  R_3  \oplus   (M-A_j )   S_3  \oplus   m_3 $
 (see the fourth  tree from the left in Figure \ref{ModuloProblemCaseOne}) 
by decreasing $R_2$ with $ 1 \leq k  \leq k_j $ or by decreasing $m_2$  with $M$.

\begin{figure}

\begin{center}

\begin{forest}
[
[
[ $R_0$ ]   [$R_0$ ]
]
[$B_j$ ]
]
\end{forest}
\hspace*{0.2cm}
\begin{forest}
[
[
[
[
[
[ $R_1$ ]      [ $R_1$ ]  
]
[ $1$ ]  
]
[ $1$ ]   
]
[$1$ ]
]
[$m_1$ ]
]
\end{forest}
\hspace*{0.2cm}
\begin{forest}
[
[
[
[
[
[ $R_2$ ]      [ $R_2$ ]  
]
[ $S_2$ ]  
]
[ $S_2$ ]   
]
[$S_2$ ]
]
[$m_2$ ]
]
\end{forest}
\hspace*{0.2cm}
\begin{forest}
[
[
[
[
[
[ $R_3$ ]      [ $R_3$ ]  
]
[ $S_3$ ]  
]
[ $S_3$ ]   
]
[$S_3$ ]
]
[$m_3$ ]
]
\end{forest}

\end{center}

\caption{Case (1) in the second proof of Theorem \ref{FirstUndecidabilityProof}}
\label{ModuloProblemCaseOne}
\end{figure}
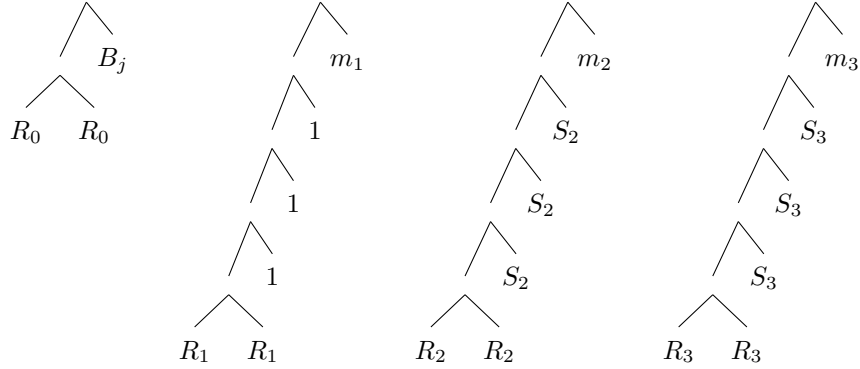

We let  $  \alpha_{j }  (R, S, T )      $ consist of the following conjuncts 

\begin{enumerate}

\item[(I) ]   for  $ 0  \leq k  \leq   k_j  $
\[   
(      \     
R = k   \;   \wedge   \;   A_j R \oplus B_j       \in  T   
\         ) 
\    
\rightarrow  
\     
B_j +  k A_j    \in   T   
\]

\item[(II) ]  for $ M \leq m  \leq  B_j   + 2M$
\[
   m   \in  T 
\   \rightarrow      \  
A_j 1  \oplus  (M-A_j )  1   \oplus     ( m - M  )     \in    T 
\]

\item[(III) ]
\begin{multline*}
(      \  
  1 \sqsubseteq R  
 \   \wedge   \  
 A_j ( R +  k_j )  \oplus B_j    \in  T   
   \         )
\   \rightarrow     \  
\\
A_j R  \oplus (M-A_j)   1   \oplus     ( B_j   +   (A_j - r_j )   )     \in  T 
\end{multline*}

\item[(IV) ]  for $ M \leq m  \leq  B_j   + 2M$
\begin{multline*}
A_j R \oplus (M-A_j )  S \oplus    m   \in  T 
  \   \rightarrow      \  
  \\
A_j ( R +  1 )    \oplus (M-A_j)   ( S + 1  )  \oplus     ( m - M  )      \in   T 
\end{multline*}

\item[(V) ]  for $ 0 \leq m  < M  $
\begin{multline*}
(      \     
S  \sqsubset  R    \;    \wedge    \;  
 A_j (R + k_j )  \oplus ( M-A_j )  S \oplus    m   \in  T 
 \      ) 
\   \rightarrow      \  
\\
A_j R   \oplus (M-A_j )   ( S +  1 )  \oplus     ( m  +  ( A_j - r_j )   )    \in    T 
\end{multline*}

\item[(VI) ]  for $ 0 \leq m  < M  $ and $ 1 \leq k  <  k_j  $
\begin{multline*}
(            \     
1 \sqsubseteq R   \   \wedge    \  
A_j (R  +  k )  \oplus  (M-A_j )   R \oplus    m   \in  T 
 \           ) 
\      \rightarrow     \  
\\
A_j R   \oplus (M-A_j )   R  \oplus     ( m  +  k A_j    )      \in  T 
\           .
\end{multline*}

\end{enumerate}

We consider case (2).
So, $ M< A_j $. 
Let $ A_j  = k_j  M +  r_j $ where $ 0 \leq r_j < M $. 
Before we we define $  \alpha_{j }  (R, S, T )      $,  we explain how the formula  works. 
Assume for example $ A_j = 5 $ and $ M = 2 $. 
At the start, we know that  $T$ contain an  element $ A_j R_0 \oplus B_j   $
(see the leftmost  tree in Figure \ref{ModuloProblemCaseTwo}). 
The first step is to transform $ A_j R_0 \oplus B_j   $ into  into an element of $T$ of the form 
$M R_1  \oplus (A_j - M)   S_1  \oplus    m_1 $
 (see the middle  tree  in Figure \ref{ModuloProblemCaseTwo}) 
by letting $ R_1 = R_0 +1 $ and $ S_1 = R_0 -1 $. 
Then, we want to transform  elements  of $T$ of the form 
$M  R_1  \oplus (A_j -M)   S_1  \oplus    m_1 $ 
into  elements of $T$ of the form  $ M R_2  \oplus (A_j -M)    S_2  \oplus    m_2 $
 (see the rightmost  tree  in Figure \ref{ModuloProblemCaseTwo})
by decreasing $S_1$ with $ 1 \leq k  \leq k_j $ or by decreasing $m_1$  with $M$.

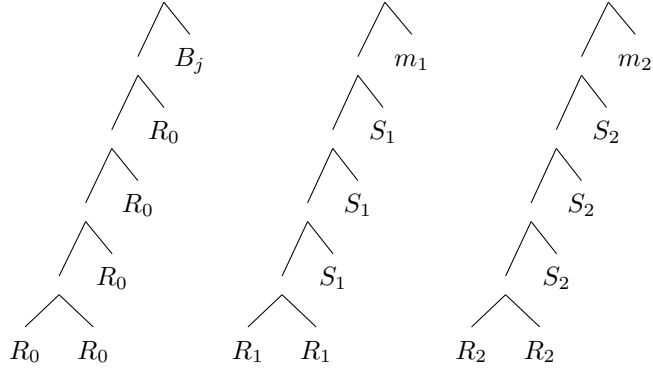
\begin{figure}

\begin{center}

\begin{forest}
[
[
[
[
[
[ $R_0$ ]      [ $R_0$ ]  
]
[ $ R_0 $ ]  
]
[ $ R_0 $ ]   
]
[$ R_0 $ ]
]
[$B_j$ ]
]
\end{forest}
\begin{forest}
[
[
[
[
[
[ $R_1$ ]      [ $R_1$ ]  
]
[ $S_1$ ]  
]
[ $S_1$ ]   
]
[$S_1$ ]
]
[$m_1$ ]
]
\end{forest}
\begin{forest}
[
[
[
[
[
[ $R_2$ ]      [ $R_2$ ]  
]
[ $S_2$ ]  
]
[ $S_2$ ]   
]
[$S_2$ ]
]
[$m_2$ ]
]
\end{forest}

\end{center}

\caption{Case (2) in the second proof of Theorem    \protect\ref{FirstUndecidabilityProof}}
\label{ModuloProblemCaseTwo}
\end{figure}

We let  $  \alpha_{j }  (R, S, T )      $ consist of the following conjuncts 

\begin{enumerate}

\item[(A) ]   
\[   
(      \     
R = 1   \;   \wedge   \;   A_j R  \oplus B_j   \in  T 
\         ) 
\    
\rightarrow  
\     
B_j +   A_j    \in   T   
\]

\item[(B) ]  for $ M \leq m  \leq  B_j +  2A_j  $
\[
 m     \in   T 
\        \rightarrow        \  
M 1    \oplus     ( m - M  )     \in   T 
\]

\item[(C) ]
\begin{multline*}
(        \   
1 \sqsubseteq R  \;  \wedge  \;  
A_j ( R +1  )  \oplus B_j   \in  T 
 \         )
\      \rightarrow        \
\\
M (R +  k_j  )   \oplus (A_j - M)    R   \oplus     ( B_j +  r_j   )     \in   T 
\end{multline*}

\item[(D) ]  for $ M \leq m  \leq  B_j + 2 A_j  $
\begin{multline*}
M  R \oplus (A_j -M)   S \oplus    m  \in  T 
\     \rightarrow      \  
\\
M ( R +1  )   \oplus (A_j-M)   S    \oplus      ( m - M  )     \in   T 
\end{multline*}

\item[(E)  ]  for $ 0 \leq m  < M  $
\begin{multline*}
(      \     
1  \sqsubseteq S  
  \;    \wedge    \;  
M R   \oplus (A_j -M)    (  S +1  ) \oplus    m  \in  T 
 \      ) 
\     \rightarrow      \   
\\
M (R  +  (k_j -1 ) )    \oplus (A_j -M)    S   \oplus     ( m  +   r_j    )     \in    T 
\end{multline*}

\item[(F)  ]  for $ 0 \leq m  < M  $
\[ 
 M R   \oplus (A_j-M)   1  \oplus    m    \in  T 
\       \rightarrow        \ 
M (R  +  (k_j -1 ) )     \oplus      ( m  +   r_j    )     \in   T 
\]

\item[(G)  ]     for $ M \leq m  \leq  B_j +  2A_j  $
\[
M R   \oplus    m     \in   T 
\        \rightarrow        \  
M ( R + 1  )    \oplus     ( m - M  )     \in   T 
\               .
\]

\end{enumerate}

This ends the second  proof of undecidability of   $ \Sigma_{1, 0, 2 }^{  \mathcal{T} (  \mathcal{L}_{ \mathsf{T} }   ) }    $.
\end{proof}

\section{Undecidable Fragments II}   \label{UndecidableFragmentsPartII}

Recall that  $ \mathcal{L}_{ \mathsf{T}^- } = \lbrace \sqsubseteq  \rbrace   $ and
 $ \mathcal{T}  ( \mathcal{L}_{ \mathsf{T}^- }  ) $  denotes   the restriction of     $ \mathcal{T}  ( \mathcal{L}_{ \mathsf{T} }  ) $
to  $  \mathcal{L}_{ \mathsf{T}^- }  $. 
In this section, we show that the $ \Sigma$-theory of $ \mathcal{T}  ( \mathcal{L}_{ \mathsf{T}^- }  ) $
is undecidable. 
That is, we show that there cannot exist an algorithm that takes as input a    
$ \mathcal{L}_{ \mathsf{T}^- } $  $ \Sigma$-sentence $ \phi $
and decides whether $ \phi $ is true in   $ \mathcal{T}  ( \mathcal{L}_{ \mathsf{T}^- }  ) $. 
The proof we give is a modification of  the proof of Theorem \ref{FirstUndecidabilityProof}.
The  basic idea is the same but we  need to work with multivalued functions since our language does not have  function symbols. 
A multivalued function from  $ \mathbf{H}^{ n} $ to $ \mathbf{H} $ is just a relation 
$ R \subseteq    \mathbf{H}^{ n +1 }  $ such that for all $ x_1,  \ldots , x_n \in \mathbf{H} $ there exists $ y \in \mathbf{H} $ such that 
$ R (x_1, \ldots , x_n , y ) $ holds. 
Recall that $ \mathbf{H} $ denotes the set of all finite full binary trees.

In the proof of Theorem \ref{FirstUndecidabilityProof}, 
we showed that the instance 
  $  \langle  a_{1}, b_{1}  \rangle  ,\ldots , \langle a_{n}, b_{n}  \rangle $ of PCP
has a solution if and only if the sentence 
\begin{multline*}
\phi \equiv   \  
\exists T  \;   \forall L, R  \sqsubseteq T   \;    
 \Big[         \    
\bigvee_{ i = 1 }^{ n } 
 \langle   a_i , \,  b_i  \rangle   \in  T     
\     \wedge   \   
\Big(   \       
\big(        \               \langle  L , \,  R   \rangle  \in  T    \;  \wedge   \;   L \neq R       \       \big)  
\rightarrow
\\
\bigvee_{ i = 1 }^{ n }     
\langle  L ^{ \frown } a_i    ,  \,     R ^{ \frown } b_i   \rangle   \in   T   
 \       \Big) 
\        \Big] 
\end{multline*}
is true in    $  \mathcal{T} (  \mathcal{L}_{ \mathsf{T} }   ) $. 
The idea is to  find a $  \mathcal{L}_{ \mathsf{T}^- }  $-sentence $ \psi $ that is true in 
$ \mathcal{T}  ( \mathcal{L}_{ \mathsf{T}^- }  ) $ if and only if $  \phi $  is true in  
$  \mathcal{T} (  \mathcal{L}_{ \mathsf{T} }   ) $. 
We use the following ingredients to construct  $ \psi $
\begin{itemize}

\item[-] we give a  definition of  $ s \in_{ \alpha }  T $ over   $ \mathcal{T}  ( \mathcal{L}_{ \mathsf{T}^- }  ) $

\item[-] we replace  $  \langle x , y \rangle  $ with a multivalued pairing function   
$ \mathsf{Pair}_{ \beta , \gamma }  (x , y , z  )  $ 
that takes $ \beta $ and $ \gamma $ as parameters

\item[-]   we replace    $  L ^{ \frown } a_i    $  with a multivalued function    $ \mathsf{Conc}_{ a_i }  (L, y )    $

\item[-]   we replace    $  R ^{ \frown } b_i    $  with a multivalued function    $ \mathsf{Conc}_{ b_i }  (R, y )    $.

\end{itemize}
The multivalued functions 
$ \mathsf{Conc}_{ a_i }  (x, y )    $, $ \mathsf{Conc}_{ b_i }  (x, y )    $ 
will be defined from the more  elementary multivalued  functions    
$ \mathsf{Conc}_{ 0 }  (x, y )    $,   $ \mathsf{Conc}_{ 1 }  (x, y )    $
by composing these functions in the obvious way.

The sentence $ \phi $ is true in  $  \mathcal{T} (  \mathcal{L}_{ \mathsf{T} }   ) $ if and only if 
 there exists   a finite nonempty sequence $i_{1},  \ldots  ,i_{m}$ of indexes such that 
$
a_{i_{1}} a_{i_{2}}\ldots  a_{i_{m}} = b_{i_{1}}  b_{i_{2}}   \ldots  b_{i_{m}} 
$.
The sequence  $ i_1,  \ldots , i_m $ exists  if and only if     (this is what  $ \psi $ captures)
\begin{itemize}

\item[-] there exist a sequence  of finite binary trees  $ A_1, B_1,   \ldots , A_m , B_m $ such that

\item[-]   $ \mathsf{Conc}_{ a_{i_{ 1 }  } }  (  \delta  , A_{ 1 } )    $   and $ \mathsf{Conc}_{ b_{i_{ 1 }  } }  ( \delta   , B_{ 1 } )    $ 
for some parameter $ \delta $

\item[-]  $ \mathsf{Conc}_{ a_{i_{ k+1 }  } }  (A_k  , A_{ k+1 } )    $   and $ \mathsf{Conc}_{ b_{i_{ k+1 }  } }  (B_k  , B_{ k+1 } )    $

\item[-]  $ A_m = B_m $.

\end{itemize}

We let  $  x  \sqsubset  y  \equiv   \   x \sqsubseteq y   \,  \wedge   \,  x  \neq y  $.

\begin{theorem}   \label{UndecidabilitSubtreeRelationOnFiniteBinaryTree}

The fragment $  \mathsf{Th}^{ \Sigma }  (  \mathcal{T}  ( \mathcal{L}_{ \mathsf{T}^- }  )  )       $   is undecidable. 

\end{theorem}

\begin{proof}

Given an instance  $ \langle a_{1}, b_{1}  \rangle  ,\ldots ,  \langle  a_{n}, b_{n}   \rangle  $ of PCP, 
we  compute a $ \mathcal{L}_{ \mathsf{T}^- } $-sentence $ \psi $ that is true in  
$ \mathcal{T}  ( \mathcal{L}_{ \mathsf{T}^- }  ) $ 
if and only if there exists a finite nonempty sequence $i_{1},  \ldots  ,i_{m}$ of indices such that 
$
a_{i_{1}} a_{i_{2}}\ldots  a_{i_{m}} = b_{i_{1}}  b_{i_{2}}   \ldots  b_{i_{m}} 
$.
Our  starting point is to translate the two string operations $ x \mapsto x0 $,  $ \;    x \mapsto x1  $ 
as multivalued functions.
Our translations need to ensure that we can tell the two functions apart.

We capture the string operation  $ x \mapsto x0 $  as a multivalued function as follows
  \begin{multline*}
\mathsf{Zero} (x , y  )  \equiv    \    
\exists z ,   w   \sqsubseteq y    \;   \big[          \   
x  \sqsubseteq   z    \;  \wedge  \;     x  \sqsubseteq     w        \;  \wedge  \;    
z   \not\sqsubseteq w   \;  \wedge  \;         w  \not\sqsubseteq  z   \;  \wedge  \;     
\\
\forall r \sqsubseteq y   \;   [     \    
r = y   \;   \vee   \;    r =  z      \;   \vee   \;    r =  w   \;   \vee   \;      r  \sqsubseteq  x    \     ]   
\      \big]  
\        .
\end{multline*}

Figure \ref{GobalInterpretationZero} shows the subtrees of  $y$ when   $ \mathsf{Zero} (x , y  )  $ holds.
Observe that for each $x$ with  at least two distinct subtrees, we can find $y$ such that $ \mathsf{Zero} (x , y  )  $ holds. 
For example,  if $u $ and $v$ are distinct subtrees of $x $, then  we can let $y$ be   
$ \big\langle  \,   \langle x, u \rangle  ,  \langle x, v \rangle   \,   \big\rangle  $.

\begin{figure}
\centering 

\[
\begin{tikzcd}
   &    y    
   \\
 z   \arrow[ru]  
 &   &   w    \arrow[lu]          
 \\
&   x      \arrow[lu]      \arrow[ru]  
\end{tikzcd}
\]

\caption{ 
Translation of the string operation $ x  \mapsto x0 $ 
in the proof of Theorem  \protect\ref{UndecidabilitSubtreeRelationOnFiniteBinaryTree}.
$ X  \rightarrow  Y $ means   $  X  \sqsubset Y  $.  
If there is no directed path connecting  $X$ and $Y$, then $X$ and $Y$ are incomparable   with respect to   $ \sqsubseteq $.
If $ Y \sqsubseteq y   \;   \wedge  \;   x \sqsubseteq Y $, then $Y$ appears in the diagram. 
}
\label{GobalInterpretationZero}
\end{figure}
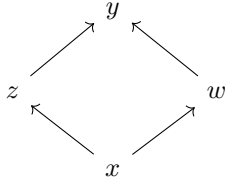

We capture the string operation  $ x \mapsto x1 $  as a multivalued function as follows
  \begin{multline*}
\mathsf{One} (x , y  )  \equiv    \    
\exists s ,  t   \sqsubseteq y   \;   \big[          \   
\mathsf{Zero} (x, s )      \    \wedge   \   \mathsf{Zero} (x, t  )     \   \wedge   \  
s   \not\sqsubseteq t   \;  \wedge  \;         t  \not\sqsubseteq  s   \;  \wedge  \;       
\\
\forall   r \sqsubseteq y   \;   [     \ 
r  = y    \;    \vee    \;   r  \sqsubseteq s      \;    \vee    \;   r  \sqsubseteq t    \         ]   
\      \big]  
\           .
\end{multline*}

Figure \ref{GobalInterpretationOne} shows the subtrees of  $y$ when   $ \mathsf{One} (x , y  )  $ holds.
Observe that for each  $x$ with at least three distinct subtrees,   we can find $y$ such that $ \mathsf{One} (x , y  )  $ holds. 
For example, if $ u, v, w $ are three distinct subtrees of $x$, then  we can let $y$ be   
\[
\Big\langle  \;  
 \big\langle  \,   \langle x, u \rangle  ,  \langle x, v \rangle   \,   \big\rangle     
 \;    ,     \;   
  \big\langle  \,   \langle x, u  \rangle  ,  \langle  x, w  \rangle   \,   \big\rangle     
\;    \Big\rangle  
\      . 
 \]

\begin{figure}
\centering 

\[
\begin{tikzcd}
   &   &  &       y    
   \\
   &  s  \arrow[rru]    &      &   &  &    t   \arrow[llu]
   \\
  z   \arrow[ru]  
 &   &  w    \arrow[lu]      
 &  &  u   \arrow[ru]  
 &  &    v    \arrow[lu]        
 \\
   &   &  &     x      \arrow[lllu]      \arrow[lu]      \arrow[ru]      \arrow[rrru]  
\end{tikzcd}
\]

\caption{ 
Translation of the string operation $ x  \mapsto x1 $
 in the proof of Theorem  \protect\ref{UndecidabilitSubtreeRelationOnFiniteBinaryTree}.
$ X  \rightarrow  Y $ means   $  X  \sqsubset Y  $.  
If there is no directed path connecting  $X$ and $Y$, then $X$ and $Y$ are incomparable  with respect to   $ \sqsubseteq $.
If $ Y \sqsubseteq y   \;   \wedge  \;   x \sqsubseteq Y $, then $Y$ appears in the diagram. 
}
\label{GobalInterpretationOne}
\end{figure}
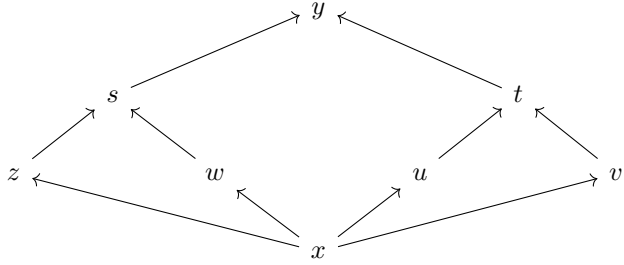

For each binary string $w$, 
we  capture  the string operation $ x \mapsto xw $ as multivalued function by recursion as follows: 
Let    $  \mathsf{Conc}_0 (x, y )   \equiv    \   \mathsf{Zero} (x, y ) $  
and  $  \mathsf{Conc}_1 (x, y )   \equiv    \   \mathsf{One} (x, y ) $.
Let $ w = w_0 w_1  $ where $ w_1 \in \lbrace 0, 1 \rbrace $ and $ w_0  \in \lbrace 0, 1 \rbrace^+ $. 
Let 
\[
\mathsf{Conc}_w (x, y )   \equiv    \  
   \begin{cases}
\exists z  \sqsubseteq y    \;   [    \     \mathsf{Conc}_{ w_0 }   (x,   z   )   
\   \wedge    \  
\mathsf{Zero} (z, y )    
\    ]      
&   \mbox{if  }      w_1 =  0 
\\
\\
\exists z    \sqsubseteq y      \;   [    \     \mathsf{Conc}_{ w_0 }   (x,   z   )   
\   \wedge    \  
\mathsf{One} (z, y )    
\    ]    
&   \mbox{if  }      w_1 =  1
\         .
\end{cases}
\]

What remains is to define a multivalued  pairing function and a notion of set membership. 
We  define the membership relation  as follows 
  \begin{multline*}
x \in_{ \alpha }  y    \equiv    \    
x \not\sqsubseteq \alpha   \   \wedge    \     \alpha  \not\sqsubseteq  x  
 \   \wedge    \  
 \exists  z   \sqsubseteq  y    \;  
 \forall u \sqsubseteq z  \;   [   \   u = z   \   \vee    \  u \sqsubseteq x     \   \vee    \    u  \sqsubseteq  \alpha    \   ]  
\end{multline*}
Observe that if  $ x \in_{ \alpha }  y  $, then the variable $z$ in the formula is such that 
 $ z =  \langle x, \alpha \rangle  $ or $ z = \langle \alpha , x \rangle \ $
 since $ x $ and $ \alpha $ are incomparable with respect to  the subtree relation.

We define a multivalued  pairing function on finite binary trees as follows 
(Figure \ref{GobalInterpretationPairing} shows the subtrees of  $z$ when   $ \mathsf{Pair}_{ \beta , \gamma }  (x , y , z  ) $ holds)
  \begin{multline*}
\mathsf{Pair}_{ \beta , \gamma }  (x , y , z  )  \equiv    \    
\beta \not\sqsubseteq  \gamma   \   \wedge   \     \gamma  \not\sqsubseteq      \beta 
\   \wedge   \   
\\
\exists s,   t,   s_0 ,  t_0    \sqsubseteq z   \;   \Big[          \   
\mathsf{Zero} (x, s )      \    \wedge   \   \mathsf{Zero} (y, t  )     \   \wedge   \  
\\
\beta \not\sqsubseteq   s   \   \wedge   \   s \not\sqsubseteq    \beta  
\   \wedge    \  
\forall r \sqsubseteq s_0   \;  [   \   r = s_0   \  \vee   \  r  \sqsubseteq \beta     \  \vee   \     r  \sqsubseteq   s   \   ]
\   \wedge    \  
\\
\gamma \not\sqsubseteq   t   \   \wedge   \   t \not\sqsubseteq    \gamma  
\   \wedge    \  
\forall r \sqsubseteq t_0   \;  [   \   r = t_0   \  \vee   \  r  \sqsubseteq \gamma     \  \vee   \       r  \sqsubseteq   t   \   ]
\   \wedge    \    
\\
\forall   r \sqsubseteq z   \;   [     \ 
r  = z    \;    \vee    \;   r  \sqsubseteq s_0       \;    \vee    \;   r  \sqsubseteq t_0     \         ]   
\            \Big]  
\      .
\end{multline*}

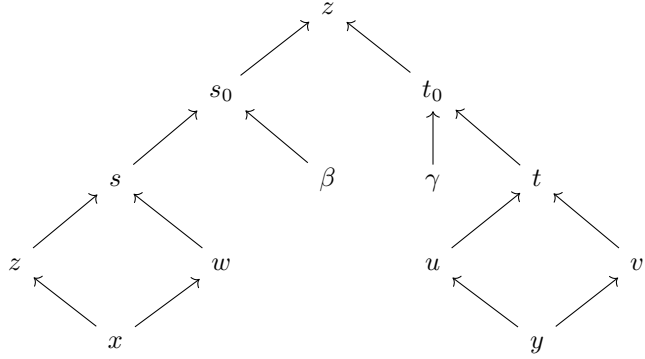
\begin{figure}
\centering 

\[
\begin{tikzcd}
   &   &  &       z   
   \\
   &    &    s_0   \arrow[ru]   &      &     t_0   \arrow[lu]
   \\
   &  s  \arrow[ru]    &      & \beta \arrow [lu]   &   \gamma \arrow[u] &    t   \arrow[lu]
   \\
  z   \arrow[ru]  
 &   &  w    \arrow[lu]      
 &  &  u   \arrow[ru]  
 &  &    v    \arrow[lu]        
 \\
   &      x      \arrow[lu]      \arrow[ru]      
&      &   &  &     y      \arrow[lu]      \arrow[ru]      
\end{tikzcd}
\]

\caption{
The multivalued pairing function  $ \mathsf{Pair}_{ \beta , \gamma }  (x , y , z  )  $
   in the proof of Theorem     \protect\ref{UndecidabilitSubtreeRelationOnFiniteBinaryTree}.
The trees $ \beta ,  \gamma  $ are parameters that allow us to recognize that an element represents a pair. 
$ X  \rightarrow  Y $ means   $  X  \sqsubset Y  $.  
If $ Y \sqsubseteq z   $ and  
$   x \sqsubseteq Y  \;   \vee  \;      y \sqsubseteq Y  \;   \vee  \;    \beta \sqsubseteq Y     \vee  \;    \gamma \sqsubseteq Y  $, 
then $Y$ appears in the diagram. 
}
\label{GobalInterpretationPairing}
\end{figure}

We let $ \psi $ be the following sentence  
\begin{multline*}
\psi \equiv   \  
\exists \alpha  \beta    \gamma   \delta   \;   
\exists  T  \;   \forall L, R, S  \sqsubseteq T   \;    
 \Big[         \    
\bigvee_{ i = 1 }^{ n } 
\exists   x,  y,  z   \sqsubseteq T    \;     \big[     \  
\\
    \mathsf{Conc}_{ a_i  }   (  \delta  , x  )    \    \wedge     \       \mathsf{Conc}_{ b_i  }   (  \delta , y  )  
 \    \wedge     \     
 \mathsf{Pair}_{   \beta ,  \gamma   }    (x, y, z )     
  \    \wedge     \     
  z  \in_{  \alpha  }   T       
    \           \big]
\     \wedge   \   
\\
\Big(   \       
\big(        \       \mathsf{Pair}_{   \beta ,  \gamma   }  (L, R, S )      
\;  \wedge   \;  S   \in_{  \alpha  }   T    \;  \wedge   \;   L \neq R       \           \big)  
\rightarrow
\bigvee_{ i = 1 }^{ n }     
\exists   u,  v ,    w   \sqsubseteq T   \;      \Big[     \  
\\
    \mathsf{Conc}_{ a_i  }   (  L , u  )    \    \wedge     \       \mathsf{Conc}_{ b_i  }   (  R , v  )  
 \    \wedge     \     
 \mathsf{Pair}_{   \beta ,  \gamma   }   (u, v, w )     
  \    \wedge     \     
  w  \in_{  \alpha  }   T          
  \             \Big]   
 \       \Big) 
\        \Big] 
\          . 
\qedhere
\end{multline*}

\end{proof}

\section{Undecidable Fragments III}   
\label{UndecidableFragmentsPartIII}

In this section, 
we show that the existential theory of  $ \mathcal{T} ( \mathcal{L}_{ \mathsf{BT} }  ) $ is undecidable by constructing  an existential interpretation of  $ ( \mathbb{N} , 0,1, +, \times ) $ in  $ \mathcal{T} ( \mathcal{L}_{ \mathsf{BT} }  )  $. 
In Section  \ref{AnalogueOfHilbers10Problem}, we show that this implies that the analogue of Hilbert`s 10th Problem for 
$ \mathcal{T} ( \mathcal{L}_{ \mathsf{BT} }  ) $ is undecidable. 
In Section \ref{UndecidableFragmentsPartIV}, 
we give a direct proof of this result by constructing a many-to-one reduction of Post`s Correspondence Problem. 
The proof builds on the methods we develop in this section.

\subsection{First Basic Lemma}

The first step towards an existential interpretation of 
$ ( \mathbb{N} , 0,1, +, \times ) $ in  $ \mathcal{T} ( \mathcal{L}_{ \mathsf{BT} }  )  $
is to associate the set $ \mathbb{N} $ of natural numbers with an existentially definable class of finite binary trees. 
In Section \ref{UndecidableFragmentsPartI},
 we associated  the set of natural numbers  with  a class $ \mathsf{N} $ of  binary trees   
 by mapping the natural number $n$ to the finite binary tree  $   \perp^{ n+1 }   $.
We can  translate addition  on $ \mathsf{N} $ as follows 
\begin{multline*}
x + y = z \equiv  \  
(   \   x = \perp    \;  \wedge   \;     z = y    \  )  
\   \vee   \  
(      \   
x \neq \perp   \;  \wedge    \;   
z = x [   \,  \perp^2   \,  \mapsto  \,    \langle  y  ,  \perp   \rangle     \,    ]   
\      . 
\end{multline*}
The translation of multiplication is a bit more demanding. 
We develop the tools we need to handle multiplication in Section  \ref{SecondBasicLemma}. 
In this section, we show that  a number of  $ \mathsf{N} $-like classes  of finite binary trees   are existentially definable
in  $ \mathcal{T} ( \mathcal{L}_{ \mathsf{BT} }  )  $.

Given a binary tree $t$, we could have  associated the natural number $n$ with the binary tree $ t^{ n+1 } $.
We can generalize this construction. 
Given  a finite  list $ t_1, \ldots , t_m $ of  binary trees, we can associate the natural number $n$ with 
the binary tree $ \langle   t_1, \ldots , t_m \rangle^{ n+1 } $
(see Section \ref{SecondNotationSystem} for the notation).  
The definition of  $ \langle  t_1, \ldots , t_m \rangle^{ n } $ does  not refer directly to the substitution operator. 
We work with a natural generalization  that uses the substitution operator. 
The construction is so simple that it has a quantifier-free  definition  in $ \mathcal{T} (  \mathcal{L}_{ \mathsf{BT} }   )  $.

To improve readability, it will  occasionally  be more convenient to represent  finite binary trees using 
notation that is closer to their  visual form.

\begin{definition}  \label{ThirdNotationSystemSecondPart}

Let 
\[
\langle  x_1  \rangle  \equiv    \    x_1  
\     \   
\mbox{  and   }  
\       \    
\langle   x_1,   \ldots , x_m , x_{ m+1 }   \rangle          \equiv     \ 
    \langle   \langle   x_1,   \ldots , x_m  \rangle  \,   ,   \,  x_{ m+1 }   \rangle 
\             .
\]

\end{definition}

Recall that $ \mathbf{H} $ denotes the set of all finite full  binary trees.

\begin{definition}    \label{ThirdNotationSystem}

Let $ \alpha  \in \mathbf{H} $. 
Let $ s_1,  \ldots , s_n  \in \mathbf{H} $ be   such that  
$  \alpha  $ is not a subtree of $  s_i $ for all $ i \leq n $  and $ s_n  \neq s_j $ for all $ j < n$.
Let 
\[
 \frac{ 1 }{  \alpha ,   \vec{s}  }     \equiv    \        \langle   \alpha , s_1,  \ldots , s_n   \rangle 
 \                 \   
 \mbox{  and   } 
 \            \
   \frac{ m+ 1 }{  \alpha ,   \vec{s}  }   \equiv     \   
    \frac{  1 }{  \alpha ,   \vec{s}  }     \big[   \,     \alpha   \mapsto      \frac{ m }{  \alpha ,  \vec{s}    }     \,      \big]  
    \                  .
\]
Let 
\[
\mathbb{N}^{ \alpha }_{ \vec{s} } = \lbrace   \frac{ m  }{ \alpha ,   \vec{ s} }   \in \mathbf{H}  :    \      
 m  \in \mathbb{N}   \    \wedge   \    m \geq   1    \       \rbrace 
 \           .
\]

\end{definition}

For example, 
the binary trees 
$  \frac{ 5 }{ \alpha , s  }     \     $, 
$  \;   \frac{ 2 }{  \langle  \alpha ,  \beta \rangle  , s , t   }        \      $,  
$ \;    \frac{ 2 }{ \alpha , r, s, t }    $ 
are drawn   in Figure \ref{ExamplesNumberLikeConstruction}.

\begin{figure}

\begin{center}

\begin{forest}
[   
[
[
[
[
[$ \alpha  $ ]
[$ s $]
]
[$ s $]
]
[$ s $] 
]
[$ s $] 
]
[$ s $]
]
\end{forest}
\qquad
\begin{forest}
[   
[
[
[
[
[$   \alpha  $ ]
[$  \beta   $]
]
[$  s  $]
]
[$ t  $] 
]
[$ s   $] 
]
[$ t  $]
]
\end{forest}
\qquad
\begin{forest}
[   
[
[
[
[
[
[$ \alpha  $]
[$ r   $ ]
]
[$  s  $]
]
[$  t   $]
]
[$  r   $] 
]
[$ s  $] 
]
[$ t   $]
]
\end{forest}

\end{center}
\caption{
Visualization of  the binary trees 
$  \frac{ 5 }{ \alpha , s  }     \  $, 
$  \;   \frac{ 2 }{  \langle  \alpha ,  \beta \rangle  , s , t   }      \      $,  
$ \;    \frac{ 2 }{ \alpha , r, s, t }    $.
}
\label{ExamplesNumberLikeConstruction}
\end{figure}
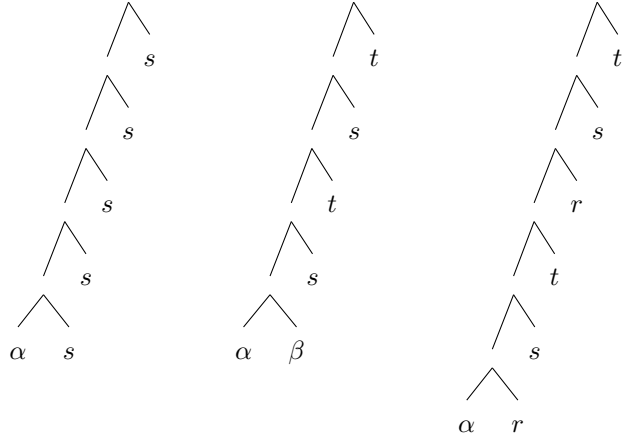

\begin{lemma}   \label{NaturalNumberLikeObjects}

Let $ \alpha  \in \mathbf{H} $. 
Let $ s_1,  \ldots , s_n  \in \mathbf{H} $ be   such that  
$  \alpha   $ is not a subtree of $  s_i $ for all $ i \leq n $  and $ s_n \neq s_j $ for all $ j < n$.
Then, for all $ T \in \mathbf{H} $ 
\begin{multline*}
T  \in  \mathbb{N}^{ \alpha } _{ \vec{s}   } 
\   \Leftrightarrow  \   
T  =  \frac{ 1 }{ \alpha , \vec{s} }
\            \vee       \   
\Big(    \    
\frac{ 2 }{ \alpha , \vec{s} }   \sqsubseteq  T   
 \       \wedge         \
T  =  \frac{ 1 }{ \alpha , \vec{s} }     \Big[   \     \alpha   \mapsto    
   T   \big[  \;    \frac{ 2 }{ \alpha , \vec{s} }        \mapsto  \frac{ 1 }{ \alpha , \vec{s} }      \big]   \      \Big]  
\      \Big)   
\          .
\end{multline*} 

\end{lemma}

\begin{proof}

The left-right implication of the claim  is straightforward. 
Let the size of a binary tree $T  $ be the number of nodes in $T$.
We prove  by induction on the size  of  $T$ that  
 \begin{multline*}
T =  \frac{ 1 }{ \alpha , \vec{s} }
\   \vee   \   
\Big(    \    
\frac{ 2 }{ \alpha , \vec{s} }   \sqsubseteq  T   
 \;   \wedge   \;  
T =  \frac{ 1 }{ \alpha , \vec{s} }     \Big[   \     \alpha   \mapsto    
   T   \big[  \;    \frac{ 2 }{ \alpha , \vec{s} }        \mapsto  \frac{ 1 }{ \alpha , \vec{s} }      \big]     \      \Big]  
\      \Big)   
   \tag{*} 
\end{multline*} 
implies $ T  \in  \mathbb{N}^{ \alpha } _{ \vec{s}   }  $.

Assume $ T $ satisfies (*). 
We need to show that $  T \in  \mathbb{N}^{ \alpha } _{ \vec{s}   }   $. 
If  $   T =  \frac{ 1 }{ \alpha , \vec{s} }  $, 
then certainly $ T   \in   \mathbb{N}^{ \alpha } _{ \vec{s}   }   $. 
Otherwise, 
by the second disjunct in (*),   $ \frac{ 2 }{ \alpha , \vec{s} }    \sqsubseteq  T    $. 
Let 
\[
S =   T   \big[  \;    \frac{ 2 }{ \alpha , \vec{s} }        \    \mapsto     \        \frac{ 1 }{ \alpha , \vec{s} }      \big] 
\       .
\]
Then,  $S$ is strictly smaller  than $T$. 
By the second disjunct in (*)
\[
T =    \frac{ 1 }{ \alpha , \vec{s} }       \big[   \;     \alpha    \,  \mapsto      \,   S     \;        \big]  
\     .  
\]
By Definition  \ref{ThirdNotationSystem},     
$      \      
 \frac{ 1 }{ \alpha , \vec{s} }      =  \langle  \alpha  , s_1, \ldots , s_n  \rangle  
$.
Since $  \alpha  $ is not a subtree of any    $s_i $
\[
\begin{array}{r c l}
T 
& =   &  
\frac{ 1 }{ \alpha , \vec{s} }       \big[   \;     \alpha    \,  \mapsto      \,   S     \;        \big]   
\\
\\
& =  &  
 \langle  \alpha  , s_1, \ldots , s_n  \rangle    \big[   \;     \alpha    \,  \mapsto      \,   S     \;        \big]  
\\
\\
& =   &  
   \langle  S  , s_1, \ldots , s_n   \rangle  
\     .   \tag{**} 
\end{array}
\]

We know that  $    \frac{ 2 }{ \alpha , \vec{s} }       \sqsubseteq  T   $. 
By Definition  \ref{ThirdNotationSystem},    
$   \    
 \frac{ 2 }{ \alpha , \vec{s} }      =   \langle    \alpha   , s_1, \ldots , s_n ,  s_1, \ldots , s_n    \rangle 
$.
Since $s_n \neq s_j $ for all  $ 1 \leq j < n $, 
it follows from $   \frac{ 2 }{ \alpha , \vec{s} }      \sqsubseteq  T   $ and  (**) that  we have one of  the following cases: 
(i)  $  S =    \frac{ 1 }{ \alpha , \vec{s} }       $, 
(ii)  occurrences  of  $  \frac{ 2 }{ \alpha , \vec{s} }     $  in $T$ can only be found in $ S$. 
In case of (ii)
\[
\begin{array}{r c l}
S
&=     &   
 T   \big[  \;    \frac{ 2 }{ \alpha , \vec{s} }        \mapsto  \frac{ 1 }{ \alpha , \vec{s} }      \big]   
 \\
 \\
 &=    &  
 \langle  S, s_1, \ldots , s_n  \rangle        \big[  \;    \frac{ 2 }{ \alpha , \vec{s} }        \mapsto  \frac{ 1 }{ \alpha , \vec{s} }      \big]    
 \\
 \\
 &=  & 
   \big \langle   S   \big[  \;   
    \frac{ 2 }{ \alpha , \vec{s} }        \mapsto  \frac{ 1 }{ \alpha , \vec{s} }      \big]        \;   ,
  s_1, \ldots , s_n       \big\rangle 
  \\
  \\
  &   =  & 
  \langle  \alpha  , s_1, \ldots , s_n  \rangle      \Big[         \     \alpha        \           \mapsto        \    
   S   \big[  \;    \frac{ 2 }{ \alpha , \vec{s} }        \mapsto  \frac{ 1 }{ \alpha , \vec{s} }      \big]   \      \Big]  
  \\
  \\
  &   =  & 
  \frac{ 1 }{ \alpha , \vec{s} }     \Big[   \     \alpha   \mapsto    
   S   \big[  \;    \frac{ 2 }{ \alpha , \vec{s} }        \mapsto  \frac{ 1 }{ \alpha , \vec{s} }      \big]   \      \Big]  
   \                  .
\end{array}
\]

We thus see that in case of either (i) or (ii), $S$ satisfies (*). 
Hence, by the induction hypothesis, $ S  \in   \mathbb{N}^{ \alpha } _{ \vec{s}   }    $. 
It then follows from (**)    that   $ T  \in      \mathbb{N}^{ \alpha } _{ \vec{s}   }   $. 
\end{proof}

\subsection{Concatenation  with the Substitution Operator}
\label{ExistentialInterpretationFreeSemigroupsLengthOperator}

Given a finite alphabet $ A = \lbrace a_1, \ldots , a_n \rbrace $, 
let $ \varepsilon $ denote the empty string and let  $ A^*$ denote  the set of all finite  strings over the alphabet $ A $.
Let  $ ^{ \frown } $ denote the  concatenation operator. 
 For a fixed letter $ 1 \in  A$, the operator 
 $ \vert \cdot  \vert  :  A^*  \to  \lbrace 1 \rbrace^* $
   takes a string and  replaces each letter with $ 1$. 
   For example,  
   $ \vert \varepsilon \vert = \varepsilon $ and   $ \vert  010 \vert = 111 $. 
We refer to $ \vert \cdot  \vert  $ as a $1$-tally length function. 
In this section, we use  Lemma \ref{NaturalNumberLikeObjects}  to give a simple existential interpretation 
of the extended free semigroup    
$ (   A^*, \varepsilon , a_1, \ldots , a_n , ^{ \frown}  , \vert \cdot  \vert  ) $ 
 in  $ \mathcal{T} (  \mathcal{L}_{ \mathsf{BT} }   )  $.
 In 1977,   Makanin \cite{Makanin1977}   showed   that  
$   \mathsf{Th}^{ \exists }    (   A^* , \varepsilon  , a_1, \ldots , a_n , ^{ \frown}   ) $  is decidable. 
But, 
for $n \geq 2 $,  
decidability of  $  \mathsf{Th}^{ \exists }   (   A^* , \varepsilon , a_1, \ldots , a_n , ^{ \frown}  , \vert \cdot  \vert  ) $ 
is a long standing open problem that dates back to the works of   B\"uchi  and Senger     \cite{Senger1988}.

Although decidability of 
$   \mathsf{Th}^{ \exists }    (   A^* , \varepsilon  , a_1, \ldots , a_n , ^{ \frown}  , \vert \cdot \vert   ) $
is an open problem for $ n \geq  2 $, 
Steven Senger showed in his doctoral dissertation that 
$ (   \mathbb{N} , 0, 1, + , \vert ) $  
 is existentially interpretable in 
$ (   A^*, \varepsilon ,  a_1, \ldots , a_n , ^{ \frown}  , \vert \cdot  \vert  ) $ for $ n \geq 2 $
 (see p.~61 of \cite{Senger1982}). 
This shows  that  $ (   A^* , \varepsilon  , a_1, \ldots , a_n , ^{ \frown}  , \vert \cdot \vert   ) $ is quite expressive. 
 It was proved by   Bel’tyukov  \cite{Beltyukov1980}   and Lipshitz \cite{Lipshitz1978}    that the existential theory of 
 $ (   \mathbb{N} , 0, 1, + , \vert ) $  
is decidable. 
The symbol $ \vert $ denotes the divisibility relation on the set of natural numbers.

\begin{theorem}    \label{SecondExistentialInterpretation}

$ (   A^*, \varepsilon ,  a_1, \ldots , a_n , ^{ \frown}  , \vert \cdot  \vert  ) $ 
is $ \exists $-interpretable in 
$ \mathcal{T} (  \mathcal{L}_{ \mathsf{BT} }   )  $.

\end{theorem}

\begin{proof}

We need to specify  a function $ \tau :  A^*   \to   \mathbf{H}    $ 
and  existential $ \mathcal{L}_{ \mathsf{BT} }    $-formulas  that describe  a structure
with universe   $ \tau (A^* ) $   and  isomorphic to the structure  
$ (   A^*, \varepsilon ,  a_1, \ldots , a_n , ^{ \frown}  , \vert \cdot  \vert  ) $.
We translate the empty string as follows: 
$ \tau ( \varepsilon )  \equiv  \    \langle  \perp ,   \perp^{ 2  }   \rangle  $. 
We translate each letter $ a_i $ of the alphabet $A$  as follows:  
$ \tau ( a_i )  \equiv    \      \langle  \varepsilon  ,     \mathsf{g}_i    \rangle   $ 
where   $ \mathsf{g}_i \equiv    \     \langle  \perp^{ 3+i }  ,   \perp^{ 3+i }   \rangle $. 
We need  the following property  to prove that $ \tau (A^*) $ is existentially definable 
\begin{itemize}

\item[(*)] $   \mathsf{g}_1  , \ldots , \mathsf{g}_n  $ are incomparable with respect to  the subtree relation.

\end{itemize}

To extend $ \tau $ to all of $ A^* $, we need to translate the concatenation operator. 
We translate concatenation as follows
\[
 x  ^{ \frown }  y = z    \equiv  \   z = y [   \,  \tau (\varepsilon )    \,  \mapsto   \,  x  \,    ]  
 \      .
 \]
 We extend $ \tau $ by recursion  to all of $A^*$ by mapping the string $w_1 \ldots w_k w_{k+1 } $ to the finite binary tree 
$ \tau (w_1  \ldots w_k ) ^{ \frown }    \tau ( w_{k+1} ) $.

Assume $ \vert \cdot \vert $ is the tally-length function that replaces each letter with the letter $ a_k$. 
We translate $ \vert \cdot \vert $ as follows 
\[
\vert x \vert = y  \equiv  \ 
y = x 
[  \,   \mathsf{g}_1    \,  \mapsto  \,   \mathsf{g}_k   \,  ]   \;  
[  \,   \mathsf{g}_2    \,  \mapsto  \,   \mathsf{g}_k   \,  ]    \;  
 \ldots 
 [  \,   \mathsf{g}_n    \,  \mapsto  \,    \mathsf{g}_k    \,  ]  
 \         .
\]

All that remains is to show that  $ \tau ( A^* ) $  is existentially definable. 
Lemma \ref{NaturalNumberLikeObjects} tells us that    the classes  
$  \mathbb{N}^{ \varepsilon }_{ \mathsf{g}_i   }   \cup \lbrace \varepsilon  \rbrace    $
are existentially definable   in   $ \mathcal{T} (  \mathcal{L}_{ \mathsf{BT} }   )  $. 
The idea is to show that  $ s \in  \tau ( A^* ) $ if and only if we can transform $s$ into an element  of  
$  \mathbb{N}^{ \varepsilon }_{ \mathsf{g}_i   }   \cup \lbrace \varepsilon  \rbrace       $. 
We show that $ \tau ( A^* ) $ is defined by the following   formula 
\[
 \phi (x)   \equiv   \    
 x [  \,  \mathsf{g}_2   \,  \mapsto \,  \mathsf{g}_1  \,  ]   \ldots 
 [  \,  \mathsf{g}_{n}    \,  \mapsto \,  \mathsf{g}_1  \,  ]  
\in 
 \mathbb{N}^{ \varepsilon }_{ \mathsf{g}_1   }   \cup \lbrace \varepsilon  \rbrace     
\         .
\]

Clearly, each element in $ \tau ( A^* ) $ has the property   $ \phi (x) $. 
To see that the converse holds, assume $ \phi (s) $. 
We need to show that $ s  \in  \tau ( A^* ) $.
Since $  \mathbb{N}^{ \varepsilon }_{ \mathsf{g}_1   }   \cup \lbrace \varepsilon  \rbrace       \subseteq \tau (A^*) $, 
it suffices to show that for each $ 1 \leq i \leq n $  and each finite binary tree $t$ 
\[
t [  \,  \mathsf{g}_i   \,  \mapsto \,  \mathsf{g}_1  \,  ]  \in    \tau (A^*)  
\     \Rightarrow    \   
t \in    \tau (A^*)  
\        .
\tag{**} 
\]
We prove (**) by induction on the size of $t$. 
Assume  $ t [  \,  \mathsf{g}_i   \,  \mapsto \,  \mathsf{g}_1  \,  ]  \in    \tau (A^*)   $. 
We need to show that  $ t \in    \tau (A^*)   $. 
If $  \mathsf{g}_i  $ is not a subtree of $t$, then 
\[
 t  =  t  [  \,  \mathsf{g}_i   \,  \mapsto \,  \mathsf{g}_1  \,  ]  \in    \tau (A^*)   
 \      .
 \]
Assume now $  \mathsf{g}_i  $ is a subtree of $t$. 
Let $ t = \langle t_0, t_1 \rangle $.
We cannot have $ t =  \mathsf{g}_i  $  since $  \mathsf{g}_1  \not\in   \tau (A^* ) $. 
Hence
\[
t  [  \,  \mathsf{g}_i   \,  \mapsto \,  \mathsf{g}_1  \,  ] 
= 
\big\langle 
t_0  [  \,  \mathsf{g}_i   \,  \mapsto \,  \mathsf{g}_1  \,  ]     \,  ,   \,  
t_1  [  \,  \mathsf{g}_i   \,  \mapsto \,  \mathsf{g}_1  \,  ]   
 \big\rangle 
\         .
\]
By how the elements of $ \tau (A^* ) $ are defined
\[
  t_0  [  \,  \mathsf{g}_i   \,  \mapsto \,  \mathsf{g}_1  \,  ]  \in \tau (A^* ) 
  \    \mbox{  and   }     \  
  t_1  [  \,  \mathsf{g}_i   \,  \mapsto \,  \mathsf{g}_1  \,  ]   = \mathsf{g}_j 
  \    \mbox{  for some  }   1 \leq j \leq n 
  \         .
  \]
  Since  $  t_0  [  \,  \mathsf{g}_i   \,  \mapsto \,  \mathsf{g}_1  \,  ]  \in \tau (A^* )  $, 
  by the induction hypothesis, $ t_0 \in \tau (A^*) $. 
If $ \mathsf{g}_i  $ is not a subtree of $ t_1 $, then 
$ t_1 =  t_1  [  \,  \mathsf{g}_i   \,  \mapsto \,  \mathsf{g}_1  \,  ]   = \mathsf{g}_j $. 
Assume now  $ \mathsf{g}_i  $ is  a subtree of $ t_1 $. 
Then,   $ \mathsf{g}_1 $ is a subtree of  $ \mathsf{g}_j $ 
since  $  t_1  [  \,  \mathsf{g}_i   \,  \mapsto \,  \mathsf{g}_1  \,  ]   = \mathsf{g}_j $. 
By (*),  $  \mathsf{g}_1  =  \mathsf{g}_j  $, which implies $ t_1 =   \mathsf{g}_i  $. 
Hence,  $ t_0 \in \tau (A^* ) $ and $ t_1 = \mathsf{g}_l $ for some $ 1 \leq l \leq n $. 
Then,  $ t = \langle t_0 , t_1 \rangle \in \tau (A^*) $ by how the elements of $ \tau (A^*) $  are defined. 
Thus, by induction, (**) holds for all $ 1 \leq i \leq n $ and all finite binary trees $t$.
\end{proof}

We have not been able to determine whether the converse of the preceding theorem holds, 
which would say something about the complexity of deciding truth of existential sentences in  
$ (   A^*, \varepsilon , a_1, \ldots , a_n , ^{ \frown}  , \vert \cdot  \vert  ) $.
Since we show in Section \ref{UndecidabilityExistentialTheorySubstitutionOperator} 
 that the existential theory of  $ \mathcal{T} (  \mathcal{L}_{ \mathsf{BT} }   )  $   is undecidable, 
 a positive solution to this problem would imply undecidability of the existential theory of 
$ (   A^*, \varepsilon ,  a_1, \ldots , a_n , ^{ \frown}  , \vert \cdot  \vert  )   \;   $.

\begin{open problem}    \label{OpenProblemExistentialInterpretabilityFreeSemigroupLengthOperator}

Let $ n \geq 2 $. 
Is   $ \mathcal{T} (  \mathcal{L}_{ \mathsf{BT} }   )  $   $ \exists $-interpretable in     the extended free semigroup
$ (   A^*, \varepsilon ,  a_1, \ldots , a_n , ^{ \frown}  , \vert \cdot  \vert  )          \;       $?

\end{open problem}

\subsection{Second Basic Lemma}
\label{SecondBasicLemma}

In the preceding section, 
we saw how   the classes $ \mathbb{N}^{ \alpha }_{ \vec{s} }   $  can be used  to existentially interpret 
finitely generated free semigroups extended with a tally-length function. 
The classes $ \mathbb{N}^{ \alpha }_{ \vec{s} }  $ were used in finding an existentially definable domain while the substitution operator was used to show that concatenation and the tally-length operator are existentially definable on this domain. 
When we existentially interpret  $ (  \mathbb{N} , 0, 1,  + \times ) $ in the next section, 
 $ \mathbb{N}^{ \alpha }_{ \vec{s} }   $ is used to find an existentially definable domain while the substitution operator is used to show that addition is existentially definable on this domain. 
In this section, we develop the  tools   that   will  allow us  to show that multiplication is existentially definable. 
The classes $ \mathbb{N}^{ \alpha }_{ \vec{s} } $ on their own are not sufficient to  show that multiplication is existentially definable
since elements  of   the classes $ \mathbb{N}^{ \alpha }_{ \vec{s} }  $  have a  simple repetitive  structure. 
We need to show that  classes of finite binary trees with a bit more complex   structure  are existentially definable.

We are interested in describing the relation  $ x \times y = z $,  on the set of natural numbers, 
by   existential  $ \mathcal{L}_{ \mathsf{BT} }    $-formulas.  
Our approach is to characterize $   x \times y = z  $ in terms of the computation tree of $ x \times y $. 
We know that given  three natural numbers $n, k, m > 0 $, 
the equality $ n \times k = m $ holds if and only if there exists a sequence    of pairs of natural numbers 
\[
( k_1 , m_1 )  ,   \   ( k_2, m_2 )  ,   \    \ldots ,   ( k_r , m_r   ) 
  \tag{*}  
\]
 such that 
  \begin{itemize}

 \item[-]    $ k_{ 1 } =  1 $ and $ m_1 =   n  $

 \item[-] given $ 1 \leq i < r $,      if $ k_{i} = j_0    $ and  $ m_{i} =  j_1     $, 
 then   $ k_{i +1 } = j_0 +1    $    and  $ m_{i +1 } = j_1 + n     $

 \item[-]  $ k_r =  k $  and    $       m_r =  m      \;   $.

 \end{itemize}

 We start by characterizing the existence of (*) in terms of finite binary trees.
 For technical reasons which  have to do with the proof of Lemma \ref{MultiplicationLemma},   
 we work with two distinct  representations of natural numbers. 
 We use one representation to encode numbers in the sequence $ k_1, \ldots , k_r $,  
 and we use another representation to encode numbers in the sequence $ m_1, \ldots , m_r $. 
 Let $ \alpha , \beta  $ be finite binary trees that are incomparable with respect to  the subtree relation. 
Let $ s  $ and  $ t  $ 
be such that $ \alpha $ is a substree of  neither $  s $  nor $  t $, 
and $ \beta $ is  a  subtree of  neither $  s $  nor $  t  $.
We have   the following two ways of  associating natural  numbers with finite binary trees
(see Definition \ref{ThirdNotationSystem}): 
\begin{itemize}

\item[(A)]   We map the natural $n$ to the binary tree $ \frac{ n }{ \alpha , s }  $  defined by recursion as follows  
\[
\frac{ 0 }{ \alpha , s }  \equiv   \    \alpha   
   \      \   
\mbox{   and    }  
  \          \
\frac{ n+1  }{ \alpha , s }  \equiv   \   \langle \alpha ,  s  \rangle   \big[    \alpha  \mapsto   \frac{ n }{ \alpha , s }     \big] 
\                  .
\]

\item[(B)]     We map the natural $n$ to the binary tree $ \frac{ n }{ \beta  , t }  $  defined by recursion as follows  
\[
\frac{ 0 }{ \beta , t }  \equiv   \    \beta   
   \      \   
\mbox{   and    }  
  \          \
\frac{ n+1  }{ \beta , t }  \equiv   \   \langle \beta ,  t  \rangle   \big[    \beta  \mapsto   \frac{ n }{ \beta , t }     \big] 
\                  .
\]

\end{itemize}

We use (A) to represent the sequence  $ k_1, \ldots , k_r $, 
and we use (B) to represent the sequence  $ m_1, \ldots , m_r $. 
We  can now  characterize the existence of (*) in terms of finite binary trees. 
Given  three natural numbers $n, k, m > 0 $, 
the equality $ n \times k = m $ holds if and only if there exists a sequence    of  pairs of finite binary trees 
\[
\langle  u _1, v_1  \rangle   ,   \   \langle  u_2, v_2  \rangle   ,   \    \ldots ,   \langle  u_r , v_r   \rangle  
  \tag{**}  
\]
 such that 
  \begin{itemize}

 \item[-]    $ u_{ 1 } =  \frac{ 1 }{ \alpha , s }   $ and $ v_1 = \frac{ n  }{ \beta , t }       $

 \item[-] given $ 1 \leq i < r $,  
 if $ u_{i} =    \frac{ j_0 }{ \alpha , s }       $    and  $ v_{i} = \frac{ j_1 }{ \beta , t }           $, 
 then   $ u_{i +1 } = \frac{ j_0  +1  }{ \alpha , s }        $    and  $ v_{i +1 } =   \frac{ j_1 +n  }{ \beta , t }      $

 \item[-]  $ u_r =   \frac{ k  }{ \alpha , s }   $  and    $       v_r =  \frac{ m  }{ \beta , t }     \;   $.

 \end{itemize}

The next step is to associate (**) with a finite binary tree. 
Let $ \gamma $ be a finite binary tree  that is such that $ \alpha ,  \beta , \gamma $ are incomparable with respect to  the subtree relation. 
Using    the notion of Definition  \ref{ThirdNotationSystemSecondPart}, 
we associate (**) with the finite binary tree 
\[
\big\langle      \;   
  \gamma  \, ,   \,   \langle u_1, v_1 \rangle    \, ,   \,   \langle u_2, v_2 \rangle    \, ,   \,  
\ldots    
 \, ,   \,     \langle u_r, v_r \rangle     
  \     \big\rangle 
 \        .
 \tag{***}
\]
For example, the left tree in Figure \ref{MultiplicationFirstFigure} represents $ 5 \times 3 = 15 $
and the right tree in Figure \ref{MultiplicationFirstFigure}  represents $ 50 \times 3  = 150 $.

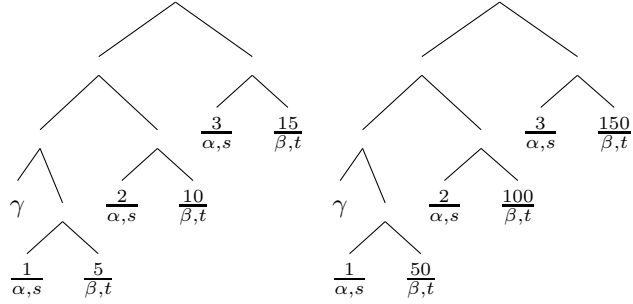
\begin{figure}

\begin{center}

\begin{forest}
[
[
[
[$  \gamma  $ ]
   [  [$  \frac{ 1 }{ \alpha , s }  $ ]     [ $  \frac{ 5 }{ \beta  , t  }      $   ]   ]
]
[   [$  \frac{ 2 }{ \alpha , s }    $ ]     [ $     \frac{ 10 }{ \beta  , t  }       $   ]     ]
]
[   [$  \frac{ 3 }{ \alpha , s }   $ ]     [ $   \frac{ 15 }{ \beta  , t  }          $   ]    ] 
]
\end{forest}
\begin{forest}
[
[
[
[$  \gamma  $ ]
   [  [$     \frac{ 1 }{ \alpha , s }     $ ]     [ $    \frac{ 50 }{ \beta  , t  }    $   ]   ]
]
[   [$      \frac{ 2 }{ \alpha , s }       $ ]     [ $      \frac{ 100 }{ \beta  , t  }     $   ]     ]
]
[   [$     \frac{ 3 }{ \alpha , s }        $ ]     [ $      \frac{ 150 }{ \beta  , t  }    $   ]    ] 
]
\end{forest}

\end{center}
\caption{
Encoding $ 5 \times 3 = 15 $ and  $ 50  \times 3 = 150 $  with finite binary trees.
}
\label{MultiplicationFirstFigure}
\end{figure}

Now that we have a way of associating $ n \times k = m $ with a particular  finite binary tree $ T(n, k, m ) $, 
we need to find an existentially definable class that contains  $ T(n, k, m ) $  and   is such that it is easy to characterize 
$ T(n, k, m )  $ in terms of representations of $n, k, m $. 
Our approach  is the following: 
Given a fixed natural number $ n  >0  $,   let 
$  \mathbb{M}^{ \alpha , \beta , \gamma }_{ s, t, n }   $
denote the class of all  finite binary trees  $S$   of the form (***), 
i.e., $S$ encodes  the computation tree of   $ n \times k  $ for some  $ k > 0 $.
After giving a formal definition of   $  \mathbb{M}^{ \alpha , \beta , \gamma }_{ s, t, n }   $, 
we prove that $   \mathbb{M}^{ \alpha , \beta , \gamma }_{ s, t, n }    $  is existentially definable.
That  $    \mathbb{M}^{ \alpha , \beta , \gamma }_{ s, t, n }  $
  is existentially definable means that we can associate it with some  existential   $ \mathcal{L}_{ \mathsf{BT}   }   $-formula 
$ \psi (x, \alpha , \beta ,  \gamma ,  s, t ,  \frac{ n }{ \beta , t }   ) $, 
where $x$ is the defining variable and  $ \alpha , \beta  ,  \gamma , s, t  ,  \frac{ n }{ \beta , t }  $ are parameters.

\begin{definition}

Let $ \alpha , \beta , \gamma   \in \mathbf{H} $ be  incomparable with respect to  the subtree relation. 
Let $ s , t \in \mathbf{H} $ be such that  $ \alpha $ is  a substree of  neither $  s  $  nor  $  t  $ 
and $ \beta $ is a subtree of neither $  s  $  nor  $  t  $.
Let $ n \geq 1 $.
Let  
 $  \mathbb{M}^{ \alpha , \beta , \gamma }_{ s, t, n }   $ 
denote  the smallest class of finite full binary trees that satisfies  the following  
\begin{itemize}

\item  $ 
 \big\langle   \gamma    \,   ,   \,     \langle  \frac{ 1 }{ \alpha , s }  ,   \frac{ n }{ \beta , t }   \rangle   \,       \big\rangle   
 \in  
 \mathbb{M}^{ \alpha , \beta , \gamma }_{ s, t, n }   $

\item   if  $  T    \in    \mathbb{M}^{ \alpha , \beta , \gamma }_{ s, t, n }  $  
where 
$ T =  \big\langle   R     \,   ,   \,   \langle   \frac{ k }{ \alpha , s }    \,  ,  \,      \frac{ m }{ \beta  , t }   \rangle   \,   \big\rangle  $, 
then 
 $  \big\langle   T    \,   ,   \,   \langle  \frac{ k +1  }{ \alpha , s }    \,  ,  \,      \frac{ m + n  }{ \beta  , t }     \rangle   \,     \big\rangle   
   \in    \mathbb{M}^{ \alpha , \beta , \gamma }_{ s, t, n }    \;   $.

\end{itemize}

\end{definition}

To improve readability, we introduce  the following abbreviation 
\[
t [ r_1 \, \mapsto \,  s_1 ,  r_2 \, \mapsto \,  s_2 ,  \ldots ,  r_k \, \mapsto \,  s_k  ]  
\equiv    \   
t [ r_1 \, \mapsto \,  s_1 ]  [ r_2 \, \mapsto \,  s_2 ]   \ldots [ r_k  \, \mapsto \,  s_k ] 
\     . 
\]

\begin{lemma}    \label{MultiplicationLemma}

Let $ \alpha , \beta , \gamma \in \mathbf{H} $ be  incomparable with respect to  the subtree relation. 
Let $ s , t \in \mathbf{H} $ be such that  $ \alpha $ is  a substree of  neither $  s  $  nor  $  t  $ 
and $ \beta $ is a subtree of neither $  s  $  nor  $  t  $.
Let $ n \geq 1 $. 
Let $T  \in \mathbf{H}$. 
Then, 
$ T \in   \mathbb{M}^{ \alpha , \beta , \gamma }_{ s, t, n }      $ if and only if 
\begin{itemize}

\item[\textup{(1) } ]  $  
   \big\langle   \gamma    \,   ,   \,     \langle  \frac{ 1 }{ \alpha , s }  ,   \frac{ n }{ \beta , t }   \rangle   \,      \big\rangle    
   \sqsubseteq T  $

\item[\textup{(2) } ]   there exist 
$  L  \in    \mathbb{N}^{ \alpha }_{ s }  \cup \lbrace \alpha \rbrace   $ and  
$ R  \in   \mathbb{N}^{ \beta  }_{ t }     \cup \lbrace \beta   \rbrace     $ 
such that 
\[
 T   = 
 \Big\langle   T      \Big[      
  \big\langle   \gamma    \,   ,   \,     \langle  \frac{ 1 }{ \alpha , s }  ,   \frac{ n }{ \beta , t }   \rangle   \,      \big\rangle  
      \,  \mapsto  \,   
       \gamma   
 \   ,      \   
 \frac{ 1 }{ \alpha , s }    \,  \mapsto   \,   \alpha    
 \   ,      \   
\frac{ n }{ \beta , t }     \,  \mapsto   \,   \beta    
\Big]   
 \   ,      \   
\langle L, R   \rangle     \;      \Big\rangle       
\               .
\]

\end{itemize}

\end{lemma}

Before we prove the lemma, we illustrate   with an example how  (1)-(2) work. 
Assume $T$ is the right tree   in Figure \ref{MultiplicationFirstFigure}, 
 which encodes the computation tree of $ 50 \times 3 $. 
So 
\[
T = \Big\langle         \    
 \gamma , \;   
  \big\langle \frac{ 1 }{ \alpha , s}   \, ,   \,   \frac{ 50 }{ \beta , t}     \big\rangle     ,   \;  
    \big\langle \frac{ 2 }{ \alpha , s}   \, ,   \,   \frac{ 100 }{ \beta , t}     \big\rangle     ,   \;  
  \big\langle \frac{ 3  }{ \alpha , s}   \, ,   \,   \frac{  150 }{ \beta , t}     \big\rangle    
\       \Big\rangle 
\                   . 
\]
Since  $ \big\langle \frac{ 1 }{ \alpha , s}   \, ,   \,   \frac{ 50 }{ \beta , t}     \big\rangle   $ occurs only at the bottom of $T$, 
we have 
\[
\begin{array}{r c l }
T_0   
& := & 
T   \Big[          \   
  \big\langle   \gamma    \,   ,   \,     \langle  \frac{ 1 }{ \alpha , s }  ,   \frac{ n }{ \beta , t }   \rangle   \,      \big\rangle  
      \,  \mapsto  \,   
       \gamma   
       \      \Big] 
       \\
       \\
       & = & 
\Big\langle    
 \gamma , \;   
    \big\langle \frac{ 2 }{ \alpha , s}   \, ,   \,   \frac{ 100 }{ \beta , t}     \big\rangle     ,   \;  
  \big\langle \frac{ 3  }{ \alpha , s}   \, ,   \,   \frac{  150 }{ \beta , t}     \big\rangle    
\       \Big\rangle        
 \                   .
\end{array}
\]
The binary tree  $  \frac{ 1 }{ \alpha , s }  $ occurs in $T_0 $ only at  the bottom  of 
$  \frac{ 2 }{ \alpha , s }  $ and $  \frac{ 3 }{ \alpha , s }  $.
Hence 
\[
\begin{array}{r c l }
T_1   
& := & 
T_0     \Big[       \    \frac{ 1 }{ \alpha , s }    \,  \mapsto   \,   \alpha       \     \Big] 
       \\
       \\
       & = & 
\Big\langle    
 \gamma , \;   
    \big\langle \frac{ 1 }{ \alpha , s}   \, ,   \,   \frac{ 100 }{ \beta , t}     \big\rangle     ,   \;  
  \big\langle \frac{ 2  }{ \alpha , s}   \, ,   \,   \frac{  150 }{ \beta , t}     \big\rangle    
\       \Big\rangle        
 \                   .
\end{array}
\]
The binary tree  $  \frac{ 50 }{ \beta , t }  $ occurs in $T_1 $ only at the bottom of 
$  \frac{ 100 }{ \beta , s }  $ and $  \frac{ 150  }{ \beta , t }  $.
Hence 
\[
\begin{array}{r c l }
T_2   
& := & 
T_1      \Big[       \   \frac{ 50  }{ \beta , t }     \,  \mapsto   \,   \beta        \     \Big] 
       \\
       \\
       & = & 
\Big\langle    
 \gamma , \;   
    \big\langle \frac{ 1 }{ \alpha , s}   \, ,   \,   \frac{ 50 }{ \beta , t}     \big\rangle     ,   \;  
  \big\langle \frac{ 2  }{ \alpha , s}   \, ,   \,   \frac{  100 }{ \beta , t}     \big\rangle   
\       \Big\rangle        
 \                   .
\end{array}
\]

\begin{proof}[Proof of Lemma \ref{MultiplicationLemma}]

The only if part is obvious. 
We prove  the if part   by induction on the size of $T$. 
We need the following properties: 
\begin{itemize}

\item[(A)]   Since  $ \alpha , \beta , \gamma $ are incomparable with respect to  the subtree relation,      the binary tree 
$    \big\langle   \gamma    \,   ,   \,     \langle  \frac{ 1 }{ \alpha , s }  ,   \frac{ n }{ \beta , t }   \rangle   \,      \big\rangle    $ 
is not a subtree of   elements  of 
$ \mathbb{N}^{ \alpha }_{ s }      \cup  \mathbb{N}^{ \beta }_{ t }      \cup \lbrace   \alpha ,   \beta \rbrace   $.

\item[(B)]    Since    $ \alpha , \beta  $ are incomparable with respect to  the subtree relation  and 
$ \alpha $ is  not a substree of   $  t  $,  the binary tree 
$ \frac{ 1 }{ \alpha , s }  $ is not a subtree of  elements  of   $  \mathbb{N}^{ \beta }_{ t }     \cup \lbrace \beta \rbrace   $.

\item[(C)]     Since     $ \alpha , \beta  $ are incomparable with respect to  the subtree relation  and  
 $ \beta $ is  not a substree of   $  s  $,  the binary tree 
$ \frac{ n }{ \beta , t  } $ is not a subtree of  elements  of  
$  \mathbb{N}^{ \alpha }_{ s }   \cup \lbrace \alpha \rbrace  $.

\end{itemize}

Assume $T$ satisfies (1)-(2). 
We need to show that   $ T \in   \mathbb{M}^{ \alpha , \beta , \gamma }_{ s, t, n }     $. 
Recall that  (see Definition   \ref{ThirdNotationSystem})
\begin{multline*}
 \frac{ 0 }{ \alpha , s }    =        \alpha 
 \    \wedge     \   
 \forall     L  \in    \mathbb{N}^{ \alpha }_{ s }     \;    \exists   k  \in    \mathbb{N}  \;    [     \   
   k \geq 1   \;   \wedge    \;     L =  \frac{ k }{ \alpha , s }       \           ] 
 \       \   
 \mbox{  and    }   
\\
  \frac{ 0 }{ \beta  , t }   =    \beta        
   \    \wedge     \   
 \forall     R  \in    \mathbb{N}^{ \beta }_{ t }     \;    \exists  m \in   \mathbb{N}  \;    [     \   
   m \geq 1   \;   \wedge    \;     R =  \frac{ m }{ \beta  , t }       \              ] 
\           .
\end{multline*}
Hence, since  $T$ satisfies (1)-(2),   there exist   natural numbers  $ k, m  \geq 0 $ such that 
\[
\begin{array}{r l}
\textup{ (i) }  
 &   
  \big\langle   \gamma    \,   ,   \,     \langle  \frac{ 1 }{ \alpha , s }  ,   \frac{ n }{ \beta , t }   \rangle   \,      \big\rangle    
   \sqsubseteq T
   \\
   \\
 \textup{ (ii) }
   & 
   T = 
    \Big\langle   T      \Big[      
  \big\langle   \gamma    \,   ,   \,     \langle  \frac{ 1 }{ \alpha , s }  ,   \frac{ n }{ \beta , t }   \rangle   \,      \big\rangle  
      \,  \mapsto  \,   
       \gamma   
 \   ,      \   
 \frac{ 1 }{ \alpha , s }    \,  \mapsto   \,   \alpha    
 \   ,      \   
\frac{ n }{ \beta , t }     \,  \mapsto   \,   \beta    
\Big]   
 \   ,      \   
\big\langle  \frac{ k }{ \alpha , s }   \,  ,   \,    \frac{ m }{ \beta , t }       \big\rangle     \;      \Big\rangle    
\                              .
\end{array}
\]

Let 
\[
S
= 
T      \Big[      
  \big\langle   \gamma    \,   ,   \,     \langle  \frac{ 1 }{ \alpha , s }  ,   \frac{ n }{ \beta , t }   \rangle   \,      \big\rangle  
      \,  \mapsto  \,   
       \gamma   
 \   ,      \   
 \frac{ 1 }{ \alpha , s }    \,  \mapsto   \,   \alpha    
 \   ,      \   
\frac{ n }{ \beta , t }     \,  \mapsto   \,   \beta    
\Big]   
\        .
\]
Assume $S = \gamma $. 
By  (i), 
$     \big\langle   \gamma    \,   ,   \,     \langle  \frac{ 1 }{ \alpha , s }  ,   \frac{ n }{ \beta , t }   \rangle   \,      \big\rangle      $
 is a subtree of $T$.
 By (A),  
 $     \big\langle   \gamma    \,   ,   \,     \langle  \frac{ 1 }{ \alpha , s }  ,   \frac{ n }{ \beta , t }   \rangle   \,      \big\rangle      $
 is a subtree of neither   $  \frac{ k }{ \alpha , s }   $  nor    $ \frac{ m }{ \beta , t }     $. 
Hence 
\[
T  =   \big\langle   \gamma    \,   ,   \,   
  \langle  \frac{ 1 }{ \alpha , s }  ,   \frac{ n }{ \beta , t }   \rangle   \,      \big\rangle   
\in  
 \mathbb{M}^{ \alpha , \beta , \gamma }_{ s, t, n }     
\          .
\]

Assume now $ S \neq    \gamma $. 
 Since $T$ satisfies (i), 
$     \big\langle   \gamma    \,   ,   \,     \langle  \frac{ 1 }{ \alpha , s }  ,   \frac{ n }{ \beta , t }   \rangle   \,      \big\rangle      $
 is a subtree of $T$. 
By (A), 
 $     \big\langle   \gamma    \,   ,   \,     \langle  \frac{ 1 }{ \alpha , s }  ,   \frac{ n }{ \beta , t }   \rangle   \,      \big\rangle      $
 is a subtree of neither   $  \frac{ k }{ \alpha , s }   $  nor    $ \frac{ m }{ \beta , t }     $. 
 Hence, by  (ii)  and the definition of $S$ 
\[
  \big\langle   \gamma    \,   ,   \,     \langle  \frac{ 1 }{ \alpha , s }  ,   \frac{ n }{ \beta , t }   \rangle   \,      \big\rangle    
\sqsubseteq S 
\               .
\tag{iii}
\]

 By (ii)  and  (A)-(C) 
\[
\begin{array}{r c l}
S
&  =   &  
T  \Big[      
  \big\langle   \gamma    \,   ,   \,     \langle  \frac{ 1 }{ \alpha , s }  ,   \frac{ n }{ \beta , t }   \rangle   \,      \big\rangle  
      \,  \mapsto  \,   
       \gamma   
 \   ,      \   
 \frac{ 1 }{ \alpha , s }    \,  \mapsto   \,   \alpha    
 \   ,      \   
\frac{ n }{ \beta , t }     \,  \mapsto   \,   \beta    
\Big]   
\\
\\  
&=&  
\Big\langle     
S   \Big[      
  \big\langle   \gamma    \,   ,   \,     \langle  \frac{ 1 }{ \alpha , s }  ,   \frac{ n }{ \beta , t }   \rangle   \,      \big\rangle  
      \,  \mapsto  \,   
       \gamma   
 \   ,      \   
 \frac{ 1 }{ \alpha , s }    \,  \mapsto   \,   \alpha    
 \   ,      \   
\frac{ n }{ \beta , t }     \,  \mapsto   \,   \beta    
\Big]   
\  ,    \    
 \big\langle   \frac{ k  \,  \dot{-} \,    1  }{ \alpha , s }    \,   ,   \,      \frac{ m \,  \dot{-} \,  n   }{ \beta , t }       \big\rangle  
   \;       \Big\rangle 
   \              .
\end{array}
\tag{iv} 
\]
where 
 \[
 x     \,  \dot{-} \,    y 
 =
  \begin{cases}
 x     &    \mbox{ if  } x  < y 
 \\
 x - y     &   \mbox{ otherwise.} 
 \end{cases}
 \]
By (iii)-(iv),  $S$ satisfies (1)-(2). 
Hence,  by the induction hypothesis,    $ S \in   \mathbb{M}^{ \alpha , \beta , \gamma }_{ s, t, n }     $. 
It then follows from (iv) and (ii)  that   $ T \in   \mathbb{M}^{ \alpha , \beta , \gamma }_{ s, t, n }     $.

Thus, by induction, $ T \in   \mathbb{M}^{ \alpha , \beta , \gamma }_{ s, t, n }     $  if $ T  $ satisfies (1)-(2).
\end{proof}

In the proof of Lemma \ref{MultiplicationLemma},  
we did not use the assumption that   $ \alpha $ is  not a substree of   $  s  $  and  $ \beta $ is  not a substree of   $  t  $. 
With these assumptions,    Lemma   \ref{NaturalNumberLikeObjects}   tells us that  
 $  \mathbb{N}^{ \alpha }_{ s }  \cup \lbrace \alpha \rbrace  $  and 
  $ \mathbb{N}^{ \beta  }_{ t }     \cup \lbrace \beta   \rbrace     $   
  are existentially definable  in    $ \mathcal{T} ( \mathcal{L}_{ \mathsf{BT} }  )  $.
 Thus,    Lemma \ref{MultiplicationLemma}    shows that   
 $   \mathbb{M}^{ \alpha , \beta , \gamma }_{ s, t, n }      $   
is  existentially definable  in    $ \mathcal{T} ( \mathcal{L}_{ \mathsf{BT} }  )  $.

\subsection{Arithmetic  with the Substitution Operator} 
\label{UndecidabilityExistentialTheorySubstitutionOperator}

We are finally ready to give an existential interpretation of 
 $ ( \mathbb{N} , 0,1, +, \times ) $ in  $ \mathcal{T} ( \mathcal{L}_{ \mathsf{BT} }  )  $.
Lemma  \ref{NaturalNumberLikeObjects} will allow us to specify an existentially definable domain 
while Lemma \ref{MultiplicationLemma} will allow us to give an existential definition of multiplication on the chosen domain. 
Addition will be handled very easily using the substitution operator.

\begin{theorem}     \label{ExistentialInterpretationOfFirstOrderArithmetic}

$ ( \mathbb{N} , 0,1, +, \times ) $ is  $ \exists $-interpretable in $ \mathcal{T} ( \mathcal{L}_{ \mathsf{BT} }  )  $.
Hence, $ \mathsf{Th}^{ \exists }   (  \mathcal{T} ( \mathcal{L}_{ \mathsf{BT} }  )  ) $ is undecidable. 

\end{theorem}

\begin{proof}

Let  
\[
 0 \equiv  \   \langle  \perp ,   \perp^{ 2 }  \rangle 
,   \    \   
 s  \equiv  \   \perp^{ 5 }    
 ,   \    \   
 t = s 
  \    \   
 \mbox{  and   } 
 \     \  
  n \equiv   \    \frac{ n }{ 0  ,  s    }  
  \   \   \mbox{ for each }   n  \geq 1 
 \    .
 \]

Lemma \ref{NaturalNumberLikeObjects} tells us that the class 
$ \mathsf{N} :=   \mathbb{N}^{ 0 }_{ s  }    \cup \lbrace 0 \rbrace $ is existentially definable in    
$ \mathcal{T} (  \mathcal{L}_{ \mathsf{BT} }   )  $. 
We translate addition   on $ \mathsf{N}   $  as follows 
\[
x + y = z \equiv  \  
 \mathsf{N}   (x)   \;  \wedge  \;    \mathsf{N}    (y)    \;  \wedge  \;    \mathsf{N}    (z)    \;  \wedge  \;  
z = y [  \,   0  \,  \mapsto  \,  x   \,   ]   
\      . 
\]

We use Lemma \ref{MultiplicationLemma}  to give a translation of multiplication on   $ \mathsf{N}   $. 
Let 
\[
\gamma   \equiv   \      \langle  \perp ,   \perp^{ 3 }  \rangle 
,    \        \   
\alpha \equiv   \      \langle  \perp ,   \perp^{ 4 }  \rangle 
,    \        \   
\beta \equiv   \      \langle  \perp ,   \perp^{ 5 }  \rangle 
\        .
\]
Let 
\begin{multline*}
x \times y = z   \equiv   \  
 \mathsf{N}   (x)   \;  \wedge  \;    \mathsf{N}    (y)    \;  \wedge  \;    \mathsf{N}    (z)    \;  \wedge  \;  
 \Big(     \    
 (   \   x = 0 \;   \wedge   \;   z = 0   \   )  
 \;    \vee   \;  
  \\
  (   \   y = 0 \;   \wedge   \;   z = 0   \   )  
 \;    \vee   \;  
  (       \   x  \neq 0 \;   \wedge   \;  y  \neq 0    \;   \wedge   \;   \Phi (x, y , z )     \       )  
  \       \Big)   
\end{multline*}
where 
\begin{multline*}
 \Phi (x, y , z )      \equiv   \  
 \exists    n ,   v ,   w    \;      
 \Big[     \   
   n   = x [  0 \,  \mapsto  \,   \beta   ]
     \;    \wedge   \;  
\mathbb{M}_{ s, t  , n   }^{ \alpha , \beta ,  \gamma   }   (w) 
     \;    \wedge   \;  
          \\
w = 
 \Big\langle  v      \,    ,    \,   
   \big\langle   y  [  0 \,  \mapsto  \,   \alpha   ]     \,   ,   \,     z    [  0 \,  \mapsto  \,   \beta   ]    \big\rangle 
     \,  
    \Big\rangle
\         \Big] 
\            .   
\end{multline*}

It follows from the definition of  $\mathbb{M}_{ s, t , n   }^{ \alpha , \beta ,  \gamma   }   $  that 
$  \Phi (x, y , z )    $   captures  correctly    $ x  \times y = z $ when $ x, y > 0 $. 
\end{proof}

Since computably enumerable sets of natural numbers are Diophantine, 
we have  the following corollary.

\begin{corollary}

$ ( \mathbb{N} , 0,1, +, \times ) $ and $ \mathcal{T} ( \mathcal{L}_{ \mathsf{BT} }  )  $
are mutually  $ \exists $-interpretable.

\end{corollary}

\subsection{Analogue of Hilbert`s 10th Problem} 
\label{AnalogueOfHilbers10Problem}

In this section,   we show that  the analogue of Hilbert`s 10th Problem for  
$ \mathcal{T} ( \mathcal{L}_{ \mathsf{BT} }  ) $ is undecidable. 
That is, we show that there cannot exist an algorithm that takes as input a  $ \mathcal{L}_{ \mathsf{BT} }  $-sentence of the form 
$  \exists \vec{x} [  \;  s =  t  \;  ]  $ 
and decides whether   $  \exists \vec{x} [  \;  s =  t  \;  ]  $   is true in  $ \mathcal{T} ( \mathcal{L}_{ \mathsf{BT} }  ) $.

\begin{theorem}     

The fragment  $ \mathsf{Th}^{ \mathsf{H10} }   (  \mathcal{T} ( \mathcal{L}_{ \mathsf{BT} }  )  ) $ is undecidable. 

\end{theorem}

\begin{proof}

Since  $ \mathsf{Th}^{ \exists  }   (  \mathcal{T} ( \mathcal{L}_{ \mathsf{BT} }  )  ) $ is undecidable, 
it suffices to show that  given an existential $ \mathcal{L}_{ \mathsf{BT} }  $-sentence  $ \phi $, 
we can compute a finite number   of $ \mathcal{L}_{ \mathsf{BT} }  $-sentences   $ \phi_1,  \ldots , \phi_n $   of the form 
$ \exists  \vec{x}  \;   [   \  s = t   \   ]   $
such that 
\[
  \mathcal{T} ( \mathcal{L}_{ \mathsf{BT} }  )    \models \phi    \leftrightarrow    \bigvee_{ i = 1 }^{ n }  \phi_i 
\        .
\]
Since 
\[
  \mathcal{T} ( \mathcal{L}_{ \mathsf{BT} }  )    \models  
   (      \        s_1 = t_1      \;      \wedge     \;   s_2 = t_2      \        )    \leftrightarrow    
   \langle  s_1,  s_2   \rangle  =  \langle   t_1, t_2  \rangle  
\]
it suffices to show that given a  $ \mathcal{L}_{ \mathsf{BT} }  $-formula of the form   $ s \neq  t $, 
we can compute a finite number of atomic  $ \mathcal{L}_{ \mathsf{BT} }  $-formulas $ s_1 = t_1,  \ldots , s_k = t_k $ 
such that 
\[
  \mathcal{T} ( \mathcal{L}_{ \mathsf{BT} }  )    \models  s  \neq   t     \leftrightarrow    \bigvee_{ j = 1 }^{ k }   s_j = t_j
\        .
\]
This is the case since 
\[
\begin{array}{r c l }
s  \neq t   
& \Leftrightarrow   &  
s  \not\sqsubseteq t   \  \vee    \   t   \not\sqsubseteq    s  
\\
\\
& \Leftrightarrow   &  
t [  \,   s   \, \mapsto    \,  \langle s, s  \rangle  \,  ]   =   t    
\    \vee    \  
s  [  \,   t   \, \mapsto    \,  \langle t, t  \rangle  \,  ]   =   s   
\           .              \qedhere 
\end{array}
\]

\end{proof}

\section{Undecidable Fragments IV}   
\label{UndecidableFragmentsPartIV}

In this section, 
we give a direct proof  of  undecidability of   $ \mathsf{Th}^{ \exists }   (  \mathcal{T} ( \mathcal{L}_{ \mathsf{BT} }  )  ) $
by constructing a many-to-one reduction of Post`s Correspondence Problem
(see  Section \ref{UndecidableFragmentsPartISubtreeRelation} for the  definition  of PCP). 
It is not clear to us whether this result can be used to give a new proof of unsolvability of Hilbert`s 10th Problem. 
In particular,   it is not clear to us whether it is possible to construct an existential interpretation of 
$   \mathcal{T} ( \mathcal{L}_{ \mathsf{BT} }  ) $ in   $ ( \mathbb{N} , 0,1, +, \times ) $
without using the exponential function to code sequences.

\begin{open problem}

Construct an   existential interpretation of   the structure 
$   \mathcal{T} ( \mathcal{L}_{ \mathsf{BT} }  ) $ in   $ ( \mathbb{N} , 0,1, +, \times ) $
which does not rely on  the solution to Hilbert`s 10th Problem. 

\end{open problem}

Given an instance  $\langle a_{1}, b_{1} \rangle,\ldots , \langle a_{n}, b_{n}  \rangle$ of PCP, 
we need to compute an existential   $ \mathcal{L}_{ \mathsf{BT} }  $-sentence $ \phi $ that is true in 
$ \mathcal{T} ( \mathcal{L}_{ \mathsf{BT} }  ) $ if and only if   
$\langle a_{1}, b_{1} \rangle,\ldots , \langle a_{n}, b_{n}  \rangle$ 
has a solution. 
Recall that  $\langle a_{1}, b_{1} \rangle,   \ldots    , \langle a_{n}, b_{n}  \rangle$  has a solution if and only if 
there exists  a finite nonempty sequence $ i_{1},   \ldots  ,i_{m} $ of indexes such that 
\[
a_{i_{1}} a_{i_{2}}\ldots  a_{i_{m}} = b_{i_{1}}  b_{i_{2}}   \ldots  b_{i_{m}} 
\  .
\tag{*} 
\]
The methods  we developed in Section   \ref{UndecidableFragmentsPartIII}  are not sufficient to encode (*). 
To highlight  the problems we need to solve, 
observe that  $\langle a_{1}, b_{1} \rangle,\ldots , \langle a_{n}, b_{n}  \rangle$ 
has a solution if and only if there exist   two  sequences 
$ u_1,   \ldots , u_k      $   and   $  v_1,    \ldots , v_m   $
such that 
\begin{itemize}

\item[(I)]    there exists  $ f_1  \in \lbrace 1,  \ldots , n \rbrace $ such that    $u_1 = a_{ f_1 }  $   and 
for all $ j \in \lbrace 1,  \ldots , k-1  \rbrace $ there exists  $ f_{ j+1 }    \in \lbrace 1,  \ldots , n \rbrace $ such that 
$ u_{ j+1 } = u_j a_{   f_{ j+1 }   } $

\item[(II)]    there exists  $ g_1  \in \lbrace 1,  \ldots , n \rbrace $ such that    $v_1 = b_{ g_1 }  $   and 
for all $ j \in \lbrace 1,  \ldots , m-1  \rbrace $ there exists  $ g_{ j+1 }    \in \lbrace 1,  \ldots , n \rbrace $ such that 
$ v_{ j+1 } = v_j b_{   g_{ j+1 }   } $

\item[(III)]  $ k =m $ and $ f_j = g_j $ for all   $j \in \lbrace 1, \ldots , k \rbrace $

\item[(IV)]  $ u_k = v_m $.

\end{itemize}
The sequences we encoded in Section  \ref{UndecidableFragmentsPartIII} 
 were simpler in the sense that they  were of the form $ m_1, m_2,  \ldots , m_k $ where there is a unique and simple transition rule that tells us how to obtain $m_{ i+1 } $ from $ m_{ i} $. 
This is not the case with $ u_1,  u_2,   \ldots , u_k    $ and   $  v_1, v_2,   \ldots , v_m  $. 
The tools we need to express that  
$ u_1,  u_2,   \ldots , u_k    $ and    $  v_1, v_2,   \ldots , v_m  $ 
satisfy  (I) and (II), respectively,   are developed   in Section   \ref{ThirdBasicLemma}. 
The tools we need to express that  the equalities in (III) hold  are developed in Section   \ref{FourthBasicLemma}.
Finally, in Section  \ref{ManyToOneReductionOfPCPExistentialFragment}, 
we put everything together and specify a many-to-one reduction of PCP to 
$ \mathsf{Th}^{ \exists }   (  \mathcal{T} ( \mathcal{L}_{ \mathsf{BT} }  )  ) $.

\subsection{Third Basic Lemma}
\label{ThirdBasicLemma}

In this section, we develop the tools we need to encode (I)-(II). 
So, we are given   a finite sequence 
$ C = \langle  c_1,  c_2,   \ldots ,   c_n  \rangle $ of nonempty binary strings, 
and we   need to express  that a sequence  $ w_1, w_2,  \ldots , w_k $ of binary strings  satisfies  
\begin{itemize}

\item[(A)] there exists $ i \in \lbrace 1,  \ldots , n \rbrace $ such that $ w_1 = c_i $

\item[(B)] for all $ j  \in \lbrace 1,  \ldots ,  k-1 \rbrace   $ there exists $ i \in \lbrace 1,  \ldots , n \rbrace $
such that $ w_{ j+1 } = w_{ j }   c_i $.

\end{itemize}
In other words, we need to  give  an existential definition of the class  $ \mathbb{P} (C) $ 
of all  sequences $ w_1, w_2,  \ldots , w_k $  that satisfy  (A)-(B). 
Our proof of existential definability of   $ \mathbb{P} (C) $  is an  extension of the  method  we used  in 
Section  \ref{SecondBasicLemma} 
to show that multiplication is existentially  definable (see Lemma \ref{MultiplicationLemma}). 
A careful  inspection  of Lemma \ref{MultiplicationLemma} shows that we use the classes  $ \mathbb{N}^{ \alpha }_{ s } $ 
in an essential way when we  prove  that  $   \mathbb{M}^{ \alpha , \beta , \gamma }_{ s, t, n }    $  is existentially definable. 
The role played by a class   $ \mathbb{N}^{ \alpha }_{ s } $    will now be played by a class  $ A^* $  of all finite strings over the  finite alphabet $A$. 
Just as with  $ \mathbb{N}^{ \alpha }_{ s } $,     we need to know that $ A^* $ is existentially definable. 
But this follows from  the existential interpretation of  
$ (   A^*, \varepsilon , a_1, \ldots , a_n , ^{ \frown}  , \vert \cdot  \vert  ) $  
we gave in  Section \ref{ExistentialInterpretationFreeSemigroupsLengthOperator}.
For technical reasons which have to do with the proof of Lemma \ref{ThirdBasicLemmaMainResult}, 
we need to modify slightly the map $ \tau :  A^* \to \mathbf{H} $   we gave in   
Section \ref{ExistentialInterpretationFreeSemigroupsLengthOperator}.

\begin{definition}  \label{ThirdBasicLemmaFirstDefinition}

Let $ A = \lbrace a_1, \ldots , a_m \rbrace $ be a finite alphabet. 
For each natural number $ i   \geq 1 $, 
let   $ \mathsf{g}_i \equiv    \     \langle  \perp^{ 3+i }  ,   \perp^{ 3+i }   \rangle $. 
Let $ \alpha \in \mathbf{H} $ be incomparable with   $ \mathsf{g}_i  $ with respect to $ \sqsubseteq $  for all $i $. 
 We define a one-to-one map    $ \tau_{ \alpha }  :  A^*    \to   \mathbf{H}    $   by recursion 
 \[
 \tau_{ \alpha  }   (  w ) =   \begin{cases}
 \alpha                                                                                          &    \mbox{ if } w = \varepsilon 
 \\
 \langle  \alpha    ,     \mathsf{g}_i    \rangle                               &    \mbox{ if } w = a_i 
 \\
  \tau_{ \alpha }   ( w_0 )   \big[      \;     \alpha    \,  \mapsto   \,       \tau_{ \alpha }  ( w_{ 1 } )     \;    \big]  
                    &    \mbox{ if } w  = w_0 w_1    \  \mbox{  and  }    \      w_0 \in A           \        . 
 \end{cases}
 \]
Given $ s \in A^* $, we write $ \frac{ s }{ \alpha } $ for $ \tau_{ \alpha } ( s ) $. 
 Furthermore, we write $ a_i $ for   $ \mathsf{g}_i    $.

\end{definition}

Observe that this definition differs slightly from the definition  we gave in the proof of 
 Theorem \ref{SecondExistentialInterpretation}.
 In the proof of   Theorem \ref{SecondExistentialInterpretation}, 
 if a binary tree $T$ represents the string $s$, then reading $T$ bottom-up corresponds to reading $s$ from left to right. 
 In the definition we have just given, 
 if a binary tree $T$ represents the string $s$, then reading $T$ bottom-up corresponds to reading $s$ from right  to left. 
 Nonetheless, it is easy to see that 
 the proof  of  Theorem \ref{SecondExistentialInterpretation}  shows that   $ \tau_{ \alpha }  (  A^*  ) $
is existentially definable in   $ \mathcal{T} ( \mathcal{L}_{ \mathsf{BT} }  )  $.

\begin{lemma}

Let $ A $ be a finite alphabet. 
Then, $ \tau_{ \alpha }  (  A^*  ) $ is existentially definable in  $ \mathcal{T} ( \mathcal{L}_{ \mathsf{BT} }  )  $. 

\end{lemma}

Since we are interested in describing sequences that satisfy (A)-(B), 
it  is  not  the set $\lbrace 0, 1 \rbrace^* $ we are interested in,  but rather the subset  generated by 
$ \lbrace c_1, c_2,  \ldots , c_n \rbrace $ under concatenation. 
We also need to treat the  $ c_i$`s as distinct objects since  we intend to replace   $C$  with    one of the sequences 
$ \langle  a_1, \ldots , a_n  \rangle  $,   $  \langle  b_1,  \ldots , b_n  \rangle  $ where 
$ \langle a_1, b_1 \rangle ,   \ldots  , \langle a_n , b_n \rangle $
is an instance of PCP.
To capture this, we  associate   elements of $ \lbrace c_1, c_2,  \ldots , c_n \rbrace^+ $
with  strings over a larger alphabet 
$ \lbrace 0,  1,   \mu_1 ,  \mu_2,    \ldots , \mu_n    \rbrace $
where $ \mu_i $ represents the last  letter of $ c_i  $. 
Assume for example $ c_1 = 110 $, $ c_2 = 011 $ and $ c_3 = 1010 $. 
Then, we associate the binary  string $ c_2 c_1 c_3 $ with the string 
$ 01 \mu_2  11  \mu_1    101   \mu_3 $.

\begin{definition}  \label{ThirdBasicLemmaFirstDefinitionPartII}

Let $ C  =  \langle  c_1,  c_2,   \ldots ,   c_n  \rangle $ be a sequence of nonempty binary strings. 
We associate   $ c_i $ with a finite binary tree in
 $ \tau_{ \alpha }  (   \lbrace 0, 1, \mu_1,  \ldots , \mu_n \rbrace^*  )  $ 
as follows 
\[
\frac{ c_i }{   C,  \alpha }   \equiv   \   
\frac{ w_i   \mu_i }{  \alpha }           
\     \   
\mbox{ where }      
\       \   
  c_i = w_i  d   \;   \wedge  \;   w_i   \in   \lbrace 0, 1   \rbrace^*  \;   \wedge  \;   d  \in   \lbrace 0, 1   \rbrace     \         . 
\]
We associate    $  c_{ i_1 }  c_{ i_2 }   \ldots c_{ i_m }  $   with a finite binary tree in
 $ \tau_{ \alpha }  (   \lbrace 0, 1, \mu_1,  \ldots , \mu_n \rbrace^* )  $ 
as follows 
\[
\frac{ c_{ i_1 }  c_{ i_2 }   \ldots c_{ i_m }   }{   C,  \alpha }   \equiv   \ 
\frac{ w_{i_1 }   \mu_{ i_1 }   w_{ i_2 }   \mu_{ i_2 }     \ldots    w_{ i_m }   \mu_{ i_m }     }{  \alpha }
\                           .
\]
We let $ \frac{ \varepsilon }{ C, \alpha }  \equiv   \   \alpha  $.

\end{definition}

See Figure   \ref{ThirdBasicLemmaFirstFigure}
for a visualization of  
$ \frac{ c_1 }{ C, \alpha } $,  $   \    \frac{ c_2 }{ C, \alpha } $, $   \    \frac{ c_1 c_2 }{ C, \alpha } $, $   \    \frac{ c_2 c_1 }{ C, \alpha } $
when   $  c_1 =  110$ and  $     c_2 =   011 $.

\begin{figure}

\begin{center}

\begin{forest}
[
[
[
[$  \alpha    $ ]
   [   $  \mu_1   $    ]
]
[  $  1 $     ]
]
[   $  1  $   ] 
]
\end{forest}
\begin{forest}
[
[
[
[$  \alpha    $ ]
   [   $  \mu_2   $    ]
]
[  $  1 $     ]
]
[   $  0   $   ] 
]
\end{forest}
\begin{forest}
[
         [
         [
         [
         [
[
[$  \alpha    $ ]
   [   $  \mu_2   $    ]
]
[  $  1 $     ]
]
[   $  0   $   ] 
        ]
        [   $  \mu_1  $  ] 
        ]
          [   $  1  $  ] 
          ]
           [   $  1 $  ] 
          ]
]
\end{forest}
\begin{forest}
[
         [
         [
         [
         [
[
[$  \alpha    $ ]
   [   $  \mu_1   $    ]
]
[  $  1 $     ]
]
[   $  1   $   ] 
        ]
        [   $  \mu_2  $  ] 
        ]
          [   $  1  $  ] 
          ]
           [   $  0 $  ] 
          ]
]
\end{forest}

\end{center}
\caption{
Visualization of  
$ \frac{ c_1 }{ C, \alpha } $,  $   \    \frac{ c_2 }{ C, \alpha } $, $   \    \frac{ c_1 c_2 }{ C, \alpha } $, $   \    \frac{ c_2 c_1 }{ C, \alpha } $, 
respectively, 
when   $  c_1 =  110$,  $   \  c_2 =   011 $.
 }
\label{ThirdBasicLemmaFirstFigure}
\end{figure}
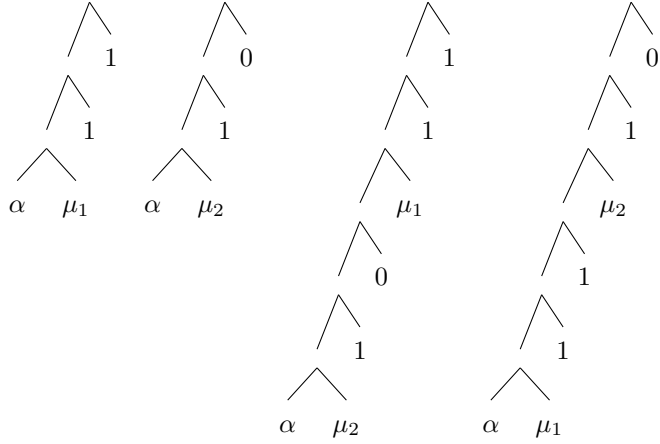

We are finally ready to give a formal definition of the class of those finite binary trees  that encode  sequences that satisfy (A)-(B).

\begin{definition}   \label{ThirdBasicLemmaSecondDefinition}

Let $ C =  \langle  c_1,  c_2,   \ldots ,   c_n  \rangle $ be a sequence of nonempty binary strings. 
Let $ \alpha ,  \gamma \in \mathbf{H} $ be incomparable  with  respect to the subtree relation. 
Assume $ \alpha $ also satisfies the condition in Definition \ref{ThirdBasicLemmaFirstDefinition}.
Let $ \mathbb{P} (C,  \alpha ,  \gamma  ) $ be the smallest subset of $ \mathbf{H} $ that satisfies 
\begin{itemize}

\item   $   \langle   \gamma  \,   ,   \,     \frac{ c_i }{ C, \alpha }       \rangle    \in  \mathbb{P} (C,  \alpha , \gamma )  $ 
for all $ i \in \lbrace 1, \ldots , n \rbrace $

\item   if $  T  \in   \mathbb{P} (C, \alpha ,  \gamma )  $  where   
$ T =  \big \langle   R     \,   ,   \,     \frac{ c_{ i_1 }  c_{ i_2 }   \ldots c_{ i_m }   }{   C,  \alpha }       \big\rangle     $,  
then  
\[
 \big\langle   \,  T  \,  ,   \,      \frac{ c_{ i_1 }  c_{ i_2 }   \ldots c_{ i_m }   c_j   }{   C,  \alpha }       \,      \big\rangle   
\in  \mathbb{P} (C ,  \alpha ,  \gamma  )  
\]
 for all $ j  \in \lbrace 1, \ldots , n \rbrace $. 

\end{itemize}

\end{definition}

See Figure \ref{ThirdBasicLemmaSecondFigure} for a visualization of  the form of  elements of  
$ \mathbb{P} (C , \alpha ,  \gamma  ) $.

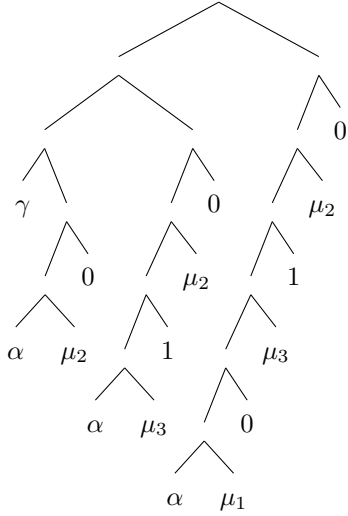
\begin{figure}

\begin{center}

\begin{forest}
[
[
[
[$  \gamma  $ ]
   [         [    [ $   \alpha   $  ]    [    $  \mu_2 $    ]  ]     [    $   0 $  ]    ] 
]
        [ 
        [   [         [    [ $  \alpha   $  ]    [    $  \mu_3 $    ]  ]     [    $   1 $  ]    ] 
        [  $   \mu_2  $  ] 
        ]
        [  $  0  $   ]  
        ]
]
      [
       [
        [ 
        [   [         [    [ $  \alpha   $  ]    [    $  \mu_1 $    ]  ]     [    $   0 $  ]    ] 
        [  $   \mu_3  $  ] 
        ]
        [  $  1  $   ]  
        ]
        [ $  \mu_2  $  ] 
        ]
        [  $  0  $   ]
        ]
]
\end{forest}

\end{center}
\caption{
Example  of  an element    of    $ \mathbb{P} (C , \alpha , \gamma ) $   when   $ c_1 = 01$, $ c_2  = 00 $, $ c_3 = 10 $. 
}
\label{ThirdBasicLemmaSecondFigure}
\end{figure}

\begin{lemma}   \label{ThirdBasicLemmaMainResult}

Let $ C =  \langle  c_1,  c_2,   \ldots ,   c_n  \rangle $ be a sequence of nonempty binary strings. 
Let $ \alpha ,  \gamma   \in \mathbf{H} $ be incomparable with  respect to the subtree relation. 
Assume $ \alpha $ also satisfies the condition in Definition \ref{ThirdBasicLemmaFirstDefinition}.
Let    $ \delta =   \langle  \alpha ,  \alpha \rangle  $. 
Let 
$
F_{ \delta }^{ \alpha } (L ) =  L [   \;   \alpha   \,   \mapsto  \,   \delta   \;   ]  
$
for all $  L    \in \mathbf{H} $.
Let $ T \in  \mathbf{H} $. 
Then, $ T \in    \mathbb{P} (C  ,  \alpha ,  \gamma )  $ if and only if 
\begin{itemize}

\item[\textup{(1) }  ]  $  \delta  \not\sqsubseteq T $

\item[\textup{(2) } ]   there exists   $ m  \in \lbrace 1, \ldots , n \rbrace $  such that 
 $   \big\langle    \gamma    \,   ,   \,    \frac{ c_m }{ C,  \alpha }     \big\rangle    \sqsubseteq T  $

\item[\textup{(3) } ]   there exists    
$  S   \in    \tau_{ \alpha }     (   \lbrace 0, 1 , \mu_1,  \ldots , \mu_n \rbrace^*  )     $ 
such that    
\[
 T   = 
 \Big\langle   
 F^{ \alpha }_{ \delta }   (T)     \Big[      
   \big\langle    \gamma    \,   ,   \,       \frac{ c_m  }{ C,  \delta  }       \big\rangle    
    \,  \mapsto  \,         \gamma   
 \   ,      \   
\frac{ c_1 }{ C,  \delta  }      \,  \mapsto   \,      \alpha  
 \   ,      \   
   \ldots 
  \   ,      \   
\frac{ c_n  }{ C,  \delta  }        \,  \mapsto   \,   \alpha  
\Big]      
 \   ,      \   
S    \;      \Big\rangle       
\               .
\]

\end{itemize}

\end{lemma}

Before we prove the lemma,   we illustrate why the left-right implication holds. 
First, observe that (1) holds    if  $ T \in \mathbb{P} ( C, \alpha , \gamma ) $ 
(see  Figure    \ref{ThirdBasicLemmaSecondFigure}).
Now, assume  for example $T$ is the tree in Figure    \ref{ThirdBasicLemmaSecondFigure}. 
So,     $ m = 2 $ and 
\[
T  = \Big\langle   \;  
\gamma   \,  ,  \,    \frac{ c_2   }{ C,  \alpha   }        \,  ,  \,   
\frac{ c_2 c_3    }{ C,  \alpha   }         \,  ,  \,   
\frac{  c_2  c_3  c_1     }{ C,  \alpha   } 
  \;     \Big\rangle
  \            .
\]
The tree $  F^{ \alpha }_{ \delta }  (T) $ is just the tree we obtain by replacing each one of the tree occurrences of $ \alpha $ in $T$ with $ \delta $. 
Hence  
\[
 F^{ \alpha }_{ \delta }   (T) = \Big\langle   \;  
\gamma   \,  ,  \,    \frac{ c_2   }{ C,  \delta  }        \,  ,  \,   
\frac{ c_2 c_3    }{ C,  \delta  }         \,  ,  \,   
\frac{  c_2  c_3  c_1     }{ C,  \delta  } 
  \;     \Big\rangle
  \            .
\]
Since there is only one occurrence of  $  \big\langle    \gamma    \,   ,   \,    \frac{ c_2 }{ C,  \delta  }     \big\rangle  $    in 
$    F^{ \alpha }_{ \delta }   (T) $, 
we have 
\[
\begin{array}{r c l }
R_0   
& := & 
 F^{ \alpha }_{ \delta }    (T)    \big[          \   
   \ \big\langle    \gamma    \,   ,   \,       \frac{ c_2  }{ C,  \delta  }       \big\rangle    
    \,  \mapsto  \,         \gamma   
       \      \big] 
       \\
       \\
       & = & 
\big\langle   \;  
\gamma       \,  ,  \,   
\frac{ c_2 c_3    }{ C,  \delta  }         \,  ,  \,   
\frac{  c_2  c_3  c_1     }{ C,  \delta  }    
  \;     \big\rangle
 \                   .
\end{array}
\]
We replace the one occurrence of  $ \frac{ c_1 }{ C,  \delta  }     $ in $ R_0 $  and obtain 
\[
\begin{array}{r c l }
R_1  
& := & 
R_0     \big[          \   
\frac{ c_1 }{ C,  \delta  }       \,  \mapsto  \,        \alpha     
       \      \big] 
       \\
       \\
       & = & 
\big\langle   \;  
\gamma       \,  ,  \,   
\frac{ c_2 c_3    }{ C,  \delta  }         \,  ,  \,   
\frac{  c_2  c_3      }{ C,  \alpha   }    
  \;     \big\rangle
 \                   .
\end{array}
\]
Since  $\frac{  c_2  c_3      }{ C,  \alpha   }       $ does not contain  a subtree   of the form $\frac{  c_i      }{ C,  \delta   }       $
by the choice of $ \delta $, 
 there is no occurrence of $ \frac{  c_2      }{ C,  \delta   }     $ in $ R_1 $. 
 Hence 
\[
R_2 := R_1    \big[          \   
\frac{  c_2      }{ C,  \delta   }       \,  \mapsto  \,      \alpha    
       \      \big] 
       = R_1 
       \                     .
\]
We replace the occurrence of $ \frac{  c_3  }{ C,  \delta   }     $ in $ R_2 $ and obtain  
\[
\begin{array}{r c l }
R_3 
& := & 
R_2     \big[          \   
\frac{  c_3  }{ C,  \delta   }      \,  \mapsto  \,        \alpha    
       \      \big] 
       \\
       \\
       & = & 
\big\langle   \;  
\gamma       \,  ,  \,   
\frac{ c_2     }{ C,  \alpha  }         \,  ,  \,   
\frac{  c_2  c_3      }{ C,  \alpha   }         
  \;     \big\rangle
 \                   .
\end{array}
\]
Now, observe  that  $   R_3 $ is  the left  subtree of $T $    (see Figure  \ref{ThirdBasicLemmaSecondFigure}).

\begin{proof}[Proof of Lemma \ref{ThirdBasicLemmaMainResult}]

The left-right implication   is obvious. 
We prove  right-left  implication    by induction on the size of $T$. 
We need the following properties: 
\begin{itemize}

\item[(A)]   Since  $ \gamma $ and $ \alpha   $ are incomparable with respect to  the subtree relation,   the binary tree 
 $\big\langle    \gamma    \,   ,   \,    \frac{ c_m }{ C,  \alpha }     \big\rangle     $
 is not a subtree of   elements  of 
$   \tau_{ \alpha }  (   \lbrace 0, 1 , \mu_1,  \ldots , \mu_n \rbrace^*  )       $.

\item[(B)]   Since  $ \gamma $ and $ \delta   $ are incomparable with respect to  the subtree relation,   the binary tree 
 $\big\langle    \gamma    \,   ,   \,    \frac{ c_m }{ C,  \alpha }     \big\rangle     $
 is not a subtree of   elements  of 
$   \tau_{ \delta }  (   \lbrace 0, 1 , \mu_1,  \ldots , \mu_n \rbrace^*  )       $.

\end{itemize}

Assume $T$ satisfies (1)-(3). 
We need to show that   $ T \in   \mathbb{P}  (  C,  \alpha ,  \gamma  )     $. 
By assumption, we have a natural number  $ m \in \lbrace 1,  \ldots , n \rbrace $  and 
a  string   $   s \in  \lbrace 0, 1 , \mu_1,  \ldots , \mu_n \rbrace^*   $
such that 
\[
\begin{array}{r l}
\textup{ (i) }  
 &   
\delta     \not\sqsubseteq    T 
   \\
   \\
\textup{ (ii) }  
 &   
\big\langle    \gamma    \,   ,   \,    \frac{ c_m }{ C,  \alpha }     \big\rangle       \sqsubseteq T 
   \\
   \\
 \textup{ (iii) }
   & 
   T = 
  \Big\langle   
  F^{ \alpha }_{ \delta }   (T)     \Big[      
   \big\langle    \gamma    \,   ,   \,       \frac{ c_m  }{ C,  \delta  }       \big\rangle    
    \,  \mapsto  \,         \gamma   
 \   ,      \   
\frac{ c_1 }{ C,  \delta  }      \,  \mapsto   \,      \alpha  
 \   ,      \   
   \ldots 
  \   ,      \   
\frac{ c_n  }{ C,  \delta  }        \,  \mapsto   \,   \alpha  
\Big]      
 \   ,      \   
  \frac{ s   }{     \alpha }        \;      \Big\rangle       
\                              .
\end{array}
\]
Let 
\[
T_0 
= 
   F^{ \alpha }_{ \delta }   (T)     \Big[      
   \big\langle    \gamma    \,   ,   \,       \frac{ c_m  }{ C,  \delta  }       \big\rangle    
    \,  \mapsto  \,         \gamma   
 \   ,      \   
\frac{ c_1 }{ C,  \delta  }      \,  \mapsto   \,      \alpha  
 \   ,      \   
   \ldots 
  \   ,      \   
\frac{ c_n  }{ C,  \delta  }        \,  \mapsto   \,   \alpha  
\Big]     
\             .
\]

Assume $ T_0  = \gamma $. 
By (ii), $ \big\langle    \gamma    \,   ,   \,    \frac{ c_m }{ C,  \alpha }     \big\rangle       \sqsubseteq T  $. 
By (A), $  \big\langle    \gamma    \,   ,   \,    \frac{ c_m }{ C,  \alpha }     \big\rangle     
\not\sqsubseteq   \frac{ s   }{    \alpha }    $. 
Hence 
\[
T = \big\langle T_0        \,  ,  \,        \frac{  s   }{     \alpha }         \big\rangle 
= 
   \big\langle    \gamma    \,   ,   \,    \frac{ c_m }{ C,  \alpha }     \big\rangle  
\in   \mathbb{P}  (  C,   \alpha ,  \gamma  )  
\              .
\]

Assume now $ T_0  \neq    \gamma $. 
Since $ T_0 \sqsubseteq T $, it follows from (i) that 
\[
\delta      \not\sqsubseteq    T_0 
\              .
\tag{iv}
\]
Since
$  \big\langle    \gamma    \,   ,   \,    \frac{ c_m }{ C,  \alpha }     \big\rangle    \sqsubseteq T $, 
  $  \   T  \neq  \big\langle    \gamma    \,   ,   \,    \frac{ c_m }{ C,  \alpha }     \big\rangle     $
and 
 $  \big\langle    \gamma    \,   ,   \,    \frac{ c_m }{ C,  \alpha }     \big\rangle     
\not\sqsubseteq   \frac{ s  }{    \alpha }     $, 
we have 
\[
 \big\langle    \gamma    \,   ,   \,    \frac{ c_m }{ C,  \alpha }     \big\rangle       \sqsubseteq T_0 
\              .
\tag{v}
\]

Finally,    we have 
\[
\begin{array}{r c l}
T_0 
&  =   &  
  F^{ \alpha }_{ \delta }   (T)     \Big[      
   \big\langle    \gamma    \,   ,   \,       \frac{ c_m  }{ C,  \delta  }       \big\rangle    
    \,  \mapsto  \,         \gamma   
 \   ,      \   
\frac{ c_1 }{ C,  \delta  }      \,  \mapsto   \,      \alpha  
 \   ,      \   
   \ldots 
  \   ,      \   
\frac{ c_n  }{ C,  \delta  }        \,  \mapsto   \,   \alpha  
\Big]     
\\
\\
&=& 
 F^{ \alpha }_{ \delta }    \big(  \big\langle T_0 \, , \,    \frac{ s   }{     \alpha }     \big\rangle    \big)     \Big[      
   \big\langle    \gamma    \,   ,   \,       \frac{ c_m  }{ C,  \delta  }       \big\rangle    
    \,  \mapsto  \,         \gamma   
 \   ,      \   
\frac{ c_1 }{ C,  \delta  }      \,  \mapsto   \,      \alpha  
 \   ,      \   
   \ldots 
  \   ,      \   
\frac{ c_n  }{ C,  \delta  }        \,  \mapsto   \,   \alpha  
\Big]     
\\
\\
&=& 
  \big\langle    F^{ \alpha }_{ \delta }  ( T_0  )  \, , \,    \frac{ s  }{     \delta  }     \big\rangle      
     \Big[      
   \big\langle    \gamma    \,   ,   \,       \frac{ c_m  }{ C,  \delta  }       \big\rangle    
    \,  \mapsto  \,         \gamma   
 \   ,      \   
\frac{ c_1 }{ C,  \delta  }      \,  \mapsto   \,      \alpha  
 \   ,      \   
   \ldots 
  \   ,      \   
\frac{ c_n  }{ C,  \delta  }        \,  \mapsto   \,   \alpha  
\Big]     
\\
\\
&=& 
 \Big\langle    F^{ \alpha }_{ \delta }   ( T_0 ) 
  \Big[      
   \big\langle    \gamma    \,   ,   \,       \frac{ c_m  }{ C,  \delta  }       \big\rangle    
    \,  \mapsto  \,         \gamma   
 \   ,      \   
\frac{ c_1 }{ C,  \delta  }      \,  \mapsto   \,      \alpha  
 \   ,      \   
   \ldots 
  \   ,      \   
\frac{ c_n  }{ C,  \delta  }        \,  \mapsto   \,   \alpha  
\Big]     
  \;  , \;      S_0 
  \;    \Big\rangle     
\end{array}
\tag{vi} 
\]
where 
\[
\begin{array}{r c l  l}
S_0 
&  =   &  
  \frac{ s   }{    \delta  }    \Big[      
   \big\langle    \gamma    \,   ,   \,       \frac{ c_m  }{ C,  \delta  }       \big\rangle    
    \,  \mapsto  \,         \gamma   
 \   ,      \   
\frac{ c_1 }{ C,  \delta  }      \,  \mapsto   \,      \alpha  
 \   ,      \   
   \ldots 
  \   ,      \   
\frac{ c_n  }{ C,  \delta  }        \,  \mapsto   \,   \alpha  
\Big]     
\\
\\
&  =   &  
  \frac{ s   }{    \delta  }    \Big[      
\frac{ c_1 }{ C,  \delta  }      \,  \mapsto   \,      \alpha  
 \   ,      \   
   \ldots 
  \   ,      \   
\frac{ c_n  }{ C,  \delta  }        \,  \mapsto   \,   \alpha  
\Big]     
&   (  \mbox{by  }  (B)  \;     )  
\\
\\
&  =   &  
  \frac{ s^{ \prime }  s^{ \prime \prime  }    }{    \delta  }    \Big[      
\frac{ c_k    }{ C,  \delta  }      \,  \mapsto   \,      \alpha  
\Big]      
\\
\\
&=& 
\frac{ s^{ \prime }   }{    \alpha   }     
\in 
\tau_{ \alpha }     (   \lbrace 0, 1 , \mu_1,  \ldots , \mu_n \rbrace^*  )   
\end{array}
\]
where  we have used that     
\[
 s = s^{ \prime }  s^{ \prime \prime  }   
 \        \  
 \mbox{  and  } 
 \      \
 \frac{ s^{ \prime \prime  } }{ \delta   } =   \frac{ c_k  }{  C, \delta   }    
 \tag{*}
 \]
 for some $ k \in \lbrace 1, \ldots , n \rbrace  $ 
since 
$   \delta   \not\sqsubseteq  T     $      by (1)    and 
\[
\frac{  t  }{  \delta  }
\Big[      
\frac{ c_1 }{ C,  \delta  }      \,  \mapsto   \,      \alpha  
 \   ,      \   
   \ldots 
  \   ,      \   
\frac{ c_n  }{ C,  \delta  }        \,  \mapsto   \,   \alpha  
\Big]     
= \frac{  t  }{  \delta  }   
\]
if $t $ is not of the form  (*).
By (iv)-(vi),     $ T_0 $ satisfies  (1)-(3). 
Hence, by the induction hypothesis, $ T_0  \in \mathbb{P} ( C , \gamma ) $. 
It then follows from (iii),   (vi)  and (*)   that $ T   \in \mathbb{P} ( C , \gamma ) $.

Thus, by induction, $ T \in   \mathbb{P}  (  C,    \alpha  ,   \gamma  )     $    if $ T  $ satisfies (1)-(3).
\end{proof}

\subsection{Fourth Basic Lemma}
\label{FourthBasicLemma}

In  this section, we develop the tools we need to encode (III).
Let $ C =  \langle  c_1,  c_2,   \ldots ,   c_n  \rangle $. 
Each element  $ T \in  \mathbb{P} ( C,   \alpha ,  \gamma ) $ represents 
a sequence of the form $ w_1, w_2, \ldots , w_m $ where 
$  w_k   = c_{i_1} c_{i_2 }  \ldots c_{ i_k }   $ and  $ i_j \in \lbrace 1,   \ldots , n \rbrace $
for all $ j \in \lbrace 1, \ldots ,  m \rbrace $. 
We need the sequence $ i_{ 1 } , \,   i_{ 2 } ,   \,    \ldots  ,  \,  i_{ m } $ to verify that (III) holds. 
We need an existential $ \mathcal{L}_{ \mathsf{BT}  }  $-formula that extracts this information from $T$. 
To achieve this, we need to encode sequences that are more complex than those we encountered in Section  \ref{ThirdBasicLemma}.

The class $ \mathbb{P} ( C, \alpha , \gamma ) $
consists of finite binary trees  that encode sequences  of the form 
 $ w_1, w_2,  \ldots , w_ k $ where 
$ w_i   \in   \tau_{ \alpha }  (   \lbrace 0, 1, \mu_1,  \ldots , \mu_n \rbrace^* )  $  for all $ i \in \lbrace 1,  \ldots  , k \rbrace $. 
We  need   to consider  the class of those binary trees that encode sequences of the form 
$ W_1, W_2,  \ldots , W_k $ where 
$ W_i \in   \mathbb{P} ( C, \alpha , \gamma )   $ for   all $ i \in \lbrace 1,  \ldots  , k \rbrace $. 
To illustrate how this helps us identify  the sequence  $ i_{ 1 } , \,   i_{ 2 } ,   \,    \ldots  ,  \,  i_{ m } $, 
let $T$ be the binary tree in Figure   \ref{ThirdBasicLemmaSecondFigure}. 
We need to find an existential  $ \mathcal{L}_{ \mathsf{BT}  }  $-formula  $ \Psi (T, X  ) $ 
that is true in  $ \mathcal{T} (   \mathcal{L}_{ \mathsf{BT}  }   )  $ if and only if  $X$ represents the string  
$ \mu_2   \mu_3    \mu_1 $. 
Instead of working with $T$, 
we work with the  binary   tree   $ \Gamma_n^{ \alpha }   (T)   $   in Figure   \ref{FourthBasicLemmaFirstFigure}. 
It   contains   the information   $ \mu_2,   \,  \mu_3  ,  \,  \mu_1 $   and has the advantage of having a simpler structure.
We give a formal definition of the operator 
 $ \Gamma_n^{ \alpha }   :  \mathbf{H}  \to \mathbf{H}   $  that takes $T$ and gives us   $ \Gamma_n^{ \alpha }   (T)  $. 
It is really the restriction of $ \Gamma _n^{ \alpha }  $ to $ \mathbb{P} ( C , \alpha , \gamma ) $ we are interested in.
It will follow from the definition that $ \Gamma_n^{ \alpha }    $ is existentially definable.

\begin{figure}

\begin{center}

\begin{forest}
[
[
[
[$  \gamma  $ ]
   [         [    [ $   \mu_2   $  ]    [    $ 0  $    ]  ]     [    $   0 $  ]    ] 
]
        [ 
        [   [         [    [ $  \mu_3    $  ]    [    $  0  $    ]  ]     [    $   0  $  ]    ] 
        [  $   0  $  ] 
        ]
        [  $  0  $   ]  
        ]
]
      [
       [
        [ 
        [   [         [    [ $  \mu_1    $  ]    [    $  0   $    ]  ]     [    $   0 $  ]    ] 
        [  $   0   $  ] 
        ]
        [  $  0  $   ]  
        ]
        [ $  0  $  ] 
        ]
        [  $  0  $   ]
        ]
]
\end{forest}

\end{center}
\caption{
Visualization of $  \Gamma_n^{ \alpha }   (T)  $  when $T$ is the binary tree in Figure   \ref{ThirdBasicLemmaSecondFigure}.
}
\label{FourthBasicLemmaFirstFigure}
\end{figure}
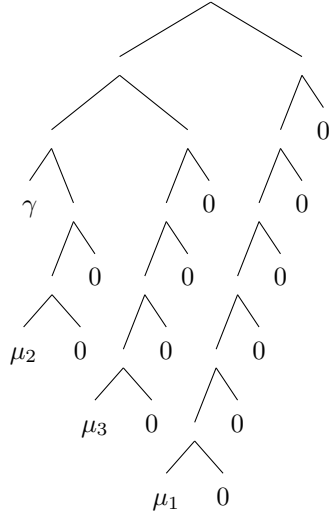

\begin{definition}   \label{FourthBasicLemmaFirstDefinition}

Let $ \alpha ,   0, 1,  \mu_1,  \ldots, \mu_n $ be as in Definition \ref{ThirdBasicLemmaFirstDefinitionPartII}. 
Let $ \mu_{ n+1 } ,  \ldots ,  \mu_{ 2n }  $ be distinct fresh letters. 
Definition  \ref{ThirdBasicLemmaFirstDefinition} associates each letter in the alphabet 
 $ \lbrace     0, 1,  \mu_1,   \mu_2  \ldots, \mu_{ 2n }   \rbrace  $ 
 with a finite binary tree. 
Let $ \Gamma_n^{ \alpha }    :    \mathbf{H}  \to   \mathbf{H} $ be the function defined by $ \Gamma_n^{ \alpha }   (T) = T_2  $ 
where 
\[
\begin{array}{r c l }
T_0  
&=&
  T  \Big[  \;   
 \frac{ \mu_1 }{ \alpha }        \,  \mapsto  \,      \frac{ \mu_1 }{   \mu_{ n+1  }   }     
\  ,    \   \ldots  \  ,    \     
 \frac{ \mu_n  }{ \alpha }        \,  \mapsto  \,     \frac{ \mu_n }{   \mu_{ n+n  }   }       
\;      \Big]  
\\
\\
T_1  
&=&
 T_0      \Big[  \;   1    \,  \mapsto  \,    0     
\   ,     \   
\mu_1   \,  \mapsto   \,   0     
\   ,     \   
\mu_2   \,  \mapsto   \,   0  
\  ,    \   \ldots  \  ,    \     
\mu_n   \,  \mapsto   \,   0   
\;      \Big]  
\\
\\
T_2   
&=&
  T_1     \Big[      \;      
\mu_{ n+1 }    \,  \mapsto  \,    \mu_{ 1  }       
\   ,     \   
\mu_{ n+2 }    \,  \mapsto   \,     \mu_2 
\  ,    \   \ldots  \  ,    \     
\mu_{ n+n }    \,  \mapsto   \,      \mu_n  
\;      \Big]  
\                 .
\end{array}
\]

\end{definition}

We continue  to use the binary tree $T$ in Figure   \ref{FourthBasicLemmaFirstFigure} for illustration. 
Recall  that we are interested in specifying an existential    $   \mathcal{L}_{ \mathsf{BT}  }    $-formula 
$ \Psi (  T, X ) $  that is true if and only if  $ X $ encodes the string   $ \mu_2   \mu_3    \mu_1 $. 
As we have just seen,  $ \Gamma^{ \alpha }_{ n }  (T) $ contains also the information 
$ \mu_2,   \,  \mu_3  ,  \,  \mu_1 $. 
So,  we let  $ \Psi (T, X ) $ be a formula of the form $ \exists W  \;   \Phi ( T , X, W )  $
where $W$ is a finite binary tree that encodes a sequence $ W_1,   W_2  , \ldots , W_k $ 
where $ W_1 = \Gamma^{ \alpha }_{ n }  (  T )  $ and $ W_k = X   $.
Before we give a formal definition of the class  $  \mathbb{P}_2  (C,  \alpha ,  \gamma  )  $
of all $W$ with this property, 
we  use  the binary tree $T$ in Figure   \ref{FourthBasicLemmaFirstFigure}    to illustrate the form of   $W$. 
Let $ W_1, \ldots , W_7 $ be the binary trees in Figure    \ref{FourthBasicLemmaSecondFigure}.
Then,  $W$  can for example be   the binary tree 
\[
 \Big\langle  \;    \alpha    \,  ,   \,  
 W_7  \,  ,   \,   W_6 \,  ,   \,   W_5 \,  ,   \,   W_4 \,  ,   \,   W_3  \,  ,   \,   W_2  \,  ,   \,     W_1 
 \;       \Big\rangle  
\]
or  the binary tree 
\[
 \Big\langle  \;    \alpha    \,  ,   \,  
 W_7  \,  ,   \,   W_7  \,  ,   \,   W_6 \,  ,   \,   W_5 \,  ,   \,   W_4 \,  ,   \,   W_3  \,  ,   \,   W_2  \,  ,   \,     W_1 
 \;       \Big\rangle  
\            .
\]
It is not a problem  that there are many  choices  for $ W $. 
What is important is that $ \Gamma^{ \alpha }_{ n }  (T) $ is the unique right subtree of $ W$, 
and   $ W_7 $ encodes the information we need in a simple format 
 and  is the unique subtree  $X$ of $W$ which is such that     $ \langle \alpha \, , \,  X \rangle \sqsubseteq W $.

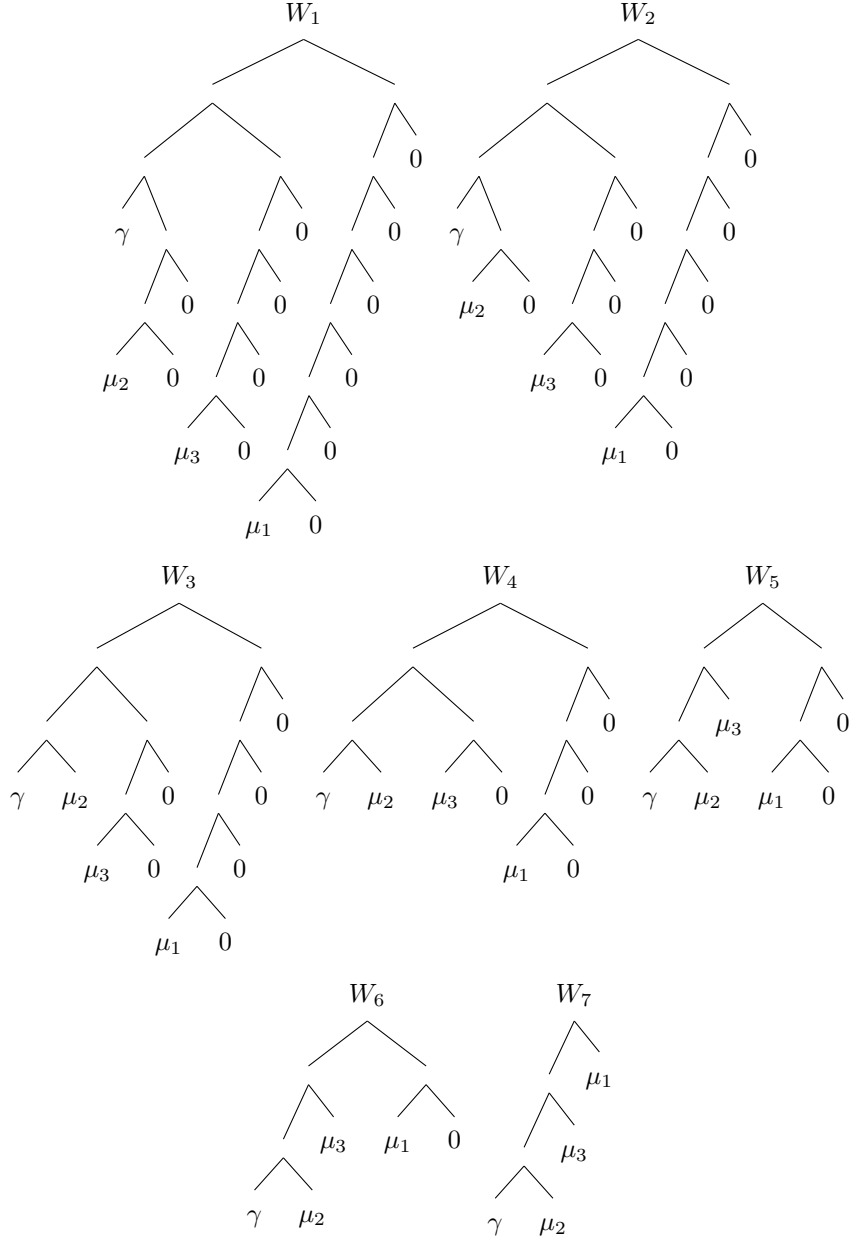
\begin{figure}

\begin{center}

\begin{forest}
[  $   W_1   $ 
[
[
[$  \gamma  $ ]
   [         [    [ $   \mu_2   $  ]    [    $ 0  $    ]  ]     [    $   0 $  ]    ] 
]
        [ 
        [   [         [    [ $  \mu_3    $  ]    [    $  0  $    ]  ]     [    $   0  $  ]    ] 
        [  $   0  $  ] 
        ]
        [  $  0  $   ]  
        ]
]
      [
       [
        [ 
        [   [         [    [ $  \mu_1    $  ]    [    $  0   $    ]  ]     [    $   0 $  ]    ] 
        [  $   0   $  ] 
        ]
        [  $  0  $   ]  
        ]
        [ $  0  $  ] 
        ]
        [  $  0  $   ]
        ]
]
\end{forest}
\begin{forest}
[  $ W_2 $ 
[
[
[$  \gamma  $ ]
   [         [ $   \mu_2   $  ]        [    $   0 $  ]    ] 
]
        [ 
        [   [          [ $  \mu_3    $  ]        [    $   0  $  ]    ] 
        [  $   0  $  ] 
        ]
        [  $  0  $   ]  
        ]
]
      [
       [
        [ 
        [   [             [ $  \mu_1    $  ]         [    $   0 $  ]    ] 
        [  $   0   $  ] 
        ]
        [  $  0  $   ]  
        ]
        [ $  0  $  ] 
        ]
        [  $  0  $   ]
        ]
]
\end{forest}
\begin{forest}
[   $W_3 $ 
[
[
[$  \gamma  $ ]
            [ $   \mu_2   $  ]       
]
        [ 
        [             [ $  \mu_3    $  ]     
        [  $   0  $  ] 
        ]
        [  $  0  $   ]  
        ]
]
      [
       [
        [ 
        [               [ $  \mu_1    $  ]        
        [  $   0   $  ] 
        ]
        [  $  0  $   ]  
        ]
        [ $  0  $  ] 
        ]
        [  $  0  $   ]
        ]
]
\end{forest}
\begin{forest}
[   $W_4 $ 
[
[
[$  \gamma  $ ]
            [ $   \mu_2   $  ]       
]
        [ 
         [ $  \mu_3    $  ]     
        [  $  0  $   ]  
        ]
]
      [
       [
        [ 
          [ $  \mu_1    $  ]        
        [  $  0  $   ]  
        ]
        [ $  0  $  ] 
        ]
        [  $  0  $   ]
        ]
]
\end{forest}
\begin{forest}
[   $W_5 $ 
[
[
[$  \gamma  $ ]
            [ $   \mu_2   $  ]       
]
         [ $  \mu_3    $  ]     
]
      [
       [
          [ $  \mu_1    $  ]        
        [ $  0  $  ] 
        ]
        [  $  0  $   ]
        ]
]
\end{forest}
\begin{forest}
[   $W_6 $ 
[
[
[$  \gamma  $ ]
            [ $   \mu_2   $  ]       
]
         [ $  \mu_3    $  ]     
]
      [
          [ $  \mu_1    $  ]        
        [  $  0  $   ]
        ]
]
\end{forest}
\begin{forest}
[   $W_7 $ 
[
[
[$  \gamma  $ ]
            [ $   \mu_2   $  ]       
]
         [ $  \mu_3    $  ]     
]
          [ $  \mu_1    $  ]        
]
\end{forest}
\end{center}
\caption{
Let  $T$ be the binary tree in Figure   \protect\ref{ThirdBasicLemmaSecondFigure}. 
Then, $ W_1 =   \Gamma^{ \alpha }_{ n }  (T)   $. 
Binary trees of the form 
$ \      W   = \Big\langle  \;    \alpha   \,  ,   \,  
 W_7  \,  ,   \,    \ldots     \,  ,   \,  W_7   \,  ,   \,      W_6  \,  ,   \,   W_5 \,  ,   \,   W_4 \,  ,   \,   W_3 \,  ,   \,   W_2  \,  ,   \,   W_1 
 \;       \Big\rangle    $   
  are  elements  of   $  \mathbb{P}_2  (C,  \alpha ,  \gamma  )   $.
}
\label{FourthBasicLemmaSecondFigure}
\end{figure}

\begin{definition}   \label{FourthBasicLemmaSecondDefinition}

Let $ C =  \langle  c_1,  c_2,   \ldots ,   c_n  \rangle $ be a sequence of nonempty binary strings. 
Let $ \alpha ,  \gamma \in \mathbf{H} $ be incomparable  with  respect to the subtree relation. 
Assume $ \alpha   $  satisfies  the condition in Definition \ref{ThirdBasicLemmaFirstDefinition}.
Assume $ \gamma $ is not a subtree of $ \mu_i $ for all $ i \in \lbrace 1, \ldots , n  \rbrace $. 
Let $ W \in  \mathbb{P}_2  (C,  \alpha ,  \gamma  ) $  if and only if 
there exists a sequence $ W_1, W_2,  \ldots , W_k  \in \mathbf{H} $ such that 
\begin{itemize}

\item  $  W =  \big\langle \;  \alpha   \,  ,   \,    W_k , W_{ k-1}     \,  ,   \,   \ldots    \,  ,   \,     W_1   \;   \big\rangle   $

\item     $ W_k  \in  \tau_{ \gamma } (   \lbrace \mu_1, \ldots , \mu_n \rbrace^{ + }   )   $

\item     for all  $ i \in \lbrace 1,  2, \ldots , k-1 \rbrace $
\[
W_{ i +1   }  = 
W_i    \Big[        \      
    \frac{ 0 }{ \mu_1 }    \,  \mapsto  \,         \mu_1     
 \   ,      \   
     \frac{ 0 }{ \mu_2 }    \,  \mapsto  \,         \mu_2     
 \   ,      \   
\ldots  
  \   ,      \   
    \frac{ 0 }{ \mu_n  }    \,  \mapsto  \,         \mu_n     
    \      \Big]      
\]

\item  there exists $ T \in    \mathbb{P}   (C  ,  \alpha ,  \gamma )   $ such that 
$ W_1 = \Gamma^{ \alpha }_{ n }    (T )  $.

\end{itemize}

\end{definition}

We  prove that $ \mathbb{P}_2   (C  ,  \alpha ,  \gamma )   $   is existentially definable.

\begin{lemma}      \label{ForthBasicLemmaMainResult}

Let $ C =  \langle  c_1,  c_2,   \ldots ,   c_n  \rangle $ be a sequence of nonempty binary strings. 
Let $ \alpha ,  \gamma \in \mathbf{H} $ be incomparable with   respect to the subtree relation. 
Assume $ \alpha $  satisfies the condition in Definition \ref{ThirdBasicLemmaFirstDefinition}.
Assume $ \gamma $ is not a subtree of $ \mu_i $ for all $ i \in \lbrace 1, \ldots , n  \rbrace $. 
Let $ W \in \mathbf{H} $. 
Then,  $  W  \in  \mathbb{P}_2   (C  ,  \alpha ,  \gamma )  $  if and only if 
\begin{itemize}

\item[\textup{(1) } ]   there exists  $ X  \in \tau_{ \gamma } (   \lbrace \mu_1, \ldots , \mu_n \rbrace^{ + }   )   $  such that 
 $ \langle    \alpha      \,   ,   \,   X   \rangle    \sqsubseteq   W   $

\item[\textup{(2) } ]   there exists   $ T \in  \mathbb{P}   (C  ,  \alpha ,  \gamma )  $ such that 
$   W =  \langle  V  \,  ,   \,   \Gamma^{ \alpha }_{ n }   (T)  \rangle  $  
where 
\[
V 
   =   
 W      \Big[        \     
  \langle    \alpha      \,   ,   \,   X   \rangle 
  \,  \mapsto  \,    \alpha  
  \   ,      \   
    \frac{ 0 }{ \mu_1 }    \,  \mapsto  \,         \mu_1     
 \   ,      \   
     \frac{ 0 }{ \mu_2 }    \,  \mapsto  \,         \mu_2     
 \   ,      \   
\ldots  
  \   ,      \   
    \frac{ 0 }{ \mu_n  }    \,  \mapsto  \,         \mu_n     
    \      \Big]      
\                   .
\]

\end{itemize}

\end{lemma}

Before we prove the lemma, we illustrate the left-right  implication using the binary tree $T$  in
 Figure   \ref{ThirdBasicLemmaSecondFigure}.
For example,  let 
\[
W = 
 \Big\langle  \;    \alpha    \,  ,   \,  
 W_7  \,  ,   \,   W_7  \,  ,   \,   W_6 \,  ,   \,   W_5 \,  ,   \,   W_4 \,  ,   \,   W_3  \,  ,   \,   W_2  \,  ,   \,     W_1 
 \;       \Big\rangle  
 \in   
 \mathbb{P}_2   (C  ,  \alpha ,  \gamma )  
\]
 where the $ W_i$`s are given in Figure  \ref{FourthBasicLemmaSecondFigure}.
Since $ \alpha $ has only one occurrence in $ W $ 
\[
W   \big[        \        \langle    \alpha      \,   ,   \,   W_7    \rangle     \,  \mapsto  \,    \alpha    \     \big] 
= 
 \Big\langle  \;    \alpha    \,  ,   \,  
  W_7  \,  ,   \,   W_6 \,  ,   \,   W_5 \,  ,   \,   W_4 \,  ,   \,   W_3  \,  ,   \,   W_2  \,  ,   \,     W_1 
 \;       \Big\rangle  
 \         .
\] 
It then follows from the third clause of    Definition   \ref{FourthBasicLemmaSecondDefinition}  
or by looking at  Figure  \ref{FourthBasicLemmaSecondFigure} that 
\begin{multline*}
 W      \Big[        \     
  \langle    \alpha      \,   ,   \,   W_7    \rangle 
  \,  \mapsto  \,    \alpha  
  \   ,      \   
    \frac{ 0 }{ \mu_1 }    \,  \mapsto  \,         \mu_1     
 \   ,      \   
     \frac{ 0 }{ \mu_2 }    \,  \mapsto  \,         \mu_2     
 \   ,      \   
\ldots  
  \   ,      \   
    \frac{ 0 }{ \mu_n  }    \,  \mapsto  \,         \mu_n     
    \      \Big]      
    = 
    \\
    \\
     \Big\langle  \;    \alpha    \,  ,   \,  
    W_7  \,  ,   \,   W_7  \,  ,   \,   W_6 \,  ,   \,   W_5 \,  ,   \,   W_4 \,  ,   \,   W_3  \,  ,   \,   W_2    \;       \Big\rangle  
 \         .
\end{multline*}

\begin{proof}[Proof of Lemma \ref{ForthBasicLemmaMainResult}]

The  left-right implication is a straightforward consequence of Definition  \ref{FourthBasicLemmaSecondDefinition}. 
We  focus on proving the right-left  implication.

Assume   $W$ satisfies (1)-(2). 
We need to show that   $  W  \in  \mathbb{P}_2   (C  ,  \alpha ,  \gamma )  $. 
By Definition  \ref{FourthBasicLemmaSecondDefinition}, 
we need to show that there exist    $ W_1,  \ldots , W_k  \in \mathbf{H} $ such that 
\begin{itemize}

\item[(A)]    $  
W =  \big\langle \;  \alpha   \,  ,   \,    W_k , W_{ k-1}     \,  ,   \,   \ldots    \,  ,   \,  W_2  \,  ,   \,      W_1   \;   \big\rangle   
$

\item[(B)]      $ W_k  \in  \tau_{ \gamma } (   \lbrace \mu_1, \ldots , \mu_n \rbrace^{ + }   )   $

\item[(C)]     for all  $ i \in \lbrace 1,  2, \ldots , k-1 \rbrace $
\[
W_{ i +1   }  = 
W_i    \Big[        \      
    \frac{ 0 }{ \mu_1 }    \,  \mapsto  \,         \mu_1     
 \   ,      \   
     \frac{ 0 }{ \mu_2 }    \,  \mapsto  \,         \mu_2     
 \   ,      \   
\ldots  
  \   ,      \   
    \frac{ 0 }{ \mu_n  }    \,  \mapsto  \,         \mu_n     
    \      \Big]      
\]

\item[(D)]    there exists $ T \in    \mathbb{P}   (C  ,  \alpha ,  \gamma )   $ such that 
$ W_1 = \Gamma^{ \alpha }_{ n }   (T )     \;   $.

\end{itemize}

Let $ X$ and $T$ be binary trees that satisfy clauses (1)-(2). 
First, we    prove by   (backward)  induction   that if 
$  \langle \alpha  \,  ,  \,  X  \rangle   \sqsubseteq     U   \sqsubseteq W  $ and $ U = \langle U_0  \, ,  \, U_1  \rangle  $, 
 then 
\[
U_0  
   =    
 U      \Big[        \     
  \langle    \alpha      \,   ,   \,   X   \rangle 
  \,  \mapsto  \,    \alpha  
  \   ,      \   
    \frac{ 0 }{ \mu_1 }    \,  \mapsto  \,         \mu_1     
 \   ,      \   
\ldots  
  \   ,      \   
    \frac{ 0 }{ \mu_n  }    \,  \mapsto  \,         \mu_n     
    \      \Big]      
\tag{*}
\] 
and 
\[
\alpha  \not\sqsubseteq   U_1  
\            .
\tag{**}
\]
The base case $ U =  W $ is Clause (2). 
So,  assume $ U = \langle V   \, ,   \,  U_1  \rangle $,  $ \  V =  \langle V_0    \, ,   \,  V_1  \rangle   $, 
$  \      \langle \alpha  \,  ,  \,  X  \rangle   \sqsubseteq     V   \sqsubseteq U  \sqsubseteq W $
and $ U $ satisfies (*) and (**). 
We need to show that $V$ satisfies (*) and (**). 
Since $U$ satisfies (**),   $   \langle \alpha  \,  ,  \,  X  \rangle     \not\sqsubseteq    U_1 $. 
Since $ \alpha $ is  incomparable with  $ 0 $ and  $ \mu_i $ with respect to $ \sqsubseteq   $, 
the binary tree $ \frac{0 }{ \mu_i } $ cannot equal a binary tree  that has  $ \alpha $ as  subtree. 
Furthermore, if $ \alpha \sqsubseteq R $, then 
$ \alpha \sqsubseteq R [  \,     \frac{0 }{ \mu_i }   \,  \mapsto  \,  \mu_i  \,  ]   $. 
Hence, by (*)
\[
\begin{array}{r c l }
V 
&=&  
U      \Big[        \     
  \langle    \alpha      \,   ,   \,   X   \rangle 
  \,  \mapsto  \,    \alpha  
  \   ,      \   
    \frac{ 0 }{ \mu_1 }    \,  \mapsto  \,         \mu_1     
 \   ,      \   
\ldots  
  \   ,      \   
    \frac{ 0 }{ \mu_n  }    \,  \mapsto  \,         \mu_n     
    \      \Big] 
    \\
  &=&  
\langle V   \, ,   \,  U_1  \rangle  
      \Big[        \     
  \langle    \alpha      \,   ,   \,   X   \rangle 
  \,  \mapsto  \,    \alpha  
  \   ,      \   
    \frac{ 0 }{ \mu_1 }    \,  \mapsto  \,         \mu_1     
 \   ,      \   
\ldots  
  \   ,      \   
    \frac{ 0 }{ \mu_n  }    \,  \mapsto  \,         \mu_n     
    \      \Big]           
\\
&= & 
  \Big\langle  \;  
U^{ \prime }  
    \;  ,  \;  
U^{ \prime   \prime   } 
    \;    \Big\rangle 
\end{array}
\]
where  by (**) 
\[
\begin{array}{r c l }
U^{ \prime   \prime   }
  &= & 
 U_1        \Big[        \     
    \frac{ 0 }{ \mu_1 }    \,  \mapsto  \,         \mu_1     
 \   ,      \   
\ldots  
  \   ,      \   
    \frac{ 0 }{ \mu_n  }    \,  \mapsto  \,         \mu_n     
    \      \Big]      
    \not\sqsupseteq  \alpha 
    \\
    \\
U^{ \prime }  
&=&  
V      \Big[        \     
  \langle    \alpha      \,   ,   \,   X   \rangle 
  \,  \mapsto  \,    \alpha  
  \   ,      \   
    \frac{ 0 }{ \mu_1 }    \,  \mapsto  \,         \mu_1     
 \   ,      \   
   \ldots 
  \   ,      \   
    \frac{ 0 }{ \mu_n  }    \,  \mapsto  \,         \mu_n     
    \      \Big]     
    \        . 
\end{array}
\]
Thus, $V$ satisfies (*) and (**). 
Thus, by induction, 
if 
$  \langle \alpha  \,  ,  \,  X  \rangle   \sqsubseteq     U   \sqsubseteq W  $ and $ U = \langle U_0  \, ,  \, U_1  \rangle  $, 
 then $U$ satisfies (*)  and (**).
 \newline

Now, to prove  that (A)-(D) hold, 
it  suffices to prove  by induction on the size of finite binary trees  that if $ U  $ is a subtree of $W $ 
which is such that $   \langle    \alpha      \,   ,   \,   X   \rangle    \sqsubseteq  U $, 
then there exists a sequence $ U_1,  \ldots , U_m $ such that 
\begin{itemize}

\item[(i) ]    $  U =  \big\langle \;  \alpha , U_m , U_{ m-1}  ,  \ldots ,  U_1    \;   \big\rangle   $

\item[(ii) ]      $ U_m   =   X    $

\item[(iii) ]     for all  $ i \in \lbrace 1,  2, \ldots , m-1 \rbrace $
\[
U_{ i  + 1  }  = 
U_{  i   }     \Big[        \      
    \frac{ 0 }{ \mu_1 }    \,  \mapsto  \,         \mu_1     
 \   ,      \   
\ldots  
  \   ,      \   
    \frac{ 0 }{ \mu_n  }    \,  \mapsto  \,         \mu_n     
    \      \Big]      
    \          .
\]

\end{itemize}

So, assume $   \langle    \alpha      \,   ,   \,   X   \rangle    \sqsubseteq   U    \sqsubseteq   W         $.
If $ U = \langle    \alpha      \,   ,   \,   X   \rangle   $, then $U$ satisfies (i)-(iii) trivially.
Otherwise, by (**),   there exist $ V $ and $ U_1 $ such that 
$ U = \langle V \, , \,  U_1  \rangle $ and $    \langle    \alpha      \,   ,   \,   X   \rangle    \sqsubseteq  V $. 
By the induction hypothesis, 
there exists a sequence   $ V_1,  \ldots , V_m $ such that 
\begin{itemize}

\item[(iv) ]    $  V =  \big\langle \;  \alpha , V_m , V_{ m-1}  ,  \ldots ,  V_1    \;   \big\rangle   $

\item[(v) ]      $ V_m   =   X    $

\item[(vi) ]     for all  $ i \in \lbrace 1,  2, \ldots , m-1 \rbrace $
\[
V_{ i  + 1  }  = 
V_{  i   }     \Big[        \      
    \frac{ 0 }{ \mu_1 }    \,  \mapsto  \,         \mu_1     
 \   ,      \   
   \ldots 
  \   ,      \   
    \frac{ 0 }{ \mu_n  }    \,  \mapsto  \,         \mu_n     
    \      \Big]      
    \          .
\]

\end{itemize}
In particular 
\[
U = \langle V \,  ,   \,   U_1  \rangle 
= 
 \big\langle \;  \alpha  \,  ,  \,   V_m   \,  ,  \,    V_{ m-1}   \,  ,  \,    \ldots ,  V_1    \,  ,  \,   U_1   \;   \big\rangle
 \        .
\]
By (v)-(vi) and   (**),    there can only be one occurrence of $ \alpha $ in $U$. 
Hence 
\[
U \big[        \      \langle    \alpha      \,   ,   \,   X   \rangle     \,  \mapsto  \,    \alpha     \    \big]
= 
 \big\langle \;  \alpha     \,  ,  \,    V_{ m-1}   \,  ,  \,    \ldots ,  V_1    \,  ,  \,   U_1   \;   \big\rangle
\           .
\]
Then,  by (*)  and (vi)
\begin{multline*}
 \big\langle \;  \alpha  \,  ,  \,   V_m   \,  ,  \,    V_{ m-1}   \,  ,  \,    \ldots ,  V_1      \;   \big\rangle
 = 
 V 
 = 
 \\
U \Big[        \     
  \langle    \alpha      \,   ,   \,   X   \rangle 
  \,  \mapsto  \,    \alpha  
  \   ,      \   
    \frac{ 0 }{ \mu_1 }    \,  \mapsto  \,         \mu_1     
 \   ,      \   
\ldots  
  \   ,      \   
    \frac{ 0 }{ \mu_n  }    \,  \mapsto  \,         \mu_n     
    \      \Big]      
    = 
     \\
 \big\langle \;  \alpha     \,  ,  \,    V_{ m-1}   \,  ,  \,    \ldots ,  V_1    \,  ,  \,   U_1   \;   \big\rangle
 \Big[        \     
    \frac{ 0 }{ \mu_1 }    \,  \mapsto  \,         \mu_1     
 \   ,      \   
\ldots  
  \   ,      \   
    \frac{ 0 }{ \mu_n  }    \,  \mapsto  \,         \mu_n     
    \      \Big]      
    =
    \\
\big\langle \;  \alpha   \, ,  \,   V_{ m}    \, ,  \,     \ldots   \, ,  \,     V_2    \, ,  \,      U_1^{ \prime }     \;   \big\rangle 
\end{multline*}
where 
\[
U_1^{ \prime } 
= 
 U_1  \Big[        \     
    \frac{ 0 }{ \mu_1 }    \,  \mapsto  \,         \mu_1     
 \   ,      \   
\ldots  
  \   ,      \   
    \frac{ 0 }{ \mu_n  }    \,  \mapsto  \,         \mu_n     
    \      \Big]      
    \                     .
\]
Hence 
\[
U 
= 
 \big\langle \;  \alpha  \,  ,  \,   V_m   \,  ,  \,    V_{ m-1}   \,  ,  \,    \ldots ,  V_1    \,  ,  \,   U_1   \;   \big\rangle
\]
and 
\[
V_1  = U_1   \Big[        \     
    \frac{ 0 }{ \mu_1 }    \,  \mapsto  \,         \mu_1     
 \   ,      \   
\ldots  
  \   ,      \   
    \frac{ 0 }{ \mu_n  }    \,  \mapsto  \,         \mu_n     
    \      \Big]      
    \                     .
\]
Thus,  $U$ satisfies (i)-(iii).

Thus, by induction, 
 if $ U  $ is a subtree of $W $  which is such that $   \langle    \alpha      \,   ,   \,   X   \rangle    \sqsubseteq  U $, 
 then $U$ satisfies (i)-(iii).
\end{proof}

\subsection{Reduction of PCP}
\label{ManyToOneReductionOfPCPExistentialFragment}

We are ready to specify  a many-to-one reduction of PCP.

\begin{theorem}

Post`s Correspondence Problem is many-to-one reducible to   the fragment 
$ \mathsf{Th}^{ \exists }   (  \mathcal{T} ( \mathcal{L}_{ \mathsf{BT} }  )  ) $. 

\end{theorem}

\begin{proof}

Let 
\[
\alpha = \langle \perp  \,  ,  \,  \perp^{ 2 }   \rangle 
\      \   
\mbox{  and   }  
\      \   
\gamma   = \langle \perp  \,  ,  \,  \perp^{ 3 }   \rangle 
\        .
\]
Then, $ \alpha $ and $ \gamma $ satisfy the conditions in 
Definition   \ref{ThirdBasicLemmaSecondDefinition} 
and 
Definition   \ref{FourthBasicLemmaSecondDefinition}.

Consider an instance $\langle a_{1}, b_{1} \rangle,\ldots , \langle a_{n}, b_{n}  \rangle$  of PCP. 
We need to construct an existential  $ \mathcal{L}_{ \mathsf{BT} }   $-sentence $ \phi $ 
that is true in  $  \mathcal{T} ( \mathcal{L}_{ \mathsf{BT} }  )  $   if and only if   
$\langle a_{1}, b_{1} \rangle,\ldots , \langle a_{n}, b_{n}  \rangle$
has a solution.
Recall that      $\langle a_{1}, b_{1} \rangle,\ldots , \langle a_{n}, b_{n}  \rangle$ 
has a solution if and only if there exist   two  sequences 
\[
u_1,  u_2,   \ldots , u_k     
\      \    
\mbox{  and   }     
 \         \  
 v_1, v_2,   \ldots , v_m 
\]
such that 
\begin{itemize}

\item[(I)]    there exists  $ f_1  \in \lbrace 1,  \ldots , n \rbrace $ such that    $u_1 = a_{ f_1 }  $   and 
for all $ j \in \lbrace 1,  \ldots , k-1  \rbrace $ there exist $ f_{ j+1 }    \in \lbrace 1,  \ldots , n \rbrace $ such that 
$ u_{ j+1 } = u_j a_{   f_{ j+1 }   } $

\item[(II)]    there exists  $ g_1  \in \lbrace 1,  \ldots , n \rbrace $ such that    $u_1 = b_{ g_1 }  $   and 
for all $ j \in \lbrace 1,  \ldots , m-1  \rbrace $ there exist $ g_{ j+1 }    \in \lbrace 1,  \ldots , n \rbrace $ such that 
$ u_{ j+1 } = u_j    b_{   g_{ j+1 }   } $

\item[(III)]  $ k =m $ and $ f_j = g_j $ for all   $j \in \lbrace 1, \ldots , k \rbrace $

\item[(IV)]  $ u_k = v_m $.

\end{itemize}

Let $ A = \langle   a_1, a_2,  \ldots , a_n    \rangle  $ and let  $ B = \langle   b_1, b_2,  \ldots , b_n   \rangle  $. 
Definition   \ref{ThirdBasicLemmaSecondDefinition}  tells us that 
the sequence $ u_1 ,  u_2 ,  \ldots  , u_k $ is encoded by a binary tree  $ L   \in   \mathbb{P} (  A,    \alpha ,  \gamma )  $
and the right subtree of $ L $, denoted $U$, encodes $ u_k $. 
Similarly, 
the   sequence $ v_1 ,  v_2 ,  \ldots  , v_m $ is encoded by   a binary tree   $ R    \in \mathbb{P} (  B,    \alpha ,  \gamma )  $
and the right subtree of $ R $, denoted $V$, encodes $ v_m $. 
Lemma  \ref{ThirdBasicLemmaMainResult}  tells us that 
$ \mathbb{P} (  A,    \alpha ,  \gamma )  $ and $   \mathbb{P} (  B,    \alpha ,  \gamma )  $ 
are existentially definable.

Definition    \ref{FourthBasicLemmaSecondDefinition}
gives us binary trees $ X_L  $ and  $ W_L  \in  \mathbb{P}_2  (  A,    \alpha ,  \gamma )  $ 
such that     $ \Gamma^{ \alpha }_{ n }     ( L) $ is the right subtree of $ W_L  $, 
   $ \;    \langle \alpha \,  ,  \,  X_L   \rangle  \sqsubseteq W_L $ and 
$ X_L  $ encodes the sequence $ f_1, f_2,  \ldots , f_k $. 
The existentially  definable operator $ \Gamma^{ \alpha }_{ n }    $ is defined in 
 Definition    \ref{FourthBasicLemmaFirstDefinition}.
Similarly, there exist  $ X_R  $ and  $ W_R  \in  \mathbb{P}_2  (  B,    \alpha ,  \gamma )  $ 
such that    $ \Gamma^{ \alpha }_{ n }   ( R) $ is the right subtree of $ W_R  $, 
    $ \;    \langle \alpha \,  ,  \,  X_R   \rangle  \sqsubseteq W_R $ and 
$ X_R  $ encodes the sequence $ g_1, g_2,  \ldots , g_m $. 
Lemma   \ref{ForthBasicLemmaMainResult}    tells us that 
$ \mathbb{P}_2  (  A,    \alpha ,  \gamma )  $ and $   \mathbb{P}_2  (  B,    \alpha ,  \gamma )  $ 
are existentially definable.

Now, encoding (III) corresponds to requiring that $ X_L = X_R   $ holds. 
To encode  (IV), we cannot simply require that  $ U = V $ holds since 
$ U $ is the representation of $ u_k $ when viewed as element of $ \lbrace 0, 1, \mu_1, \ldots , \mu_n \rbrace^+  $ 
and  $ V $ is the representation of $ v_m $ when viewed as element of $ \lbrace 0, 1, \mu_1, \ldots , \mu_n \rbrace^+ $. 
So, let $ \Theta^A_n  (U)  $ be the binary tree we obtain by replacing $ \mu_i $ with the last letter of $ a_i $
and let $ \Theta^B_n   (V) $   be the binary tree we obtain by replacing $ \mu_j $ with the last letter of $ b_j $.
Then,  encoding  (IV)  corresponds to requiring  that $ \Theta^A_n    (U) =  \Theta^B_n   (V) $ holds.

It is now obvious how to  specify an existential   $ \mathcal{L}_{ \mathsf{BT} }   $-sentence $ \phi $ 
that is true in  $  \mathcal{T} ( \mathcal{L}_{ \mathsf{BT} }  )  $   if and only if   
$\langle a_{1}, b_{1} \rangle,\ldots , \langle a_{n}, b_{n}  \rangle$
has a solution.
We let 
\begin{multline*} 
\phi \equiv   \  
\exists L \in \mathbb{P} (  A,    \alpha ,  \gamma )     \;     \exists  U, U^{ \prime }  \;  
\exists R  \in \mathbb{P} (  B,  \alpha ,   \gamma )      \;     \exists  V, V^{ \prime }  \;  
\\
\exists W_L \in \mathbb{P}_2  (  A,    \alpha ,  \gamma )     \;     
\exists  X_L   ,  S_L    \;  
\exists W_R  \in \mathbb{P}_2  (  B,    \alpha ,  \gamma )     \;     
\exists  X_R   ,   S_R   \;  
\Big[    \   
\\ 
L  =   \langle U^{ \prime }  \,  ,   \,     U   \rangle   
\    \wedge    \   
R  =   \langle V^{ \prime }  \,  ,   \,     V   \rangle   
\    \wedge    \   
\\  
 \langle  \alpha  \,  ,  \,  X_L  \rangle  \sqsubseteq   W_ L
\    \wedge    \   
W_L =  \langle  S_L   \,  ,   \,     \Gamma^{ \alpha }_{ n }  (L )     \rangle 
\    \wedge    \ 
\\
 \langle  \alpha  \,  ,  \,  X_R  \rangle  \sqsubseteq   W_ R
\    \wedge    \   
W_R =  \langle  S_R   \,  ,   \,     \Gamma^{ \alpha }_{ n }  (R )     \rangle 
\    \wedge    \ 
\\
\Theta^A_n  (U)  =   \Theta^B_n  (V) 
\    \wedge    \   
X_L =  X_R 
\          \Big]
\end{multline*}  
where  
\begin{itemize}

\item[-] $   \Theta^A_n   (U)  =  U  \big[  \;  \mu_1 \,  \mapsto  \,  d_1    \;  ,     \;   \ldots    \;   ,  \;  
 \mu_n    \,  \mapsto  \,   d_n   \big]   $
 and $ d_i $ is the last letter of $ a_i $

\item[-] $\Theta^B_n   (V)  =  V  \big[  \;  \mu_1 \,  \mapsto  \,  e_1    \;  ,     \;   \ldots    \;   ,  \;  
 \mu_n    \,  \mapsto  \,  e_n    \big]   $
 and $ e_j $ is the last letter of $ b_j $.     \qedhere 

\end{itemize}

\end{proof}

\section{Existential Definability    in   $  \mathcal{T}  ( \mathcal{L}_{ \mathsf{BT}   }  )  $  }

Our proofs  of undecidability  of the existential theory   of $ \mathcal{T} ( \mathcal{L}_{ \mathsf{BT} }  )  $ 
 do  not provide a characterization of the  existentially definable subsets of  $ \mathbf{H} $. 
Given a computably enumerable set $ A \subseteq   \mathbf{H}  $, 
we do not know if there exists an existential formula $ \theta_{ A } ( x  )  $ with only $x$ free such that 
for all $ t \in \mathbf{H}  $
\[
 t \in A   
   \      \Leftrightarrow    \   
  \mathcal{T} ( \mathcal{L}_{ \mathsf{BT} }  )    \models   \theta_A  ( t ) 
  \          .
  \]
For example,      Theorem \ref{ExistentialInterpretationOfFirstOrderArithmetic}  
  tells us  that $A$ is  one-to-one reducible to  an existentially definable  subset of   $   \mathbf{H}  $.
Indeed, let 
 $ g:    \mathbb{N}  \to   \mathbf{H} $ be the map that sends each  natural number to the corresponding element of 
$   \mathbb{N}^{ 0 }_{ s  }    \cup \lbrace 0 \rbrace $
and  choose  a one-to-one computable  function   $ f : \mathbf{H}  \to  \mathbb{N} $. 
Since  every computably enumerable set of natural numbers is existentially definable in  
$ ( \mathbb{N} , 0 , 1 , + , \times ) $ 
(see for example  Davis \cite{Davis1973}), 
there exists  an existential formula  $ \phi_{ A } (x) $ that defines $ f(A) $ in  $ ( \mathbb{N} , 0 , 1 , + , \times ) $. 
The proof  of   Theorem \ref{ExistentialInterpretationOfFirstOrderArithmetic} 
tells us how to compute  an existential   $ \mathcal{L}_{ \mathsf{BT} }  $-formula $ \psi_A (x) $ 
that defines  $ gf (A) $ in   $ \mathcal{T} ( \mathcal{L}_{ \mathsf{BT} }  )  $. 
Now, to show that $ A$  is existentially definable in  $ \mathcal{T} ( \mathcal{L}_{ \mathsf{BT} }  )  $, 
it suffices to find an $f$ such that $ gf $ is existentially definable in  $ \mathcal{T} ( \mathcal{L}_{ \mathsf{BT} }  )  $. 
It is not clear to us  whether such an $f$ exists. 
It  appears as if  the  coding techniques we have developed  are not sufficient to show that $ gf$ is existentially definable. 
Say we try to encode  $ gf(x) = y $ by describing the computation sequence  $ w_1,  w_2,  \ldots , w_k $ of $  gf(x) $. 
The problem is that describing  $ w_1,  w_2,  \ldots , w_k $ requires   that we refer to arbitrary subtrees of $ x $   since 
$ \langle \cdot , \cdot \rangle $ is not associative.

\begin{open problem} \label{OpenProblemDefinabilityTermAlgebraSubstitutionOperator}

Let $ A \subseteq \mathbf{H} $ be a computably enumerable set.
Is  $ A $  existentially definable in   $ \mathcal{T} ( \mathcal{L}_{ \mathsf{BT} }  )  $?

\end{open problem}

To put the preceding problem in context, 
we observe that there exist natural examples of  computable structures  with undecidable existential theory but where the corresponding problem has a negative solution, 
in contrast  to  $ ( \mathbb{N} , 0, 1, +, \times ) $. 
For example,  let us introduce the following relations on $ \lbrace 0, 1, 2  \rbrace^*  $: 
$ (u, v) \in  \mathcal{L} $ if and only if $ \vert u \vert \leq \vert v \vert $ and 
$ (u, v) \in  \mathcal{P} $ if and only if the number of $ 0$`s in $ u$ is the same as the number of $0$`s in $ v$. 
Then, it follows easily from the proof of Theorem 2 of  B\"uchi  and Senger     \cite{Senger1988}
 that  $ ( \mathbb{N} , 0, 1, +, \times ) $  is $ \exists $-interpretable in   
 $ (   \lbrace 0, 1, 2  \rbrace^*   , \varepsilon ,  0, 1 , 2  , ^{ \frown}  ,   \mathcal{L}  ,   \mathcal{P}   )   $.
We show that   $ \lbrace 0, 1 \rbrace^* $ is not existentially definable in 
 $ (   \lbrace 0, 1, 2  \rbrace^*   , \varepsilon ,  0, 1 , 2  , ^{ \frown}  ,   \mathcal{L}  ,   \mathcal{P}   )   $.
To prove this, 
we need to make  a  minor modification  to  the proof of   Theorem 16 of  Karhum\"aki et al.   \cite{Karhumaki2000}.
The necessary changes will be obvious to a  reader familiar with  \cite{Karhumaki2000}.
For completeness, we include the necessary details. 
 Theorem 16 of  \cite{Karhumaki2000}  is a  pumping lemma-like result   for    finitely generated free semigroups.

It is well known that   any  existential  formula in the language $ \lbrace \varepsilon , 0, 1, 2  , ^{ \frown}  \rbrace $
is computably   equivalent  in   
$ (   \lbrace 0, 1, 2  \rbrace^*   , \varepsilon ,  0, 1 , 2  , ^{ \frown}   ) $
  to a formula  of the form   $ \exists \vec{x}  \;  [   \   s= t  \   ]  $ 
(see Theorem 6 of  \cite{Karhumaki2000}   or 
Section 3 of   Kristiansen \& Murwanashyaka \cite{KristiansenMurwanashyakaAML}).
Since 
\[
 (u, v) \not\in    \mathcal{L}     \   \Leftrightarrow   \        (v0 , u)  \in    \mathcal{L}  
\      \       
\mbox{   and    }   
\            \   
(u, v )   \not\in    \mathcal{P}       \   \Leftrightarrow   \   
 \exists x \;   [          \       ( u0x , v ) \in    \mathcal{P}      \     \vee    \       ( u , v 0 x  ) \in    \mathcal{P}         \         ] 
\]
it  follows that  each existential  formula in the language 
$ \lbrace \varepsilon ,0, 1, 2  , ^{ \frown} ,  \mathcal{L}  ,   \mathcal{P}   \rbrace $ 
is computably   equivalent in  
$ (   \lbrace 0, 1, 2  \rbrace^*   , \varepsilon ,  0, 1 , 2  , ^{ \frown}  ,   \mathcal{L}  ,   \mathcal{P}   )   $
 to a formula  of the form 
\[
\exists \vec{x}  \;     \exists \vec{y}   \;     \exists \vec{z}   \;     \exists \vec{w}   \;     \exists \vec{v}       \; [       \  
 s=t    
  \    \wedge       \    
 \bigwedge_{i=1}^{n}  \mathcal{L} ( y_j  ,   z_j )   
  \    \wedge       \     
   \bigwedge_{j=1}^{m}  \mathcal{P} ( w_j  ,   v_j )   
  \      ]
\    . 
\] 
Let us call  this representation the  \emph{normal form}.

We need a notion of factorization. 
The factorization of $ \varepsilon $ is $ \varepsilon $. 
The   factorization   of  $ w \in  \lbrace 0, 1, 2  \rbrace^+  $  is  the  sequence 
$ w_1,  w_2 ,  \ldots, w_k   \in  \lbrace 0, 1, 2  \rbrace^+  $  that  satisfies 
\begin{itemize}

\item[-] $ w = w_1 w_2   \ldots w_ k $

\item[-]  for each $ i  \in \lbrace 1,  2,  \ldots , k \rbrace $, there exists $ a_i \in   \lbrace 0, 1, 2  \rbrace $ such that 
$ w_i \in  \lbrace a  \rbrace^+  $.

\item[-] for each $  i  \in \lbrace 1,  2,  \ldots , k-1 \rbrace $,  $   \     a_i \neq a_{ i+1 }  $.

\end{itemize}

We let  $ \mathsf{F}  (w) $ denote  the number of distinct factors in the factorization of $ w $. 
For example, 
the factorization of $ 0^{ 5 } 10 2^{ 6 }  0  2^{ 6 } $ is 
\[
 0^{ 5 } ,   \;   1 ,   \;  0 ,   \;    2^{ 6 }    ,     \;   0      ,     \;     2^{ 6 } 
 \      \      
 \mbox{    and      }  
 \        \    
  \mathsf{F}  (  0^{ 5 } 10 2^{ 6 }  0  2^{ 6 } ) = 4 
  \         .
  \]

The result we need is the following.

\begin{claim} \label{DefinabilityLengthProjection}

Let  $\phi (x) $ be an existential formula on normal form  that defines  the set  
$L \subseteq \lbrace 0, 1 , 2   \rbrace^*  $ 
in 
$ (   \lbrace 0, 1, 2  \rbrace^*   , \varepsilon ,  0, 1 , 2  , ^{ \frown}  ,   \mathcal{L}  ,   \mathcal{P}   )   $.
Let $ w \in L $ be such that $ \mathsf{F} ( w ) >  2 \vert \phi \vert + 2  $.
Then, there exists a term $ p (x) $  and a   factor  $ u $ of $w $   such that 
$ w = p(u) $ and 
\[
\lbrace   p(v)      :     \    
v \in    \lbrace 0, 1, 2  \rbrace^*            \     \wedge     \   
\vert u \vert = \vert v  \vert     \     \wedge     \       \mathcal{P} (u  ,  v)        \        \rbrace
\subseteq L  
\                .
\]

\end{claim}

Before we prove the claim, let us use it to show that  $   \lbrace 0, 1  \rbrace^*  $ is not existentially definable in   
$ (   \lbrace 0, 1, 2  \rbrace^*   , \varepsilon ,  0, 1 , 2  , ^{ \frown}  ,   \mathcal{L}  ,   \mathcal{P}   )   $.

\begin{theorem}

$   \lbrace 0, 1  \rbrace^*  $ is not existentially definable in   
$ (   \lbrace 0, 1, 2  \rbrace^*   , \varepsilon ,  0, 1 , 2  , ^{ \frown}  ,   \mathcal{L}  ,   \mathcal{P}   )   $.

\end{theorem}

\begin{proof}

Assume for the sake of a contradiction   $   \lbrace 0, 1  \rbrace^*  $  is definable  in  the structure 
$ (   \lbrace 0, 1, 2  \rbrace^*   , \varepsilon ,  0, 1 , 2  , ^{ \frown}  ,   \mathcal{L}  ,   \mathcal{P}   )   $
 by an existential formula  $ \phi $,  which we may assume is on  the normal form. 
Let $ k = 2  \vert \phi  \vert + 2  $. 
We consider the word 
\[
w = 0    1^{ k+1  }       0    1^{ k }     0     1^{k-1}      \ldots  0    1^{  2  }    0 1 
\]
which has the factorization 
\[
 0  , \     1^{ k+1  }    , \        0    , \      1^{ k }    , \      0       , \     1^{k-1}     , \        \ldots     , \    0   , \        1^{  2  }     , \    0     , \     1 
\        .
\]
Claim  \ref{DefinabilityLengthProjection} gives us a  term $ p(x) $ such that $ w = p ( a^{ i } ) $ for some
 $ i  \in  \lbrace 1,  2,  \ldots , k+1  \rbrace  $ and $ a \in \lbrace 0, 1 \rbrace $.
Claim  \ref{DefinabilityLengthProjection}
tells us also that  $ p ( 2^{ i } ) \in   \lbrace 0, 1  \rbrace^*  $.
 But  this contradicts  the fact that no word in  $ \lbrace 0, 1  \rbrace^*   $ has an occurrence of $2$. 
Thus,  $   \lbrace 0, 1  \rbrace^*  $ is not existentially definable in 
$ (   \lbrace 0, 1, 2  \rbrace^*   , \varepsilon ,  0, 1 , 2  , ^{ \frown}  ,   \mathcal{L}  ,   \mathcal{P}   )   $.
\end{proof}

It is also possible to show that  $ \lbrace 010 , 0110 \rbrace^* $ is not existentially definable in  the structure 
$ (  \lbrace 0, 1  \rbrace^*   , \varepsilon ,  0, 1    , ^{ \frown}  ,    \mathcal{L}  ,   \mathcal{P}  ) $, 
which also has undecidable existential theory. 
To prove this, it suffices to work with   the notion of $ \mathcal{F}_{ Q } $-factorization given  in Section 5 of    \cite{Karhumaki2000}  
instead of the naive  block factorization we worked with. 
We leave the  verification of this  to the interested reader.

\begin{proof}[Proof of Claim \ref{DefinabilityLengthProjection} ]

We know that  $\phi $ is of the form 
\[
\exists \vec{x}  \;     \exists \vec{y}   \;     \exists \vec{z}   \;     \exists \vec{w}   \;     \exists \vec{v}       \; [       \  
 s=t    
  \    \wedge       \    
 \bigwedge_{i=1}^{n}  \mathcal{L} ( y_j  ,   z_j )   
  \    \wedge       \     
   \bigwedge_{j=1}^{m}  \mathcal{P} ( w_j  ,   v_j )   
  \      ]
  \        . 
\] 
So, $L $  is defined by   an equation $s = t$ with constraints of the form 
$(y, z) \in \mathcal{L} $ and  $(w, v) \in \mathcal{P} $.
Let $h  $  be a solution to  $s=t$ that also satisfies the given constraints. 
We start as in the proof of Theorem 16 of   \cite{Karhumaki2000}, with minor changes. 
Let $f_1, \ldots , f_N $ be the  factorization of the word  $h(s)   $. 
For each variable $Y$,  the word $h(Y)$ has a factorization $y_1, \ldots , y_t$.
For $ i \in \lbrace 2, \ldots , t-1 \rbrace $, we call $ y_i $ an inside factor. 
We call $ y_1 $ and $y_t $ outside factors.
Assume $Y$ occurs in $ s$ or $t$. 
Then, $h(Y) $ is a substring of $ h(s) $.
Since we factor words into blocks of the same letters, 
there exist $ \ell_Y  \in \lbrace 1, \ldots , N \rbrace $ such that
\begin{itemize}

\item[(I)]   $ y_{1+  i } =  f_{ \ell_Y + i }  $ for all $ i \in \lbrace 1 , \ldots , t-2 \rbrace $

\item[(II)] $ y_1  $ is a suffix of $ f_{ \ell_Y } $   and    $ y_t  $ is a prefix of $ f_{ \ell_Y   + t -1 }  $.

\end{itemize}

Two  functions $\Left $ and $\Right $ are defined on $\lbrace  1, \ldots , N \rbrace$. 
We let  $\Left i = (j, Y)$ if   $Y$ occurs in $s$  and  $f_i$ is the inside factor $y_j$ of $h(Y) $. 
We let  $\Right i = (j, Y)$ if $f_i$ is the inside factor $y_j$ of $h(Y)$ and $Y$ occurs in $t$. 
We call $f_i$ a proper factor if both $ \Left i$ and $\Right i $ are defined. 
The remaining factors are called unproper. 
We make  $\Left $ and $\Right $ total function by  mapping  any  undefined point $i$  to $(1, f_i )  $.
Let $d $ denote the length of $ s= t $. 
Observe that a factor $ f_i$ is unproper because of one of the following reasons: 
(i) there exist a letter $ a \in \lbrace 0, 1, 2 \rbrace $ that occurs in $s$ or $t$ and the corresponding position in $ h(s) $  is in $f_i $
when $f_i$ is viewed as part of $h(s) $, 
(ii) $f_i$ overlaps as in (II) with an end factor of $ h(Y) $ for some variable $Y$. 
If $ z $ is a letter or a  variable that occurs in $ s = t $, 
then each occurrence of $z $ contributes at most    two unproper factors.  
Thus, there are at most $2d $ unproper factors.

The set $ \lbrace 1, \ldots , N \rbrace$ is partitioned by letting $i $ and $j$ be in the same equivalence class if and only if 
\[
\Left i = \Right j  
\    \  \mbox{  or  }  \  \  
\Left i = \Left j
\    \  \mbox{  or  }  \  \  
\Right  i = \Right j
\    .
\]

Let $X$ be the variable defining $L$. 
We have two  cases: 
\begin{enumerate}

\item[(1)] Each equivalence class containing a position of an inside factor of $h(X)$ also contains a position of an unproper factor.

\item[(2)] Some equivalence classes containing positions of  inside factors of $h(X)$ do not contain positions of  unproper factors.

\end{enumerate}

Assume (1) holds. 
Then, the number of   distinct  inside factors of $h(X)$ is at most $2 d $ since there are at most $2d $  unproper factors. 
Since $ h(X) $ has two end factors, the number of distinct factors of $ h(X) $ has the following upper bound 
\[
\mathsf{F} ( h (X)  )     \leq  2  d + 2   \leq    2  \vert \phi  \vert + 2  
\            .
\]
Thus, if $  \mathsf{F} ( h (X)  )  >    2  \vert \phi  \vert + 2   $, then (2) holds.

We now consider (2).
Let $ h (X) = w_1  w_2   \ldots w_k $ where $ w_1,  w_2,  \ldots , w_k $ is the factorization of $ h(X) $. 
Choose an inside factor $ w_i $ of    $h(X)$  that belongs to  an equivalence class containing only positions of proper factors. 
Let $p(x) $ be the term we obtain by replacing with the fresh variable $x $ each occurrence of $ w_i $ as a factor  of  $h(X) $. 
Let $ h(s) (x) = h(t) (x) $ be the equation we obtain by replacing with $x $  each occurrence of $ w_i $ as a factor of $ h(s) $. 
Similarly,  for each variable $Y$, let $ h(Y) (x)  $ be term we  obtain by replacing each occurrence of $ w_i $ as a factor  of  $ h(Y) $. 
Now, for any word $ v  \in  \lbrace 0, 1, 2 \rbrace^* $, 
we have $ h(s) (v ) = h(t) (v) $. 
So, each $ v  \in  \lbrace 0, 1, 2 \rbrace^* $ defines a new solution to the equation  $ s = t $. 
However, the solution may not satisfy  the constraints given by  $ \mathcal{L}   $ and $    \mathcal{P}  $. 
If we choose $ v $  with the same length as $ w_i $, then 
$ ( h(Y) (v) , h(Z)  (v)  ) \in \mathcal{L}    $ if $ ( h(Y), h(Z) ) \in \mathcal{L}   $.
Similarly, if we choose   $ v $ where the number of $0$`s in $v$ is the same as the number of $ 0$`s $w_i $,   then 
$ ( h(Y) (v) , h(Z)  (v)  ) \in \mathcal{P}    $ if $ ( h(Y), h(Z) ) \in \mathcal{P}   $.
\end{proof}

\section*{References}

\bibliographystyle{elsarticle-num}

\end{document}